\documentclass[11pt, a4paper, oneside, fleqn]{article}
\usepackage{verbatim}
\usepackage[utf8]{inputenc}
\usepackage[british]{babel}
\usepackage[english]{varioref}
\usepackage[square, sort, numbers]{natbib}
\usepackage{comment}
\usepackage{amsmath,amsthm,amssymb}
\usepackage{mathrsfs}
\usepackage{multicol}
\usepackage{multirow}
\usepackage{enumerate}
\usepackage{longtable}
\usepackage{float}
\usepackage[plain]{algorithm}
\usepackage{algpseudocode}
\usepackage{tikz}

\includecomment{comment:upperBound}

\includecomment{comment:kBPP:solutionAlgorithm}

\includecomment{comment:kBPP:optimalityProof}

\includecomment{comment:kBPP:objectiveValue}

\includecomment{comment:kBPP:approximationRatio}

\includecomment{comment:appendix}

\makeatletter
\newenvironment{breakablealgorithm}
  {
   \begin{center}
     \refstepcounter{algorithm}
     \renewcommand{\caption}[2][\relax]{
       {{\ALG@name~\thealgorithm:} ##2\par}%
       \ifx\relax##1\relax 
         \addcontentsline{loa}{algorithm}{\protect\numberline{\thealgorithm}##2}%
       \else 
         \addcontentsline{loa}{algorithm}{\protect\numberline{\thealgorithm}##1}%
       \fi
     }
  }{
   \end{center}
  }
\makeatother

\algrenewcommand\Return{\State \algorithmicreturn{} }%

\newenvironment{tikzpictureGuestGraph}
	{
		\begin{tikzpicture}[
			xscale=1.40, yscale=0.72,
			node/.style={circle, draw=black!100, fill=white!100, thick, inner sep=0pt, minimum size=7.3mm},
			nodeBlack/.style={circle, draw=black!100, fill=black!50, thick, inner sep=0pt, minimum size=7.3mm},
			nodeRed/.style={circle, draw=black!100, fill=red!50, thick, inner sep=0pt, minimum size=7.3mm},
			nodeBlue/.style={circle, draw=black!100, fill=blue!50, thick, inner sep=0pt, minimum size=7.3mm},
			nodeGreen/.style={circle, draw=black!100, fill=green!50, thick, inner sep=0pt, minimum size=7.3mm},
			nodeCyan/.style={circle, draw=black!100, fill=cyan!50, thick, inner sep=0pt, minimum size=7.3mm},
			nodeMagenta/.style={circle, draw=black!100, fill=magenta!50, thick, inner sep=0pt, minimum size=7.3mm},
			nodeYellow/.style={circle, draw=black!100, fill=yellow!50, thick, inner sep=0pt, minimum size=7.3mm},
			nodeGrey/.style={circle, draw=black!100, fill=black!25, thick, inner sep=0pt, minimum size=7.3mm},
			nodeBlackDashed/.style={circle, dashed, draw=black!100, fill=black!50, thick, inner sep=0pt, minimum size=7.3mm},
			nodeRedDashed/.style={circle, dashed, draw=black!100, fill=red!50, thick, inner sep=0pt, minimum size=7.3mm},
			nodeBlueDashed/.style={circle, dashed, draw=black!100, fill=blue!50, thick, inner sep=0pt, minimum size=7.3mm},
			nodeGreenDashed/.style={circle, dashed, draw=black!100, fill=green!50, thick, inner sep=0pt, minimum size=7.3mm},
			nodeCyanDashed/.style={circle, dashed, draw=black!100, fill=cyan!50, thick, inner sep=0pt, minimum size=7.3mm},
			nodeMagentaDashed/.style={circle, dashed, draw=black!100, fill=magenta!50, thick, inner sep=0pt, minimum size=7.3mm},
			nodeYellowDashed/.style={circle, dashed, draw=black!100, fill=yellow!50, thick, inner sep=0pt, minimum size=7.3mm},
			nodeGreyDashed/.style={circle, dashed, draw=black!100, fill=black!25, thick, inner sep=0pt, minimum size=7.3mm}
		]
	}
	{
		\end{tikzpicture}
	}

\newenvironment{tikzpictureHostGraph}
	{
		\begin{tikzpicture}[
			xscale=1.40, yscale=0.68,
			leafUsed/.style={circle, draw=black!100, fill=black!50, thick, inner sep=0pt, minimum size=7.3mm},
			leafUsedBlack/.style={circle, draw=black!100, fill=black!50, thick, inner sep=0pt, minimum size=7.3mm},
			leafUsedRed/.style={circle, draw=black!100, fill=red!50, thick, inner sep=0pt, minimum size=7.3mm},
			leafUsedBlue/.style={circle, draw=black!100, fill=blue!50, thick, inner sep=0pt, minimum size=7.3mm},
			leafUsedGreen/.style={circle, draw=black!100, fill=green!50, thick, inner sep=0pt, minimum size=7.3mm},
			leafUsedCyan/.style={circle, draw=black!100, fill=cyan!50, thick, inner sep=0pt, minimum size=7.3mm},
			leafUsedMagenta/.style={circle, draw=black!100, fill=magenta!50, thick, inner sep=0pt, minimum size=7.3mm},
			leafUsedYellow/.style={circle, draw=black!100, fill=yellow!50, thick, inner sep=0pt, minimum size=7.3mm},
			leafUsedGrey/.style={circle, draw=black!100, fill=black!25, thick, inner sep=0pt, minimum size=7.3mm},
			leafUsedBlackDashed/.style={circle, dashed, draw=black!100, fill=black!50, thick, inner sep=0pt, minimum size=7.3mm},
			leafUsedRedDashed/.style={circle, dashed, draw=black!100, fill=red!50, thick, inner sep=0pt, minimum size=7.3mm},
			leafUsedBlueDashed/.style={circle, dashed, draw=black!100, fill=blue!50, thick, inner sep=0pt, minimum size=7.3mm},
			leafUsedGreenDashed/.style={circle, dashed, draw=black!100, fill=green!50, thick, inner sep=0pt, minimum size=7.3mm},
			leafUsedCyanDashed/.style={circle, dashed, draw=black!100, fill=cyan!50, thick, inner sep=0pt, minimum size=7.3mm},
			leafUsedMagentaDashed/.style={circle, dashed, draw=black!100, fill=magenta!50, thick, inner sep=0pt, minimum size=7.3mm},
			leafUsedYellowDashed/.style={circle, dashed, draw=black!100, fill=yellow!50, thick, inner sep=0pt, minimum size=7.3mm},
			leafUsedGreyDashed/.style={circle, dashed, draw=black!100, fill=black!25, thick, inner sep=0pt, minimum size=7.3mm},
			leafUnused/.style={circle, draw=black!100, fill=white!100, thick, inner sep=0pt, minimum size=7.3mm},
			otherNode/.style={circle, draw=black!50, fill=white!100, thick, inner sep=0pt, minimum size=5mm}
		]
	}
	{
		\end{tikzpicture}
	}

\algnewcommand\algorithmicsubroutine{\textbf{\large Subroutine: }}
\algnewcommand\Subroutine{\item[\algorithmicsubroutine]}
\algnewcommand\algorithmicmain{\textbf{\large Main: }}
\algnewcommand\Main{\item[\algorithmicmain]}
\algnewcommand\Andoperator{\textbf{and }}
\algrenewcommand\Return{\State \algorithmicreturn{} }
\algnewcommand\algorithmicto{\textbf{to}}
\algrenewtext{For}[2]{\algorithmicfor\ #1 \algorithmicto\ #2 \algorithmicdo}
\algnewcommand\algorithmicdownto{\textbf{downto}}
\algblockdefx[Loop]{Fordownto}{EndFor}[3][]{\algorithmicfor\ #2 \algorithmicdownto\ #3 \algorithmicdo}{\algorithmicend\ \algorithmicfor}
\algnewcommand{\LineComment}[1]{\State \(\triangleright\) #1}

\usetikzlibrary{petri,decorations.pathreplacing,decorations.pathmorphing}

\begin{document}

\newtheorem{axiom}{Axiom}
\newtheorem{conjecture}[axiom]{Conjecture}
\newtheorem{corollary}[axiom]{Corollary}
\newtheorem{definition}[axiom]{Definition}
\newtheorem{example}[axiom]{Example}
\newtheorem{lemma}[axiom]{Lemma}
\newtheorem{observation}[axiom]{Observation}
\newtheorem{proposition}[axiom]{Proposition}
\newtheorem{theorem}[axiom]{Theorem}

\newcommand{\rz}{{\mathbb{R}}}
\newcommand{\nz}{{\mathbb{N}}}
\newcommand{\zz}{{\mathbb{Z}}}
\newcommand{\eps}{\varepsilon}
\newcommand{\cei}[1]{\left\lceil #1\right\rceil}
\newcommand{\flo}[1]{\left\lfloor #1\right\rfloor}
\newcommand{\seq}[1]{\langle #1\rangle}

\newcommand{\toDo}{{\fontencoding{OT1}\fontfamily{cmtt}\fontseries{m}\fontshape{sc}\fontsize{12}{16pt}\selectfont \#\#\# to do \#\#\#}}
\newcommand{\toSolve}{{\fontencoding{OT1}\fontfamily{cmtt}\fontseries{m}\fontshape{sc}\fontsize{12}{16pt}\selectfont \#\#\# to solve \#\#\#}}
\newcommand{\reference}{{\fontencoding{OT1}\fontfamily{cmtt}\fontseries{m}\fontshape{sc}\fontsize{12}{16pt}\selectfont \#\#\# reference \#\#\#}}
\newcommand{\otherComment}[1]{{\fontencoding{OT1}\fontfamily{cmtt}\fontseries{m}\fontshape{sc}\fontsize{12}{16pt}\selectfont \#\#\# #1 \#\#\#}}

\newcommand{\qap}{\mbox{\sc QAP}}
\newcommand{\minqap}{\mbox{$\min$-{\sc QAP}}}
\newcommand{\maxqap}{\mbox{$\max$-{\sc QAP}}}
\newcommand{\onion}{\mbox{\sc Onion}}
\newcommand{\ocone}{\mbox{\sc OnionCone}}
\newcommand{\aaa}{\alpha}
\newcommand{\minctv}{\mbox{$\min$-{\sc CTV}}}
\newcommand{\maxctv}{\mbox{$\max$-{\sc CTV}}}
\newcommand{\cbar}{\overline{C}{}}
\newcommand{\var}{\mbox{\sc Var}}
\newcommand*{\defeq}{\mathrel{\vcenter{\baselineskip0.5ex \lineskiplimit0pt
                     \hbox{\scriptsize.}\hbox{\scriptsize.}}}%
                     =}
\newcommand*{\eqdef}{=\mathrel{\vcenter{\baselineskip0.5ex \lineskiplimit0pt
                     \hbox{\scriptsize.}\hbox{\scriptsize.}}}%
                     }

\makeatletter
	\newenvironment{figurehere}
		{\def\@captype{figure}}
		{}
\makeatother

\title{\bf The data arrangement problem on binary trees}
\author{
	\sc Eranda \c{C}ela{\footnotemark[1]}
	\and
	\sc Joachim Schauer{\footnotemark[7]}
	\and
	\sc Rostislav Stan\v{e}k{\footnotemark[7]}
}
\date{}
\maketitle
\renewcommand{\thefootnote}{\fnsymbol{footnote}}
\footnotetext[1]{
	{\tt cela@opt.math.tu-graz.ac.at}. 
	Department of Optimization and Discrete Mathematics, Graz University of Technology, Steyrergasse~30, A-8010 Graz, Austria
}
\footnotetext[7]{
	{\tt \{joachim.schauer, rostislav.stanek\}@uni-graz.at}.
	Department of Statistics and Operations Research, University of Graz, Universitaetsstrasse~15, A-8010 Graz, Austria}	

\begin{abstract}
	The {\em data arrangement problem on regular trees} ({\em DAPT}) consists in assigning the vertices of a  given graph $G$, called the guest graph,  to the leaves of a $d$-regular tree $T$, called the host graph, such that the sum of the pairwise distances of all pairs of leaves in $T$ which correspond to the edges of $G$ is minimised. {\sc Luczak} and {\sc Noble}~\cite{LuczakNoble:OptimalArrangementOfDataInATreeDirectory} have shown that this problem is ${\cal NP}$-hard for every fixed $d \geq 2$. In this paper we show  that the DAPT remains ${\cal NP}$-hard even if the guest graph is a tree, an issue which was posed as  an open question in \cite{LuczakNoble:OptimalArrangementOfDataInATreeDirectory}.
	
	We deal with a special case of the DAPT where both the guest and the host graph are binary regular trees and  provide a $1.015$-approximation algorithm for special case. The solution produced by the algorithm and the corresponding value of the objective function are given in closed form. The analysis of the approximation algorithm involves an auxiliary problem which is interesting on its own, namely the {\em $k$-balanced partitioning problem} ({\em $k$-BPP}) for  binary regular trees and particular choices of $k$. We derive a lower bound for the later problem and obtain a lower bound for the original problem by solving $h_G$ instances of the $k$-BPP, where $h_G$ is the height of the host graph $G$.
	
\end{abstract}
\medskip

\section{Introduction}
	\label{section:introduction}
	Given an undirected graph $G = \big(V(G), E(G)\big)$ with $\big|V(G)\big| = n$, an undirected graph 
$H = \big(V(H), E(H)\big)$
 with $|V(H)| \geq n$ and some subset $B$ of the vertex set of $H$, $B \subseteq V(H)$ with $|B| \geq n$, the
 {\bf generic graph embedding problem (GEP}) consists of finding an injective embedding of the vertices of $G$ 
into the vertices 
in $B$ such that some prespecified objective function is minimised. Throughout this paper we will call $G$ 
{\bf the guest graph} 
and $H$ {\bf the host graph}. A commonly used objective function maps an embedding $\phi\colon V(G) \to B$ to
	\begin{equation}
		\sum_{(i, j) \in E(G)}{d\big(\phi(i), \phi(j)\big)},
	\end{equation}
	where $d(x, y)$ denotes the length of the shortest path between $x$ and $y$ in $H$. The host graph $H$ may 
be a weighted 
or a non-weighted graph; in the second case the path lengths coincide with the respective number of edges. 
Given a non-negative
 number $A \in \rz$, the decision version of the GEP asks whether there is an injective embedding $\phi\colon V(G) \to B$ 
such that 
the objective function does not exceed $A$.
	
	Different versions of the GEP have been studied in the literature; the {\bf linear arrangement problem}, 
where the guest graph is a one dimensional equidistant grid with $n$ vertices is probably the most prominent among them 
(see {\sc Chung}~\cite{Chung:AnOptimalLinearArrangementOfTrees}, {\sc Juvan} and 
{\sc Mohar}~\cite{JuvanMohar:OptimalLinearLabelingsAndEigenvaluesOfGraphs}, 
{\sc Shiloach}~\cite{Shiloach:AMinimumLinearArrangementAlgorithmForUndirectedTrees}).

	This paper deals with the version of the GEP where the guest graph $G$ has $n$ vertices, the host graph $H$ is a
 complete $d$-regular tree of height $\lceil\log_d{n}\rceil$ and the set $B$ consists of the leaves of $H$. From now on 
 we will denote the host graph by $T$. The height of $T$ as specified above guarantees that the number $|B|$ of leaves 
fulfills $|B| \geq n$ and that the number of the direct predecessors of the leaves in $T$ is smaller than $n$. 
Thus $\lceil\log_d{n}\rceil$ is the smallest height of a $d$-regular tree which is able to accommodate an injective
 embedding of the vertices of the guest graph on its leaves. This problem is originally motivated by real problems in 
communication systems and was first posed by {\sc Luczak} and 
{\sc Noble}~\cite{LuczakNoble:OptimalArrangementOfDataInATreeDirectory}. 

We call the above described  version of the GEP the {\bf data arrangement problem on regular trees (DAPT)}. {\sc Luczak} and
 {\sc Noble}~\cite{LuczakNoble:OptimalArrangementOfDataInATreeDirectory} have shown that the DAPT is ${\cal NP}$-hard for every 
fixed $d \geq 2$ and have posed as an open question the computational complexity of the DAPT in the case where the guest graph is a tree.
In this paper we answer this question and show that this particular case of the problem is ${\cal NP}$-hard for every $d\ge 2$.
 In the special case where both the guest graph $G$ and the host graph $T$ are binary regular trees we give a $1.015$-approximation 
algorithm.  
\smallskip
	
	The paper is organised as follows. Section~\ref{section:notationsAndGeneralPropertiesOfTheDAPT} discusses some general 
properties of the problem and introduces the notation used throughout the paper. Section~\ref{section:solutionAlgorithm} 
presents an algorithm for the DAPT on regular binary trees,  where the guest graph is also a regular binary tree. 
Section~\ref{section:approximationRatio:kBalancedPartitioningProblem}  deals with the {\em $k$-balanced partitioning problem} 
($k$-BPP) in binary regular trees. This version of the $k$-BPP serves as an auxiliary problem 
in the sense that it  leads to  a lower bound for the objective function value of the DAPT on regular binary trees. 
In Section~\ref{section:approxratio}
we  use the  auxiliary problem and the lower bound mentioned above to analyze  the algorithm presented 
in Section~\ref{section:solutionAlgorithm} and show that the latter is an  $1.015$-approximation  algorithm. 
In Section~\ref{section:complexity} it is proven that the DAPT is ${\cal NP}$-hard for every $d\ge 2$ even if the guest graph 
is a tree.
Finally, in Section~\ref{section:finalNotesConclusionsAndOutlook} we  provide some final notes, conclusions and questions for
 further research.

\section{Notations and general properties of the DAPT}
	\label{section:notationsAndGeneralPropertiesOfTheDAPT}
	First, we formally define a {\bf {\em d}-regular tree} as follows:
	\begin{definition}[{\bf {\em d}-regular tree}]
		\label{definition:dRegularTree}
		A tree $T = \big(V(T), E(T)\big)$ is called a {\bf {\em d}-regular tree},
 $d \in \nz$, $d \geq 2$, if
		\begin{enumerate}
			\item it contains a specific vertex $v_1 \in V$ of degree $d$ which is called the {\bf root} of $T$ 
and is also denoted  by $r(T)$ ,
			\item every vertex but the leaves and the root has degree $d + 1$ and
			\item there is a number $h \in \nz$ such that the length $d(l, v_1)$ of the path between the root $v_1$ and a
 leaf $l$ equals $h$ for every leaf $l$ of $T$.
		\end{enumerate}
		The number $h$ is called the {\bf height} of the tree $T$, and is also denoted by $h(T)$. For every vertex $v \in V \backslash \{v_1\}$, i.e. for any 
vertex $v$ but the root $v_1$, the unique neighbour of $v$ in the path between $v_1$ and $v$ in $T$ is called the {\bf father} of $v$. 
All other neighbours of $v$ (if any) are called the {\bf children of $v$}. The neighbours of the root $v_1$ are called {\bf children 
of $v_1$}. The {\bf level of a vertex} $v$, denoted by $\mbox{level}(v)$, is the length (i.e. the number of edges) of the unique path 
joining $v$ and the root $v_1$ of the tree. Thus in a $d$-regular tree of height $h$ the level of each leaf equals $h$, whereas the 
level of the root $v_1$ equals $0$. All vertices $w$, $w \neq v$, of the unique path joining $v$ and the root $v_1$ of the tree are 
called {\bf ancestors of $v$}. Given two vertices $v$ and $u$ their {\bf most recent common ancestor $w$} is their common ancestor 
with the highest level, i.e. $w = \mbox{argmax}\{\mbox{level}(t)\colon t \text{ is a common ancesteor of $v$ and $u$}\}$.
		
	A {\bf subtree of {\em k}-th order} of a $d$-regular tree $T$ is a $d$-regular subtree $T^\prime$
of $T$ of  height $h(T^\prime) = h(T) - k$, rooted at some vertex of level $k$ in $T$. A subtree of first order  will be  
called a
{\bf basic subtree}.
	\end{definition}

	Consider a guest graph $G = (V, E)$ with $n$ vertices, and a host graph $T$ which is a $d$-regular tree of height $h$, 
$h \defeq \lceil \log_d{n} \rceil$. Let $B$ be the set of leaves of $T$. Notice that due to the above choice of $h$ we get the
 following 
upper bound for the number $b = |B|$ of leaves:
	\begin{equation}
		\label{equation:b}
		b \defeq |B| = d^h = d^{h-1} d < nd.
	\end{equation}
	
	\begin{definition}[{\bf data arrangement problem on regular trees}]
		\label{definition:dAPT}
		Given a guest graph $G = (V, E)$ with $|V|=n$ and a host graph $T$ which is a $d$-regular tree with set of 
leaves $B$ and height 
equal to $\lceil \log_d{n} \rceil$, an
 {\bf arrangement} is an injective mapping $\phi\colon V \to B$. The {\bf data arrangement problem on regular trees (DAPT)} 
asks for an arrangement $\phi$ that minimises the objective value $OV(G, d, \phi)$
		\begin{equation}
			\label{equation:definition:dAPT:oV}
			OV(G, d, \phi) \defeq \sum_{(u, v) \in E}{d_T\big(\phi(u), \phi(v)\big)},
		\end{equation}
		where $d_T\big(\phi(u), \phi(v)\big)$ denotes the length of the unique $\phi(u)$-$\phi(v)$-path in the $d$-regular
 tree $T$. Such an arrangement is called an {\bf optimal arrangement}. The corresponding value of the objective functions is called the {\bf optimal value} of the problem.  An instance of the DAPT is fully determined by the guest graph 
and the parameter $d$ of the $d$-regular tree $T$ which serves as  a host graph. Such an instance of the problem will be denoted by 
$DAPT(G, d)$ and its optimal value will be denoted by $OPT(G,d)$.
	\end{definition}
	
	\begin{theorem}
		\label{theorem:dAPTIsNPHard}
		The DAPT is ${\cal NP}$-hard for every fixed $d \geq 2$.
		
		\begin{proof}
			See {\sc Luczak} and {\sc Noble}~\cite{LuczakNoble:OptimalArrangementOfDataInATreeDirectory}.
		\end{proof}
	\end{theorem}

\begin{example}
		\label{example:dataArrangementProblemOnRegularTrees}	
	A  guest graph $G$ of height $3$ is shown in  
Figure~\ref{figure:guestGraphGForHG3DepictingTheArrangementPhi}. 
Figure  \ref{figure:guestGraphGForHG3DepictingTheArrangementPhiA} represnets the same guest graph $G$,
 but with another colouring of its vertices; the role of the colouring will be explained below. 
Figures  \ref{figure:arrangementPhiForHG3} and \ref{figure:arrangementPhiUForHG3} depict  a feasible $\phi$
arrangement and an optimal arrangement $\phi_A$ of $G$, yielding  the objective function values 
$OV(G, 2, \phi) = 58$ and $OV(G, 2, \phi_A) = 56$, respectively. 

Note that the labels in the vertices of the guest graphs denote the index of the  vertices 
 in the so-called   canonical ordering (defined below). 
The labels of the leaves  in the host graphs represent the arrangement: the label of each leaf coincides 
with the index of the vertex arranged at that leaf (in the canocinal ordering). 

 The colours should help to capture some properties of the arrangement at a glance: the set of vertices 
of a certain colour in  the guest graph is arranged at the set of leaves of the same colour in the host 
graph $T$. Moreover,  some of the vertices  in the guest graph have a dashed boundary,
 the others have a solid boundary. The graphical representation of an arrangement preserves  the boundary property, in the sense that  
vertices with  a dashed boundary in $G$ are arranged at  dashed-boundary  leaves of the same colour in $T$. 
	\end{example}
	
\begin{definition}[{\bf canonical order}]	
The  canonical order of the vertices of the guest graph and the canonical order of the leaves of the host 
graph are defined 
recursively as follows.

\begin{itemize}
\item[(a)] 
 The {\bf canonical order of the leaves of a {\em d}-regular tree  $T$} is an arbitrary but fixed order 
if $h(T)=1$.
If $h(T)>1$ then an order of the leaves is called canonical if (i) it implies a canonical order of 
 the leaves of every basic subtree of $T$,  and (ii) for an arbitrary but fixed order of the children $ch_1,\ldots,ch_d$ of 
the root $r(T)$ of $T$ all leaves of 
the basic subtree rooted at $ch_i$ precede all leaves of the basic subtree rooted at $ch_j$, for $i<j$, $i,j\in \{1,2,\ldots,d\}$,
 in this order. 
\item[(b)] 
 A {\bf canonical ordering of the vertices  of a {\em d}-regular tree $T$} is the unique   order 
if 
$h(T)=0$. If $h(T)\ge 1$, a canonical order of the vertices of $T$ is an order obtained by extending the 
canonical order of 
the vertices
 of the $d$-regular tree $T'$ of height $h(T')=h(T)-1$ obtained from $T$ by removing all of its leaves and 
fulfilling the following 
two properties:
(i) all vertices of $T'$ precede the leaves of $T$, and (ii) for any two leaves $a$ and $b$ of $T'$, if   $a$ precedes $b$, 
then  all children of $a$ in $T$ precede all children of $b$ in $T$.  
\end{itemize}
\end{definition} 

	If the leaves of a $d$-regular tree $T$ are ordered according to the canonical order as above, then the pairwise 
distances
 between them are given by a simple formula.
	\begin{observation}
		\label{observation:distance}
		Let $T$ be a $d$-regular tree of height $h \defeq h(T)$ and let its $b$ leaves be labelled according to the canonical 
order $b_1\prec b_2\prec \ldots \prec b_b$. Then the  distances between the leaves in $T$ are given as $d_T(b_i, b_j)= 2 l$, 
where
		\begin{equation}
			\label{equation:observation:distance:l}
			l \defeq \min \left\{ k \in \{1, 2, \ldots, h\}\colon \left\lfloor \frac{i - 1}{d^k} \right\rfloor = 
\left \lfloor\frac{j - 1}{d^k} \right\rfloor \right\},
		\end{equation}
		for all $i, j \in \{1, 2, \ldots, b\}$. If vertex $u$ is the most recent common ancestor of $b_i$ and $b_j$, 
then $h - l = \mbox{level}(u)$.
		
		\begin{proof}
			See {\sc \c{C}ela} and {\sc Stan\v{e}k}~\cite{CelaStanek:HeuristicsForTheDataArrangementProblemOnRegularTrees}.
		\end{proof}
	\end{observation}
	
	In this paper we deal with the special case where both the guest graph $G$ and the host graph $T$ are binary regular
 trees; an instance of this problem is fully specified by the guest graph $G$  and will be denoted by  $DAPT(G,2)$.
From now on we  denote by  $h_G$ the height of the guest graph $G$ and by $h$ the height of the host graph $T$.
  Moreover we will 
always use the canonical order $v_1\prec v_2 \prec \ldots 
\prec v_n$ of the vertices $v_i$, $1\le i \le n$, $n \defeq |V(G)|$,  of the guest graph, and the canonical order 
$b_1\prec b_2\prec \ldots \prec b_b$ of the $b$ leaves  of the host graph $T$  as in the 
observation above.  See e.g.\ Figure~\ref{figure:guestGraphGForHG3DepictingTheArrangementPhi} for an illustration of the canonical order of the vertices 
of  
a regular tree of height $3$; for simplicity  we specify  the indices $i$, $1\le i\le 15$  instead of the labels $v_i$, 
$1\le i\le 15$. 
	
	In the follwoing  we list some obvious equalities which will be used through the rest of this paper.
	\begin{observation}
		\label{observation:someObviousEqualities}
		\begin{alignat}{2}
			\label{equation:observation:someObviousEqualities:n}
			n &
				= 2^{h_G + 1} - 1\\
			\label{equation:observation:someObviousEqualities:h}
			h &
				= \lceil\log_2{(2^{h_G + 1} - 1)}\rceil = h_G + 1\\
			\label{equation:observation:someObviousEqualities:b}
			b &
				= 2^{h_G + 1} = 2 \cdot 2^{h_G} = n + 1
		\end{alignat}
	\end{observation}

\section{An approximation algorithm}\label{section:solutionAlgorithm}
In this section we  describe an approximation algorithm $\cal A$ for the $DAPT(G,2)$ where the guest graph $G$ is a binary 
regular tree.
Later in Section~\ref{section:approxratio} it will be shown that this is an $\alpha$-approximation algorithm with 
$\alpha=\frac{203}{200}$, i.e.\ 
$ OV(G,2,\phi_A)\le \alpha OV(G,2,\phi_{\ast})$ holds for every binary regular tree $G$, where $\phi_{\ast}$ denotes the optimal 
arrangement of $DAPT(G,2)$ and $\phi_A$ denotes the arrangement computed by algorithm $\cal A$ described below.

	\begin{breakablealgorithm}
\title{\bf An approximation  algorithm} 
		\begin{algorithmic}[1]

			\Require binary regular tree $G = (V, E)$ of height $h_G$ whose vertices are labelled according to the canonical order
			\Ensure arrangement $\phi_A$
			\State $b \defeq 2^{h_G + 1}$;
			\label{algorithm:arrangementPhiU:settingOfB}
			\If{$h_G = 0$}
			\label{algorithm:arrangementPhiU:hGEqual0Condition}
				\State $\phi_A(v_1) \defeq b_1$;
				\label{algorithm:arrangementPhiU:hGEqual0Statement}
			\Else [$h_G > 0$]
			\label{algorithm:arrangementPhiU:hGEqual0ConditionElse}
				\State solve the problem for the basic subtrees $\widehat{G_1}$ and $\widehat{G_2}$ of 
height $\widehat{h_G} = h_G - 1$ and obtain the respective arrangements $\widehat{\phi_A}^{(1)}$ and $\widehat{\phi_A}^{(2)}$;
				\label{algorithm:arrangementPhiU:hGEqual0ConditionElseSolveTheSubproblems}
				\State arrange the vertices of the left basic subtree on the leaves $b_1$, $b_2$, \ldots, 
$b_{\frac{1}{2} b}$ according to the arrangement $\widehat{\phi_A}^{(1)}$ and the vertices of the right basic subtree on the leaves 
$b_{\frac{1}{2} b + 1}$, $b_{\frac{1}{2} b + 2}$, \ldots, $b_{b}$ according to the arrangement $\widehat{\phi_A}^{(2)}$;
				\label{algorithm:arrangementPhiU:hGEqual0ConditionElseArrangeTheSubproblems}
				\State $\phi_A(v_1) \defeq b_{\frac{1}{2} b}$;
				\label{algorithm:arrangementPhiU:hGEqual0ConditionElseArrangeTheRoot}
				\If{$h_G$ is odd and $h_G \geq 3$}
				\label{algorithm:arrangementPhiU:hGEqual0ConditionElsePairExchangeCondition}
					\State exchange the vertices arranged on the leaves $b_{\frac{1}{4} b - 1}$ 
and $b_{\frac{1}{2} b}$ (pair-exchange);
					\label{algorithm:arrangementPhiU:hGEqual0ConditionElsePairExchangeStatement}
				\EndIf
			\EndIf
	\Return $\phi_A$;
			\label{algorithm:arrangementPhiU:return}
		\end{algorithmic}
		\caption{The  approximation algorithm  ${\cal A}$ computes the arrangement $\phi_A$.}
		\label{algorithm:arrangementPhiU}
	\end{breakablealgorithm}
 \smallskip
	
\noindent In the following we apply this algorithm on an instamce of $DAPT(G,2)$ with $h_G=3$.
Observe that the leaf $b_{\frac{1}{2} b}$ is always free prior to the execution of  pseudocode 
line~\ref{algorithm:arrangementPhiU:hGEqual0ConditionElseArrangeTheRoot} due to the recursion and due to the assignment in 
pseudocode line~\ref{algorithm:arrangementPhiU:hGEqual0Statement}.
	
	\begin{example}
		Consider the guest graph $G = (V, E)$ of height $h_G = 3$ depicted in Figure~\ref{figure:guestGraphGForHG3DepictingTheArrangementPhi} and apply algorithm $\cal A$.
		\begin{figure}[htb]
			\begin{multicols}{2}
				\begin{figurehere}
					\centering
					\begin{comment:upperBound}
						\begin{tikzpictureGuestGraph}
							\node[nodeRedDashed] (node1) at (1.75, 3.00) {1};
				
							\node[nodeBlackDashed] (node2) at (0.75, 2.00) {2};
							\node[nodeBlueDashed] (node3) at (2.75, 2.00) {3};
				
							\node[nodeBlack] (node4) at (0.25, 1.00) {4};
							\node[nodeRed] (node5) at (1.25, 1.00) {5};
							\node[nodeBlue] (node6) at (2.25, 1.00) {6};
							\node[nodeGreen] (node7) at (3.25, 1.00) {7};
				
							\node[nodeBlack] (node8) at (0.00, 0.00) {8};
							\node[nodeBlackDashed] (node9) at (0.50, 0.00) {9};
							\node[nodeRed] (node10) at (1.00, 0.00) {10};
							\node[nodeRedDashed] (node11) at (1.50, 0.00) {11};
							\node[nodeBlue] (node12) at (2.00, 0.00) {12};
							\node[nodeBlueDashed] (node13) at (2.50, 0.00) {13};
							\node[nodeGreen] (node14) at (3.00, 0.00) {14};
							\node[nodeGreenDashed] (node15) at (3.50, 0.00) {15};
							
							\draw (node1) to (node2);
							\draw (node1) to (node3);
				
							\draw (node2) to (node4);
							\draw (node2) to (node5);
							\draw (node3) to (node6);
							\draw (node3) to (node7);
				
							\draw (node4) to (node8);
							\draw (node4) to (node9);
							\draw (node5) to (node10);
							\draw (node5) to (node11);
							\draw (node6) to (node12);
							\draw (node6) to (node13);
							\draw (node7) to (node14);
							\draw (node7) to (node15);
						\end{tikzpictureGuestGraph}
					\end{comment:upperBound}
					\caption{A guest graph $G = (V, E)$ (binary regular tree of height $h_G = 3$). The colours are related to  the arrangement $\phi$ depicted in Figure~\ref{figure:arrangementPhiForHG3}.}
					\label{figure:guestGraphGForHG3DepictingTheArrangementPhi}
				\end{figurehere}

				\begin{figurehere}
					\centering
					\begin{comment:upperBound}
						\begin{tikzpictureGuestGraph}
							\node[nodeBlackDashed] (node1) at (1.75, 3.00) {1};
				
							\node[nodeBlackDashed] (node2) at (0.75, 2.00) {2};
							\node[nodeBlueDashed] (node3) at (2.75, 2.00) {3};
				
							\node[nodeBlack] (node4) at (0.25, 1.00) {4};
							\node[nodeRed] (node5) at (1.25, 1.00) {5};
							\node[nodeBlue] (node6) at (2.25, 1.00) {6};
							\node[nodeGreen] (node7) at (3.25, 1.00) {7};
				
							\node[nodeBlack] (node8) at (0.00, 0.00) {8};
							\node[nodeRedDashed] (node9) at (0.50, 0.00) {9};
							\node[nodeRed] (node10) at (1.00, 0.00) {10};
							\node[nodeRedDashed] (node11) at (1.50, 0.00) {11};
							\node[nodeBlue] (node12) at (2.00, 0.00) {12};
							\node[nodeBlueDashed] (node13) at (2.50, 0.00) {13};
							\node[nodeGreen] (node14) at (3.00, 0.00) {14};
							\node[nodeGreenDashed] (node15) at (3.50, 0.00) {15};
							
							\draw (node1) to (node2);
							\draw (node1) to (node3);
				
							\draw (node2) to (node4);
							\draw (node2) to (node5);
							\draw (node3) to (node6);
							\draw (node3) to (node7);
				
							\draw (node4) to (node8);
							\draw (node4) to (node9);
							\draw (node5) to (node10);
							\draw (node5) to (node11);
							\draw (node6) to (node12);
							\draw (node6) to (node13);
							\draw (node7) to (node14);
							\draw (node7) to (node15);
						\end{tikzpictureGuestGraph}
					\end{comment:upperBound}
					\caption{A guest graph $G = (V, E)$ (binary regular tree of height $h_G = 3$). The colours are related to the arrangement $\phi$ depicted in Figure~\ref{figure:arrangementPhiUForHG3}.}
					\label{figure:guestGraphGForHG3DepictingTheArrangementPhiA}
				\end{figurehere}
			\end{multicols}
		\end{figure}
		
\noindent Since $h_G = 3 > 0$, the algorithm executes  the else part beginning in pseudocode 
line~\ref{algorithm:arrangementPhiU:hGEqual0ConditionElse}. 
In pseudocode lines~\ref{algorithm:arrangementPhiU:hGEqual0ConditionElseSolveTheSubproblems} and 
\ref{algorithm:arrangementPhiU:hGEqual0ConditionElseArrangeTheSubproblems}  the arrangements for both basic subtrees, 
i.e.\ for graphs of height $\widehat{h_G} = h_G - 1 = 3 - 1 = 2$ are computed.
(The arrangement $\widehat{\phi_A}$ for $\widehat{h_G} = 2$ is depicted in Figure~\ref{figure:arrangementPhiUForHG2} in Appendix.)
 In the next step, the root is arranged at the middle leaf (see pseudocode 
line~\ref{algorithm:arrangementPhiU:hGEqual0ConditionElseArrangeTheRoot}) and  the arrangement $\phi$ depicted in 
Figure~\ref{figure:arrangementPhiForHG3} is obtained. The label of each leaf corresponds to the index of the vertex 
of the guest graph  arranged at that leaf.  The  objective value which  corresponds to  arrangement $\phi$ is   $OV(G, 2, \phi) = 58$.

		\begin{figure}[htb]
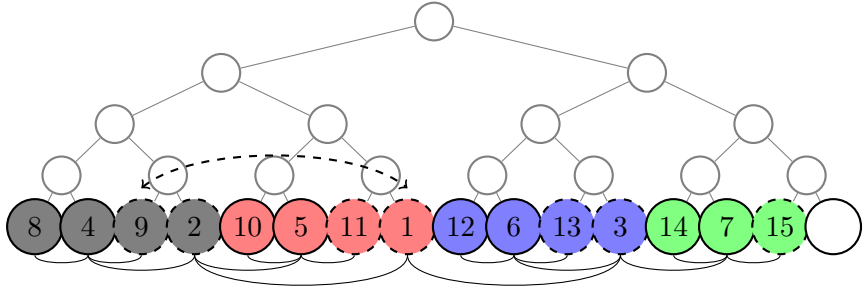

			\centering
			\begin{comment:upperBound}
				\begin{tikzpictureHostGraph}
					\node[otherNode] (node1) at (3.75, 4.00) {};
		
					\node[otherNode] (node2) at (1.75, 3.00) {};
					\node[otherNode] (node3) at (5.75, 3.00) {};
		
					\node[otherNode] (node4) at (0.75, 2.00) {};
					\node[otherNode] (node5) at (2.75, 2.00) {};
					\node[otherNode] (node6) at (4.75, 2.00) {};
					\node[otherNode] (node7) at (6.75, 2.00) {};
		
					\node[otherNode] (node8) at (0.25, 1.00) {};
					\node[otherNode] (node9) at (1.25, 1.00) {};
					\node[otherNode] (node10) at (2.25, 1.00) {};
					\node[otherNode] (node11) at (3.25, 1.00) {};
					\node[otherNode] (node12) at (4.25, 1.00) {};
					\node[otherNode] (node13) at (5.25, 1.00) {};
					\node[otherNode] (node14) at (6.25, 1.00) {};
					\node[otherNode] (node15) at (7.25, 1.00) {};
		
					\node[leafUsedBlack] (node16) at (0.00, 0.00) {8};
					\node[leafUsedBlack] (node17) at (0.50, 0.00) {4};
					\node[leafUsedBlackDashed] (node18) at (1.00, 0.00) {9};
					\node[leafUsedBlackDashed] (node19) at (1.50, 0.00) {2};
					\node[leafUsedRed] (node20) at (2.00, 0.00) {10};
					\node[leafUsedRed] (node21) at (2.50, 0.00) {5};
					\node[leafUsedRedDashed] (node22) at (3.00, 0.00) {11};
					\node[leafUsedRedDashed] (node23) at (3.50, 0.00) {1};
					\node[leafUsedBlue] (node24) at (4.00, 0.00) {12};
					\node[leafUsedBlue] (node25) at (4.50, 0.00) {6};
					\node[leafUsedBlueDashed] (node26) at (5.00, 0.00) {13};
					\node[leafUsedBlueDashed] (node27) at (5.50, 0.00) {3};
					\node[leafUsedGreen] (node28) at (6.00, 0.00) {14};
					\node[leafUsedGreen] (node29) at (6.50, 0.00) {7};
					\node[leafUsedGreenDashed] (node30) at (7.00, 0.00) {15};
					\node[leafUnused] (node31) at (7.50, 0.00) {};
		
					\draw [-, black!50] (node1) to (node2);
					\draw [-, black!50] (node1) to (node3);
		
					\draw [-, black!50] (node2) to (node4);
					\draw [-, black!50] (node2) to (node5);
					\draw [-, black!50] (node3) to (node6);
					\draw [-, black!50] (node3) to (node7);
		
					\draw [-, black!50] (node4) to (node8);
					\draw [-, black!50] (node4) to (node9);
					\draw [-, black!50] (node5) to (node10);
					\draw [-, black!50] (node5) to (node11);
					\draw [-, black!50] (node6) to (node12);
					\draw [-, black!50] (node6) to (node13);
					\draw [-, black!50] (node7) to (node14);
					\draw [-, black!50] (node7) to (node15);
		
					\draw [-, black!50] (node8) to (node16);
					\draw [-, black!50] (node8) to (node17);
					\draw [-, black!50] (node9) to (node18);
					\draw [-, black!50] (node9) to (node19);
					\draw [-, black!50] (node10) to (node20);
					\draw [-, black!50] (node10) to (node21);
					\draw [-, black!50] (node11) to (node22);
					\draw [-, black!50] (node11) to (node23);
					\draw [-, black!50] (node12) to (node24);
					\draw [-, black!50] (node12) to (node25);
					\draw [-, black!50] (node13) to (node26);
					\draw [-, black!50] (node13) to (node27);
					\draw [-, black!50] (node14) to (node28);
					\draw [-, black!50] (node14) to (node29);
					\draw [-, black!50] (node15) to (node30);
					\draw [-, black!50] (node15) to (node31);
					
					\draw [-] (node23) to [out=-90,in=-90] (node19);
					\draw [-] (node23) to [out=-90,in=-90] (node27);
		
					\draw [-] (node19) to [out=-90,in=-90] (node17);
					\draw [-] (node19) to [out=-90,in=-90] (node21);
					\draw [-] (node27) to [out=-90,in=-90] (node25);
					\draw [-] (node27) to [out=-90,in=-90] (node29);
		
					\draw [-] (node17) to [out=-90,in=-90] (node16);
					\draw [-] (node17) to [out=-90,in=-90] (node18);
					\draw [-] (node21) to [out=-90,in=-90] (node20);
					\draw [-] (node21) to [out=-90,in=-90] (node22);
					\draw [-] (node25) to [out=-90,in=-90] (node24);
					\draw [-] (node25) to [out=-90,in=-90] (node26);
					\draw [-] (node29) to [out=-90,in=-90] (node28);
					\draw [-] (node29) to [out=-90,in=-90] (node30);
					
					\draw [<->, dashed, thick] (1.00, 0.75) to [out=60,in=120] (3.50, 0.75);
				\end{tikzpictureHostGraph}
			\end{comment:upperBound}
			\caption{Arrangement $\phi$ with objective function value $OV(G, 2, \phi) = 58$.}
			\label{figure:arrangementPhiForHG3}
		\end{figure}
		
		Next  consider the condition in pseudocode
 line~\ref{algorithm:arrangementPhiU:hGEqual0ConditionElsePairExchangeCondition}: 
Since $h_G = 3$ is odd and $h_G = 3 \geq 3$, the pair-exchange marked in Figure~\ref{figure:arrangementPhiForHG3} by 
arrows is performed. The value of the objective function corresponding to the resulting arrangement $\phi_A$ 
in Figure~\ref{figure:arrangementPhiUForHG3}
is $OV(G, 2, \phi_A) = 56$. The guest graph coloured according to this arrangement is depicted in Figure~\ref{figure:guestGraphGForHG3DepictingTheArrangementPhiA}. In fact, this arrangement is optimal, but in general Algorithm~\ref{algorithm:arrangementPhiU} 
does not yield an optimal arrangement.
		
		\begin{figure}[htb]
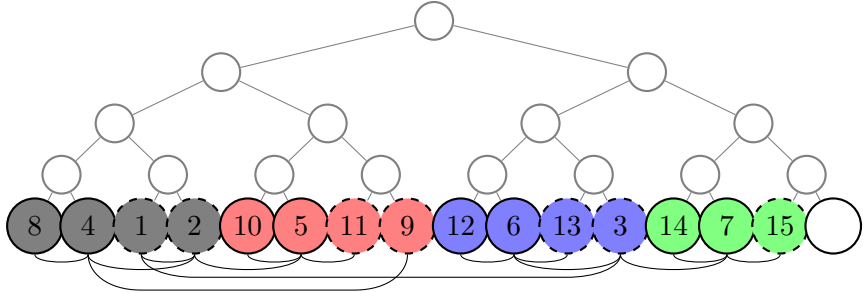

			\centering
			\begin{comment:upperBound}
				\begin{tikzpictureHostGraph}
					\node[otherNode] (node1) at (3.75, 4.00) {};
		
					\node[otherNode] (node2) at (1.75, 3.00) {};
					\node[otherNode] (node3) at (5.75, 3.00) {};
		
					\node[otherNode] (node4) at (0.75, 2.00) {};
					\node[otherNode] (node5) at (2.75, 2.00) {};
					\node[otherNode] (node6) at (4.75, 2.00) {};
					\node[otherNode] (node7) at (6.75, 2.00) {};
		
					\node[otherNode] (node8) at (0.25, 1.00) {};
					\node[otherNode] (node9) at (1.25, 1.00) {};
					\node[otherNode] (node10) at (2.25, 1.00) {};
					\node[otherNode] (node11) at (3.25, 1.00) {};
					\node[otherNode] (node12) at (4.25, 1.00) {};
					\node[otherNode] (node13) at (5.25, 1.00) {};
					\node[otherNode] (node14) at (6.25, 1.00) {};
					\node[otherNode] (node15) at (7.25, 1.00) {};
		
					\node[leafUsedBlack] (node16) at (0.00, 0.00) {8};
					\node[leafUsedBlack] (node17) at (0.50, 0.00) {4};
					\node[leafUsedBlackDashed] (node18) at (1.00, 0.00) {1};
					\node[leafUsedBlackDashed] (node19) at (1.50, 0.00) {2};
					\node[leafUsedRed] (node20) at (2.00, 0.00) {10};
					\node[leafUsedRed] (node21) at (2.50, 0.00) {5};
					\node[leafUsedRedDashed] (node22) at (3.00, 0.00) {11};
					\node[leafUsedRedDashed] (node23) at (3.50, 0.00) {9};
					\node[leafUsedBlue] (node24) at (4.00, 0.00) {12};
					\node[leafUsedBlue] (node25) at (4.50, 0.00) {6};
					\node[leafUsedBlueDashed] (node26) at (5.00, 0.00) {13};
					\node[leafUsedBlueDashed] (node27) at (5.50, 0.00) {3};
					\node[leafUsedGreen] (node28) at (6.00, 0.00) {14};
					\node[leafUsedGreen] (node29) at (6.50, 0.00) {7};
					\node[leafUsedGreenDashed] (node30) at (7.00, 0.00) {15};
					\node[leafUnused] (node31) at (7.50, 0.00) {};
		
					\draw [-, black!50] (node1) to (node2);
					\draw [-, black!50] (node1) to (node3);
		
					\draw [-, black!50] (node2) to (node4);
					\draw [-, black!50] (node2) to (node5);
					\draw [-, black!50] (node3) to (node6);
					\draw [-, black!50] (node3) to (node7);
		
					\draw [-, black!50] (node4) to (node8);
					\draw [-, black!50] (node4) to (node9);
					\draw [-, black!50] (node5) to (node10);
					\draw [-, black!50] (node5) to (node11);
					\draw [-, black!50] (node6) to (node12);
					\draw [-, black!50] (node6) to (node13);
					\draw [-, black!50] (node7) to (node14);
					\draw [-, black!50] (node7) to (node15);
		
					\draw [-, black!50] (node8) to (node16);
					\draw [-, black!50] (node8) to (node17);
					\draw [-, black!50] (node9) to (node18);
					\draw [-, black!50] (node9) to (node19);
					\draw [-, black!50] (node10) to (node20);
					\draw [-, black!50] (node10) to (node21);
					\draw [-, black!50] (node11) to (node22);
					\draw [-, black!50] (node11) to (node23);
					\draw [-, black!50] (node12) to (node24);
					\draw [-, black!50] (node12) to (node25);
					\draw [-, black!50] (node13) to (node26);
					\draw [-, black!50] (node13) to (node27);
					\draw [-, black!50] (node14) to (node28);
					\draw [-, black!50] (node14) to (node29);
					\draw [-, black!50] (node15) to (node30);
					\draw [-, black!50] (node15) to (node31);
					
					\draw [-] (node18) to [out=-90,in=-90] (node19);
					\draw [-] (node18) to [out=-90,in=180] (1.50, -1.00) to [out=0,in=180] (5.00, -1.00) 
to [out=0,in=-90] (node27);
		
					\draw [-] (node19) to [out=-90,in=-90] (node17);
					\draw [-] (node19) to [out=-90,in=-90] (node21);
					\draw [-] (node27) to [out=-90,in=-90] (node25);
					\draw [-] (node27) to [out=-90,in=-90] (node29);
		
					\draw [-] (node17) to [out=-90,in=-90] (node16);
					\draw [-] (node17) to [out=-90,in=180] (1.00, -1.25) to [out=0,in=180] (3.00, -1.25) 
to [out=0,in=-90] (node23);
					\draw [-] (node21) to [out=-90,in=-90] (node20);
					\draw [-] (node21) to [out=-90,in=-90] (node22);
					\draw [-] (node25) to [out=-90,in=-90] (node24);
					\draw [-] (node25) to [out=-90,in=-90] (node26);
					\draw [-] (node29) to [out=-90,in=-90] (node28);
					\draw [-] (node29) to [out=-90,in=-90] (node30);
				\end{tikzpictureHostGraph}
			\end{comment:upperBound}
			\caption{Arrangement $\phi_A$ obtained from Algorithm~\ref{algorithm:arrangementPhiU} with objective function value  $OV(G, 2, \phi_A) = 56$.}
			\label{figure:arrangementPhiUForHG3}
		\end{figure}
\end{example}
	
	Next we give a closed formula for the objective function value corresponding to the arrangement $\phi_A$ computed by the  
algorithm ${\cal A}$.
	
\begin{lemma}\label{lemma:improvementForEveryPairExchange}
		Let the guest graph $G = (V, E)$ and the host graph $T$ be  binary regular trees  of heights $h_G$ and $h = h_G + 1$ 
respectively, where $h_G$ is odd, $h_G \geq 3$. Then the pair-exchange defined in Algorithm~\ref{algorithm:arrangementPhiU} in
 pseudocode 
line~\ref{algorithm:arrangementPhiU:hGEqual0ConditionElsePairExchangeStatement} 
decreases by $1$ the number of edges which contribute to the objective by $4$, increases by $1$ the number of edges which 
contribute to 
the objective value by $2$, and does not change the number of edges which  contribute to 
the objective value by $2i$ for $i\ge 3$. Summarizing such a  pair-exchange 
improves the value of the objective function  by $2$ as compared to the value corresponding to the arrangement available
 prior to this pair-exchange.
		
		\begin{proof}
Let $\phi$ be the  arrangement available prior to the pair-exchange steps done in pseudocode 
line~\ref{algorithm:arrangementPhiU:hGEqual0ConditionElsePairExchangeStatement} in  Algorithm~\ref{algorithm:arrangementPhiU}. 
Denote the arrangement obtained after the pair-exchange by $\phi_A$. 
Consider the vertices and edges which are affected by the pair-exchange in pseudocode 
line~\ref{algorithm:arrangementPhiU:hGEqual0ConditionElsePairExchangeStatement}.
						According to the algorithm (see  pseudocode 
line~\ref{algorithm:arrangementPhiU:hGEqual0ConditionElseArrangeTheRoot}) the root $v_1$ of $G$ is arranged   
on the leaf $b_{\frac{1}{2} b}$ 
  and its left and right children $v_2$ and $v_3$ are arranged   on the leaves $b_{\frac{1}{4} b}$ and $b_{\frac{3}{4} b}$, respectively.
(Recall that  the pair-exchange is performed only if $h_G$ is odd). Moreover the algorithm places the rightmost leaf, say $x$, 
 of the basic subtree of 
$G$ rooted at the child $v_2$ on leaf  $b_{\frac{1}{4} b-1}$ of $T$. So the pair-exchange involves the vertices $v_1$ and $x$ of $G$
 and $\phi_A(v_1)=\phi(x)$, $\phi_A(x)=\phi(v_1)$, $\phi_A(y)=\phi(y)$, for $y\in V\setminus \{v_1,x\}$, hold. 
The change $\Delta$ in  the value of the objective functions corresponding to $\phi$ and $\phi_A$, respectively, 
 is then  given as follows
\[
 \Delta \defeq OV(G,2,\phi)-OV(G,2,\phi_A)=\sum_{\stackrel{v\in V\setminus\{x\}}{\{v,v_1\}\in E}} d_T\big(\phi(v),\phi(v_1)\big)+\]
\[
\sum_{\stackrel{v\in V\setminus\{v_1\}}{\{v,x\}\in E}} d_T\big(\phi(v),\phi(x)\big)-\sum_{\stackrel{v\in V\setminus\{x\}}{\{v,v_1\}\in E}} d_T\big(\phi_A(v),\phi_A(v_1)\big)-
\sum_{\stackrel{v\in V\setminus\{v_1\}}{\{v,x\}\in E}} d_T\big(\phi_A(v),\phi_A(x)\big)=\]
\[
\sum_{\stackrel{v\in V\setminus\{x\}}{\{v,v_1\}\in E}} \Big [d_T\big(\phi(v),\phi(v_1)\big)-d_T\big(\phi_A(v),\phi_A(v_1)\big) \Big ] +\]
\[
\sum_{\stackrel{v\in V\setminus\{v_1\}}{\{v,x\}\in E}}
\Big [d_T\big(\phi(v),\phi(x)\big)-d_T\big(\phi_A(v),\phi_A(x)\big) \Big ]=\]
\[
\sum_{\stackrel{v\in V\setminus\{x\}}{\{v,v_1\}\in E}} \Big [d_T\big(\phi(v),\phi(v_1)\big)-d_T\big(\phi(v),\phi(x)\big) \Big] +\]
\[
\sum_{\stackrel{v\in V\setminus\{v_1\}}{\{v,x\}\in E}} \Big [d_T\big(\phi(v),\phi(x)\big)-d_T\big(\phi(v),\phi(v_1)\big)\Big ]
\]
Considering that  $v_1$ has only  two neighbours, namely $v_2$ and $v_3$,  and denoting by $y$ the unique  neighbour 
(i.e. the father) of leaf $x$ in $G$ we get  
\[ \Delta= d_T\big(\phi(v_2)\phi(v_1)\big)-d_T\big(\phi(v_2)\phi(x)\big)+d_T\big(\phi(v_3)\phi(v_1)\big)-d_T\big(\phi(v_3)\phi(x)\big)+\]
\[d_T\big(\phi(y)\phi(x)\big)-d_T\big(\phi(y)\phi(v_1)\big) 
=2h_G-2+2(h_G+1)-2(h_G+1)+4 - 2h_G=2\, .\]
		\end{proof}
	\end{lemma}

	\begin{lemma}
		\label{lemma:upperBound}
Let the guest graph $G = (V, E)$ and the host graph $T$ be  binary regular trees  of heights $h_G$ and $h = h_G + 1$ 
respectively. Then the value  $OV(G, 2, \phi_A)$ of the objective  function of the DAPT 
correspondig to the arrangement   $\phi_A$  obtained from Algorithm~\ref{algorithm:arrangementPhiU} is given as follows:

		\begin{equation}
			\label{equation:lemma:upperBound}
				OV(G_{h_G}, 2, \phi_A) \defeq OV(G, 2, \phi_A) = \left\{ 
				\begin{array}{ll}
					0
						& \text{for } h_G = 0\\
					\frac{29}{3} \cdot 2^{h_G} - 4 h_G - 9 + \frac{1}{3} (-1)^{h_G}
						& \text{for } h_G \geq 1\\
				\end{array} \right.
				.
		\end{equation}
		
\begin{proof}
The proof is done by induction with respect to $h_G$. 
Let us denote $G_h$ a binary regular tree of height $h$ throughout this proof. 
Clearly we have $OV(G_0, 2, \phi_A) = 0$.
For $h_G = 1$ we obviously have $OV(G_1, 2, \phi_A) = 2 + 4 = 6$ by the construction 
(see Figures~\ref{figure:guestGraphGForHG1} and \ref{figure:arrangementPhiUForHG1} in Appendix). Both this equalities 
 are consistent with (\ref{equation:lemma:upperBound}). 
\smallskip
			
\noindent Assume that (\ref{equation:lemma:upperBound}) holds for some $h_G \geq 1$. For $h_G + 1$ we get
\begin{equation}
				\label{equation:equation:lemma:upperBound:proof}
				\begin{split}
					OV(G_{h_G + 1}, 2, \phi_A) = \left\{
						\begin{array}{l}
							2 OV(G_{h_G}, 2, \phi_A) + 2 (h_G + 1) + 2 (h_G + 2) - 2\\
							\qquad \text{for } h_G + 1 \text{ odd}\\
							2 OV(G_{h_G}, 2, \phi_A) + 2 (h_G + 1) + 2 (h_G + 2)\\
							\qquad \text{for } h_G + 1 \text{ even}\\
						\end{array} \right.
						,
				\end{split}
\end{equation}
			where:

\begin{itemize}
\item $2 OV(G_{h_G + 1}, 2, \phi_A)$ represents the objective function value corresponding to the arrangements of the basic subtrees.
\item $2 (h_G + 1)$ and $2 (h_G + 2)$ represent the contribution of   the edges connecting the root $v_1$ with its left and right 
child in the objective function value, respectively. 
(Prior to the pair-exchange step the root $v_1$ is arranged at the leaf $b_{\frac{1}{2} b}$ while  its children, 
$v_2$ and $v_3$ are arranged   at the leaves $v_{\frac{1}{4} b}$ and $b_{\frac{3}{4} b}$, respectively.)
\item  $-2$ represents the contribution of the pair-exchange step  if $h_G + 1$ is odd 
($h_G + 1 \geq 3$ since $h_G \geq 1$), according to Lemma~\ref{lemma:improvementForEveryPairExchange}.
\end{itemize}
\smallskip
									
According to the induction assumption we substitute $OV(G_{h_G},2,\phi_A)$ by the expression on the right hand side of 
equation~(\ref{equation:lemma:upperBound}) and after simplifying we get
		
$$
OV(G_{h_G + 1}, 2, \phi_A)=\left \{
\begin{array}{ll}
\frac{29}{3} \cdot 2^{h_G + 1} - 4 (h_G + 1) - 9 + \frac{1}{3} (-1)^{h_G + 1} & \mbox{if $h_G+1$ is odd}\\[0.2cm]
 \frac{29}{3} \cdot 2^{h_G + 1} - 4 (h_G + 1) - 9 + \frac{1}{3} (-1)^{h_G + 1} &  \mbox{if $h_G+1$ is even.}
\end{array} \right .
$$
 		\end{proof}
	\end{lemma}
	
	Finally, notice that this approximation algorithm ${\cal A}$ does not solve the problem to optimality as illustrated by the following example.
	\begin{example}
		\label{exmaple:thisApproximationAlgorithmDoesNotSolveTheProblemToOptimality}
		Consider a guest graph $G = (V, E)$ of height $h_G = 6$ depicted in 
Figure~\ref{figure:guestGraphGForHG6DepictingTheArrangementPhiA}. The arrangement $\phi_A$ computed by  
 Algorithm~\ref{algorithm:arrangementPhiU} is depicted in Figures~\ref{figure:arrangementPhiAForHG6FirstPart} and \ref{figure:arrangementPhiAForHG6SecondPart}; it yields  an  objective function value of $OV(G, 2, \phi_A) = 586$. Consider now another arrangement $\phi$ for the same graph yilding  an  objective function  value of $OV(G, 2, \phi) = 584$ and depicted explicitly in  Figures~\ref{figure:guestGraphGForHG6DepictingTheArrangementPhi}, \ref{figure:arrangementPhiForHG6FirstPart} and \ref{figure:arrangementPhiForHG6SecondPart}. 
Later in Section~\ref{section:approxratio}  we will show  that  the approximation algorithm 
${\cal A}$ yields an optimal arrangement $\phi_A$ for $h_G \leq 5$. Thus this is the smallest instance of the $DAPT(G,2)$
for which the algorithm ${\cal A}$ does not compute an optimal arrangement.

		\begin{figure}[p]
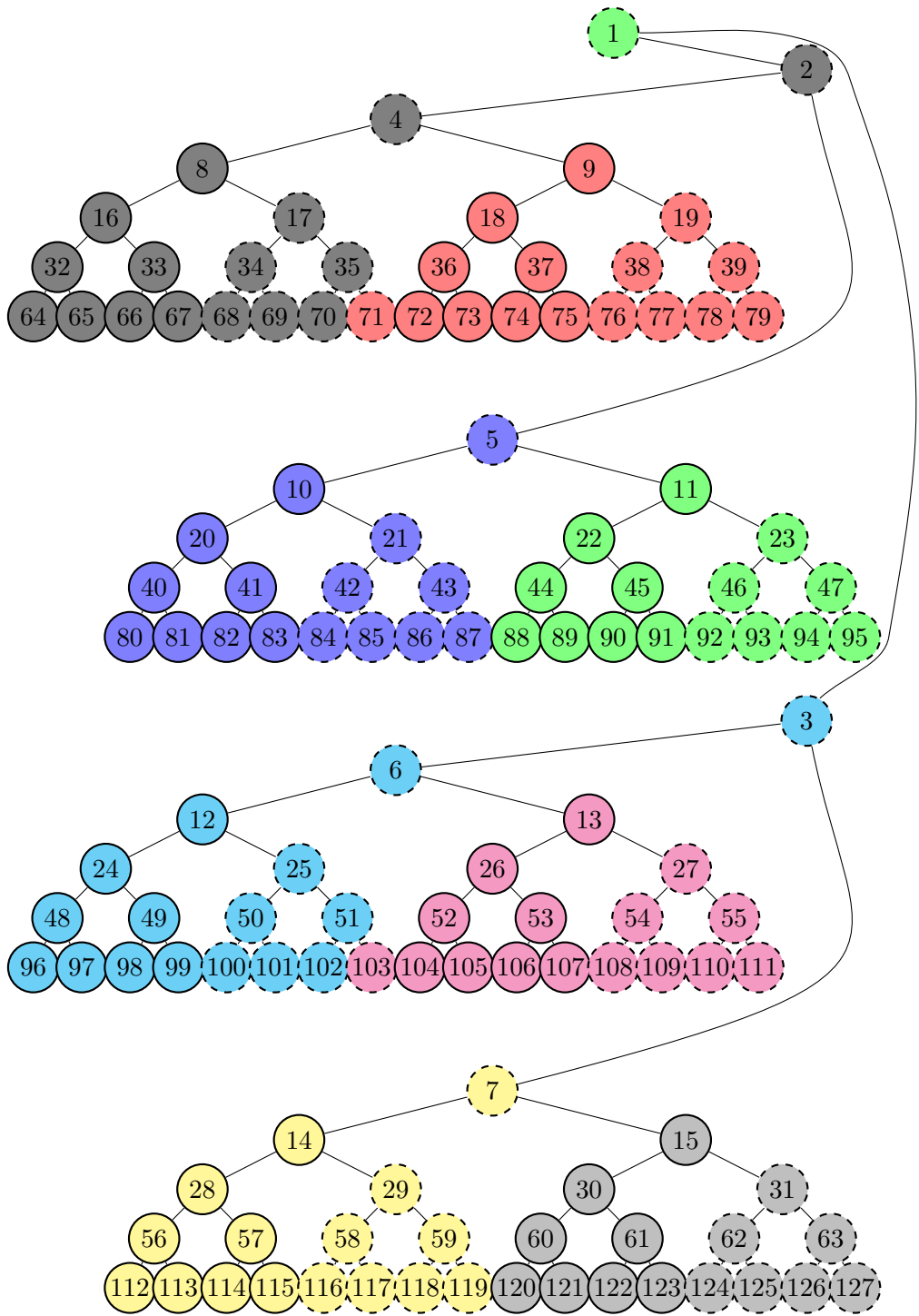

			\centering
			\begin{comment:upperBound}
				\begin{tikzpictureGuestGraph}
					\node[nodeGreenDashed] (node0) at (6.00, 5.75) {1};
					
					\node[nodeBlackDashed] (node1) at (8.00, 5.00) {2};
					
					\node[nodeBlackDashed] (node2) at (3.75, 4.00) {4};
		
					\node[nodeBlack] (node4) at (1.75, 3.00) {8};
					\node[nodeRed] (node5) at (5.75, 3.00) {9};
		
					\node[nodeBlack] (node8) at (0.75, 2.00) {16};
					\node[nodeBlackDashed] (node9) at (2.75, 2.00) {17};
					\node[nodeRed] (node10) at (4.75, 2.00) {18};
					\node[nodeRedDashed] (node11) at (6.75, 2.00) {19};
		
					\node[nodeBlack] (node16) at (0.25, 1.00) {32};
					\node[nodeBlack] (node17) at (1.25, 1.00) {33};
					\node[nodeBlackDashed] (node18) at (2.25, 1.00) {34};
					\node[nodeBlackDashed] (node19) at (3.25, 1.00) {35};
					\node[nodeRed] (node20) at (4.25, 1.00) {36};
					\node[nodeRed] (node21) at (5.25, 1.00) {37};
					\node[nodeRedDashed] (node22) at (6.25, 1.00) {38};
					\node[nodeRedDashed] (node23) at (7.25, 1.00) {39};
		
					\node[nodeBlack] (node32) at (0.00, 0.00) {64};
					\node[nodeBlack] (node33) at (0.50, 0.00) {65};
					\node[nodeBlack] (node34) at (1.00, 0.00) {66};
					\node[nodeBlack] (node35) at (1.50, 0.00) {67};
					\node[nodeBlackDashed] (node36) at (2.00, 0.00) {68};
					\node[nodeBlackDashed] (node37) at (2.50, 0.00) {69};
					\node[nodeBlackDashed] (node38) at (3.00, 0.00) {70};
					\node[nodeRedDashed] (node39) at (3.50, 0.00) {71};
					\node[nodeRed] (node40) at (4.00, 0.00) {72};
					\node[nodeRed] (node41) at (4.50, 0.00) {73};
					\node[nodeRed] (node42) at (5.00, 0.00) {74};
					\node[nodeRed] (node43) at (5.50, 0.00) {75};
					\node[nodeRedDashed] (node44) at (6.00, 0.00) {76};
					\node[nodeRedDashed] (node45) at (6.50, 0.00) {77};
					\node[nodeRedDashed] (node46) at (7.00, 0.00) {78};
					\node[nodeRedDashed] (node47) at (7.50, 0.00) {79};
		
					\node[nodeBlueDashed] (node3) at (4.75, -2.50) {5};
		
					\node[nodeBlue] (node6) at (2.75, -3.50) {10};
					\node[nodeGreen] (node7) at (6.75, -3.50) {11};
		
					\node[nodeBlue] (node12) at (1.75, -4.50) {20};
					\node[nodeBlueDashed] (node13) at (3.75, -4.50) {21};
					\node[nodeGreen] (node14) at (5.75, -4.50) {22};
					\node[nodeGreenDashed] (node15) at (7.75, -4.50) {23};
		
					\node[nodeBlue] (node24) at (1.25, -5.50) {40};
					\node[nodeBlue] (node25) at (2.25, -5.50) {41};
					\node[nodeBlueDashed] (node26) at (3.25, -5.50) {42};
					\node[nodeBlueDashed] (node27) at (4.25, -5.50) {43};
					\node[nodeGreen] (node28) at (5.25, -5.50) {44};
					\node[nodeGreen] (node29) at (6.25, -5.50) {45};
					\node[nodeGreenDashed] (node30) at (7.25, -5.50) {46};
					\node[nodeGreenDashed] (node31) at (8.25, -5.50) {47};
		
					\node[nodeBlue] (node48) at (1.00, -6.50) {80};
					\node[nodeBlue] (node49) at (1.50, -6.50) {81};
					\node[nodeBlue] (node50) at (2.00, -6.50) {82};
					\node[nodeBlue] (node51) at (2.50, -6.50) {83};
					\node[nodeBlueDashed] (node52) at (3.00, -6.50) {84};
					\node[nodeBlueDashed] (node53) at (3.50, -6.50) {85};
					\node[nodeBlueDashed] (node54) at (4.00, -6.50) {86};
					\node[nodeBlueDashed] (node55) at (4.50, -6.50) {87};
					\node[nodeGreen] (node56) at (5.00, -6.50) {88};
					\node[nodeGreen] (node57) at (5.50, -6.50) {89};
					\node[nodeGreen] (node58) at (6.00, -6.50) {90};
					\node[nodeGreen] (node59) at (6.50, -6.50) {91};
					\node[nodeGreenDashed] (node60) at (7.00, -6.50) {92};
					\node[nodeGreenDashed] (node61) at (7.50, -6.50) {93};
					\node[nodeGreenDashed] (node62) at (8.00, -6.50) {94};
					\node[nodeGreenDashed] (node63) at (8.50, -6.50) {95};
		
					\node[nodeCyanDashed] (node1a) at (8.00, 5.00 - 13.20) {3};
					
					\node[nodeCyanDashed] (node2a) at (3.75, 4.00 - 13.20) {6};
		
					\node[nodeCyan] (node4a) at (1.75, 3.00 - 13.20) {12};
					\node[nodeMagenta] (node5a) at (5.75, 3.00 - 13.20) {13};
		
					\node[nodeCyan] (node8a) at (0.75, 2.00 - 13.20) {24};
					\node[nodeCyanDashed] (node9a) at (2.75, 2.00 - 13.20) {25};
					\node[nodeMagenta] (node10a) at (4.75, 2.00 - 13.20) {26};
					\node[nodeMagentaDashed] (node11a) at (6.75, 2.00 - 13.20) {27};
		
					\node[nodeCyan] (node16a) at (0.25, 1.00 - 13.20) {48};
					\node[nodeCyan] (node17a) at (1.25, 1.00 - 13.20) {49};
					\node[nodeCyanDashed] (node18a) at (2.25, 1.00 - 13.20) {50};
					\node[nodeCyanDashed] (node19a) at (3.25, 1.00 - 13.20) {51};
					\node[nodeMagenta] (node20a) at (4.25, 1.00 - 13.20) {52};
					\node[nodeMagenta] (node21a) at (5.25, 1.00 - 13.20) {53};
					\node[nodeMagentaDashed] (node22a) at (6.25, 1.00 - 13.20) {54};
					\node[nodeMagentaDashed] (node23a) at (7.25, 1.00 - 13.20) {55};
		
					\node[nodeCyan] (node32a) at (0.00, 0.00 - 13.20) {96};
					\node[nodeCyan] (node33a) at (0.50, 0.00 - 13.20) {97};
					\node[nodeCyan] (node34a) at (1.00, 0.00 - 13.20) {98};
					\node[nodeCyan] (node35a) at (1.50, 0.00 - 13.20) {99};
					\node[nodeCyanDashed] (node36a) at (2.00, 0.00 - 13.20) {100};
					\node[nodeCyanDashed] (node37a) at (2.50, 0.00 - 13.20) {101};
					\node[nodeCyanDashed] (node38a) at (3.00, 0.00 - 13.20) {102};
					\node[nodeMagentaDashed] (node39a) at (3.50, 0.00 - 13.20) {103};
					\node[nodeMagenta] (node40a) at (4.00, 0.00 - 13.20) {104};
					\node[nodeMagenta] (node41a) at (4.50, 0.00 - 13.20) {105};
					\node[nodeMagenta] (node42a) at (5.00, 0.00 - 13.20) {106};
					\node[nodeMagenta] (node43a) at (5.50, 0.00 - 13.20) {107};
					\node[nodeMagentaDashed] (node44a) at (6.00, 0.00 - 13.20) {108};
					\node[nodeMagentaDashed] (node45a) at (6.50, 0.00 - 13.20) {109};
					\node[nodeMagentaDashed] (node46a) at (7.00, 0.00 - 13.20) {110};
					\node[nodeMagentaDashed] (node47a) at (7.50, 0.00 - 13.20) {111};
		
					\node[nodeYellowDashed] (node3a) at (4.75, -2.50 - 13.20) {7};
		
					\node[nodeYellow] (node6a) at (2.75, -3.50 - 13.20) {14};
					\node[nodeGrey] (node7a) at (6.75, -3.50 - 13.20) {15};
		
					\node[nodeYellow] (node12a) at (1.75, -4.50 - 13.20) {28};
					\node[nodeYellowDashed] (node13a) at (3.75, -4.50 - 13.20) {29};
					\node[nodeGrey] (node14a) at (5.75, -4.50 - 13.20) {30};
					\node[nodeGreyDashed] (node15a) at (7.75, -4.50 - 13.20) {31};
		
					\node[nodeYellow] (node24a) at (1.25, -5.50 - 13.20) {56};
					\node[nodeYellow] (node25a) at (2.25, -5.50 - 13.20) {57};
					\node[nodeYellowDashed] (node26a) at (3.25, -5.50 - 13.20) {58};
					\node[nodeYellowDashed] (node27a) at (4.25, -5.50 - 13.20) {59};
					\node[nodeGrey] (node28a) at (5.25, -5.50 - 13.20) {60};
					\node[nodeGrey] (node29a) at (6.25, -5.50 - 13.20) {61};
					\node[nodeGreyDashed] (node30a) at (7.25, -5.50 - 13.20) {62};
					\node[nodeGreyDashed] (node31a) at (8.25, -5.50 - 13.20) {63};
		
					\node[nodeYellow] (node48a) at (1.00, -6.50 - 13.20) {112};
					\node[nodeYellow] (node49a) at (1.50, -6.50 - 13.20) {113};
					\node[nodeYellow] (node50a) at (2.00, -6.50 - 13.20) {114};
					\node[nodeYellow] (node51a) at (2.50, -6.50 - 13.20) {115};
					\node[nodeYellowDashed] (node52a) at (3.00, -6.50 - 13.20) {116};
					\node[nodeYellowDashed] (node53a) at (3.50, -6.50 - 13.20) {117};
					\node[nodeYellowDashed] (node54a) at (4.00, -6.50 - 13.20) {118};
					\node[nodeYellowDashed] (node55a) at (4.50, -6.50 - 13.20) {119};
					\node[nodeGrey] (node56a) at (5.00, -6.50 - 13.20) {120};
					\node[nodeGrey] (node57a) at (5.50, -6.50 - 13.20) {121};
					\node[nodeGrey] (node58a) at (6.00, -6.50 - 13.20) {122};
					\node[nodeGrey] (node59a) at (6.50, -6.50 - 13.20) {123};
					\node[nodeGreyDashed] (node60a) at (7.00, -6.50 - 13.20) {124};
					\node[nodeGreyDashed] (node61a) at (7.50, -6.50 - 13.20) {125};
					\node[nodeGreyDashed] (node62a) at (8.00, -6.50 - 13.20) {126};
					\node[nodeGreyDashed] (node63a) at (8.50, -6.50 - 13.20) {127};
					
					\draw (node0) to (node1);
		
					\draw (node1) to (node2);
					\draw (node1) .. controls (8.50, 0.10) and (9.50, -0.10) .. (node3);
		
					\draw (node2) to (node4);
					\draw (node2) to (node5);
					\draw (node3) to (node6);
					\draw (node3) to (node7);
		
					\draw (node4) to (node8);
					\draw (node4) to (node9);
					\draw (node5) to (node10);
					\draw (node5) to (node11);
					\draw (node6) to (node12);
					\draw (node6) to (node13);
					\draw (node7) to (node14);
					\draw (node7) to (node15);
		
					\draw (node8) to (node16);
					\draw (node8) to (node17);
					\draw (node9) to (node18);
					\draw (node9) to (node19);
					\draw (node10) to (node20);
					\draw (node10) to (node21);
					\draw (node11) to (node22);
					\draw (node11) to (node23);
					\draw (node12) to (node24);
					\draw (node12) to (node25);
					\draw (node13) to (node26);
					\draw (node13) to (node27);
					\draw (node14) to (node28);
					\draw (node14) to (node29);
					\draw (node15) to (node30);
					\draw (node15) to (node31);
		
					\draw (node16) to (node32);
					\draw (node16) to (node33);
					\draw (node17) to (node34);
					\draw (node17) to (node35);
					\draw (node18) to (node36);
					\draw (node18) to (node37);
					\draw (node19) to (node38);
					\draw (node19) to (node39);
					\draw (node20) to (node40);
					\draw (node20) to (node41);
					\draw (node21) to (node42);
					\draw (node21) to (node43);
					\draw (node22) to (node44);
					\draw (node22) to (node45);
					\draw (node23) to (node46);
					\draw (node23) to (node47);
					\draw (node24) to (node48);
					\draw (node24) to (node49);
					\draw (node25) to (node50);
					\draw (node25) to (node51);
					\draw (node26) to (node52);
					\draw (node26) to (node53);
					\draw (node27) to (node54);
					\draw (node27) to (node55);
					\draw (node28) to (node56);
					\draw (node28) to (node57);
					\draw (node29) to (node58);
					\draw (node29) to (node59);
					\draw (node30) to (node60);
					\draw (node30) to (node61);
					\draw (node31) to (node62);
					\draw (node31) to (node63);
		
					\draw (node0) to [out=-0,in=150] (8.00, 5.60) to [out=-30,in=100] (8.50, 4.50) to [out=-80,in=95] (9.00, 1.00) to [out=-85,in=85] (8.85, -6.50) to [out=-95,in=60] (node1a);				
					\draw (node1a) to (node2a);
					\draw (node1a) .. controls (8.50, 0.10 - 13.20) and (9.50, -0.10 - 13.20) .. (node3a);
		
					\draw (node2a) to (node4a);
					\draw (node2a) to (node5a);
					\draw (node3a) to (node6a);
					\draw (node3a) to (node7a);
		
					\draw (node4a) to (node8a);
					\draw (node4a) to (node9a);
					\draw (node5a) to (node10a);
					\draw (node5a) to (node11a);
					\draw (node6a) to (node12a);
					\draw (node6a) to (node13a);
					\draw (node7a) to (node14a);
					\draw (node7a) to (node15a);
		
					\draw (node8a) to (node16a);
					\draw (node8a) to (node17a);
					\draw (node9a) to (node18a);
					\draw (node9a) to (node19a);
					\draw (node10a) to (node20a);
					\draw (node10a) to (node21a);
					\draw (node11a) to (node22a);
					\draw (node11a) to (node23a);
					\draw (node12a) to (node24a);
					\draw (node12a) to (node25a);
					\draw (node13a) to (node26a);
					\draw (node13a) to (node27a);
					\draw (node14a) to (node28a);
					\draw (node14a) to (node29a);
					\draw (node15a) to (node30a);
					\draw (node15a) to (node31a);
		
					\draw (node16a) to (node32a);
					\draw (node16a) to (node33a);
					\draw (node17a) to (node34a);
					\draw (node17a) to (node35a);
					\draw (node18a) to (node36a);
					\draw (node18a) to (node37a);
					\draw (node19a) to (node38a);
					\draw (node19a) to (node39a);
					\draw (node20a) to (node40a);
					\draw (node20a) to (node41a);
					\draw (node21a) to (node42a);
					\draw (node21a) to (node43a);
					\draw (node22a) to (node44a);
					\draw (node22a) to (node45a);
					\draw (node23a) to (node46a);
					\draw (node23a) to (node47a);
					\draw (node24a) to (node48a);
					\draw (node24a) to (node49a);
					\draw (node25a) to (node50a);
					\draw (node25a) to (node51a);
					\draw (node26a) to (node52a);
					\draw (node26a) to (node53a);
					\draw (node27a) to (node54a);
					\draw (node27a) to (node55a);
					\draw (node28a) to (node56a);
					\draw (node28a) to (node57a);
					\draw (node29a) to (node58a);
					\draw (node29a) to (node59a);
					\draw (node30a) to (node60a);
					\draw (node30a) to (node61a);
					\draw (node31a) to (node62a);
					\draw (node31a) to (node63a);
				\end{tikzpictureGuestGraph}
			\end{comment:upperBound}
			\vspace*{-0.27cm}
			\caption[Guest graph $G = (V, E)$ (binary regular tree of height $h_G = 6$). 
The colours are related to  the arrangement  the arrangement $\phi$ depicted in Figures~\ref{figure:arrangementPhiForHG6FirstPart} and \ref{figure:arrangementPhiForHG6SecondPart}.]{Guest graph $G = (V, E)$ (binary regular tree of height $h_G = 6$). The colours are related to  the arrangement $\phi_A$ depicted in Figures~\ref{figure:arrangementPhiAForHG6FirstPart} and \ref{figure:arrangementPhiAForHG6SecondPart}.}
			\label{figure:guestGraphGForHG6DepictingTheArrangementPhiA}
		\end{figure}
		\begin{figure}[p]
			\centering
			\begin{comment:upperBound}
				\begin{tikzpictureHostGraph}
					\node[otherNode] (node00) at (3.75, 6.00) {};
				
					\node[otherNode] (node0) at (6.00, 5.75) {};
					
					\node[otherNode] (node1) at (8.00, 5.00) {};
					
					\node[otherNode] (node2) at (3.75, 4.00) {};
		
					\node[otherNode] (node4) at (1.75, 3.00) {};
					\node[otherNode] (node5) at (5.75, 3.00) {};
		
					\node[otherNode] (node8) at (0.75, 2.00) {};
					\node[otherNode] (node9) at (2.75, 2.00) {};
					\node[otherNode] (node10) at (4.75, 2.00) {};
					\node[otherNode] (node11) at (6.75, 2.00) {};
		
					\node[otherNode] (node16) at (0.25, 1.00) {};
					\node[otherNode] (node17) at (1.25, 1.00) {};
					\node[otherNode] (node18) at (2.25, 1.00) {};
					\node[otherNode] (node19) at (3.25, 1.00) {};
					\node[otherNode] (node20) at (4.25, 1.00) {};
					\node[otherNode] (node21) at (5.25, 1.00) {};
					\node[otherNode] (node22) at (6.25, 1.00) {};
					\node[otherNode] (node23) at (7.25, 1.00) {};
		
					\node[leafUsedBlack] (node32) at (0.00, 0.00) {64};
					\node[leafUsedBlack] (node33) at (0.50, 0.00) {32};
					\node[leafUsedBlack] (node34) at (1.00, 0.00) {8};
					\node[leafUsedBlack] (node35) at (1.50, 0.00) {16};
					\node[leafUsedBlack] (node36) at (2.00, 0.00) {66};
					\node[leafUsedBlack] (node37) at (2.50, 0.00) {33};
					\node[leafUsedBlack] (node38) at (3.00, 0.00) {67};
					\node[leafUsedBlack] (node39) at (3.50, 0.00) {65};
					\node[leafUsedBlackDashed] (node40) at (4.00, 0.00) {68};
					\node[leafUsedBlackDashed] (node41) at (4.50, 0.00) {34};
					\node[leafUsedBlackDashed] (node42) at (5.00, 0.00) {69};
					\node[leafUsedBlackDashed] (node43) at (5.50, 0.00) {17};
					\node[leafUsedBlackDashed] (node44) at (6.00, 0.00) {70};
					\node[leafUsedBlackDashed] (node45) at (6.50, 0.00) {35};
					\node[leafUsedBlackDashed] (node46) at (7.00, 0.00) {2};
					\node[leafUsedBlackDashed] (node47) at (7.50, 0.00) {4};
		
					\node[otherNode] (node3) at (4.75, -2.50) {};
		
					\node[otherNode] (node6) at (2.75, -3.50) {};
					\node[otherNode] (node7) at (6.75, -3.50) {};
		
					\node[otherNode] (node12) at (1.75, -4.50) {};
					\node[otherNode] (node13) at (3.75, -4.50) {};
					\node[otherNode] (node14) at (5.75, -4.50) {};
					\node[otherNode] (node15) at (7.75, -4.50) {};
		
					\node[otherNode] (node24) at (1.25, -5.50) {};
					\node[otherNode] (node25) at (2.25, -5.50) {};
					\node[otherNode] (node26) at (3.25, -5.50) {};
					\node[otherNode] (node27) at (4.25, -5.50) {};
					\node[otherNode] (node28) at (5.25, -5.50) {};
					\node[otherNode] (node29) at (6.25, -5.50) {};
					\node[otherNode] (node30) at (7.25, -5.50) {};
					\node[otherNode] (node31) at (8.25, -5.50) {};
		
					\node[leafUsedRed] (node48) at (1.00, -6.50) {72};
					\node[leafUsedRed] (node49) at (1.50, -6.50) {36};
					\node[leafUsedRed] (node50) at (2.00, -6.50) {9};
					\node[leafUsedRed] (node51) at (2.50, -6.50) {18};
					\node[leafUsedRed] (node52) at (3.00, -6.50) {74};
					\node[leafUsedRed] (node53) at (3.50, -6.50) {37};
					\node[leafUsedRed] (node54) at (4.00, -6.50) {75};
					\node[leafUsedRed] (node55) at (4.50, -6.50) {73};
					\node[leafUsedRedDashed] (node56) at (5.00, -6.50) {76};
					\node[leafUsedRedDashed] (node57) at (5.50, -6.50) {38};
					\node[leafUsedRedDashed] (node58) at (6.00, -6.50) {77};
					\node[leafUsedRedDashed] (node59) at (6.50, -6.50) {19};
					\node[leafUsedRedDashed] (node60) at (7.00, -6.50) {78};
					\node[leafUsedRedDashed] (node61) at (7.50, -6.50) {39};
					\node[leafUsedRedDashed] (node62) at (8.00, -6.50) {79};
					\node[leafUsedRedDashed] (node63) at (8.50, -6.50) {71};
		
					\node[otherNode] (node1a) at (8.00, 5.00 - 13.20) {};
					
					\node[otherNode] (node2a) at (3.75, 4.00 - 13.20) {};
		
					\node[otherNode] (node4a) at (1.75, 3.00 - 13.20) {};
					\node[otherNode] (node5a) at (5.75, 3.00 - 13.20) {};
		
					\node[otherNode] (node8a) at (0.75, 2.00 - 13.20) {};
					\node[otherNode] (node9a) at (2.75, 2.00 - 13.20) {};
					\node[otherNode] (node10a) at (4.75, 2.00 - 13.20) {};
					\node[otherNode] (node11a) at (6.75, 2.00 - 13.20) {};
		
					\node[otherNode] (node16a) at (0.25, 1.00 - 13.20) {};
					\node[otherNode] (node17a) at (1.25, 1.00 - 13.20) {};
					\node[otherNode] (node18a) at (2.25, 1.00 - 13.20) {};
					\node[otherNode] (node19a) at (3.25, 1.00 - 13.20) {};
					\node[otherNode] (node20a) at (4.25, 1.00 - 13.20) {};
					\node[otherNode] (node21a) at (5.25, 1.00 - 13.20) {};
					\node[otherNode] (node22a) at (6.25, 1.00 - 13.20) {};
					\node[otherNode] (node23a) at (7.25, 1.00 - 13.20) {};
		
					\node[leafUsedBlue] (node32a) at (0.00, 0.00 - 13.20) {80};
					\node[leafUsedBlue] (node33a) at (0.50, 0.00 - 13.20) {40};
					\node[leafUsedBlue] (node34a) at (1.00, 0.00 - 13.20) {10};
					\node[leafUsedBlue] (node35a) at (1.50, 0.00 - 13.20) {20};
					\node[leafUsedBlue] (node36a) at (2.00, 0.00 - 13.20) {82};
					\node[leafUsedBlue] (node37a) at (2.50, 0.00 - 13.20) {41};
					\node[leafUsedBlue] (node38a) at (3.00, 0.00 - 13.20) {83};
					\node[leafUsedBlue] (node39a) at (3.50, 0.00 - 13.20) {81};
					\node[leafUsedBlueDashed] (node40a) at (4.00, 0.00 - 13.20) {84};
					\node[leafUsedBlueDashed] (node41a) at (4.50, 0.00 - 13.20) {42};
					\node[leafUsedBlueDashed] (node42a) at (5.00, 0.00 - 13.20) {85};
					\node[leafUsedBlueDashed] (node43a) at (5.50, 0.00 - 13.20) {21};
					\node[leafUsedBlueDashed] (node44a) at (6.00, 0.00 - 13.20) {86};
					\node[leafUsedBlueDashed] (node45a) at (6.50, 0.00 - 13.20) {43};
					\node[leafUsedBlueDashed] (node46a) at (7.00, 0.00 - 13.20) {87};
					\node[leafUsedBlueDashed] (node47a) at (7.50, 0.00 - 13.20) {5};
		
					\node[otherNode] (node3a) at (4.75, -2.50 - 13.20) {};
		
					\node[otherNode] (node6a) at (2.75, -3.50 - 13.20) {};
					\node[otherNode] (node7a) at (6.75, -3.50 - 13.20) {};
		
					\node[otherNode] (node12a) at (1.75, -4.50 - 13.20) {};
					\node[otherNode] (node13a) at (3.75, -4.50 - 13.20) {};
					\node[otherNode] (node14a) at (5.75, -4.50 - 13.20) {};
					\node[otherNode] (node15a) at (7.75, -4.50 - 13.20) {};
		
					\node[otherNode] (node24a) at (1.25, -5.50 - 13.20) {};
					\node[otherNode] (node25a) at (2.25, -5.50 - 13.20) {};
					\node[otherNode] (node26a) at (3.25, -5.50 - 13.20) {};
					\node[otherNode] (node27a) at (4.25, -5.50 - 13.20) {};
					\node[otherNode] (node28a) at (5.25, -5.50 - 13.20) {};
					\node[otherNode] (node29a) at (6.25, -5.50 - 13.20) {};
					\node[otherNode] (node30a) at (7.25, -5.50 - 13.20) {};
					\node[otherNode] (node31a) at (8.25, -5.50 - 13.20) {};
		
					\node[leafUsedGreen] (node48a) at (1.00, -6.50 - 13.20) {88};
					\node[leafUsedGreen] (node49a) at (1.50, -6.50 - 13.20) {44};
					\node[leafUsedGreen] (node50a) at (2.00, -6.50 - 13.20) {11};
					\node[leafUsedGreen] (node51a) at (2.50, -6.50 - 13.20) {22};
					\node[leafUsedGreen] (node52a) at (3.00, -6.50 - 13.20) {90};
					\node[leafUsedGreen] (node53a) at (3.50, -6.50 - 13.20) {45};
					\node[leafUsedGreen] (node54a) at (4.00, -6.50 - 13.20) {91};
					\node[leafUsedGreen] (node55a) at (4.50, -6.50 - 13.20) {89};
					\node[leafUsedGreenDashed] (node56a) at (5.00, -6.50 - 13.20) {92};
					\node[leafUsedGreenDashed] (node57a) at (5.50, -6.50 - 13.20) {46};
					\node[leafUsedGreenDashed] (node58a) at (6.00, -6.50 - 13.20) {93};
					\node[leafUsedGreenDashed] (node59a) at (6.50, -6.50 - 13.20) {23};
					\node[leafUsedGreenDashed] (node60a) at (7.00, -6.50 - 13.20) {94};
					\node[leafUsedGreenDashed] (node61a) at (7.50, -6.50 - 13.20) {47};
					\node[leafUsedGreenDashed] (node62a) at (8.00, -6.50 - 13.20) {95};
					\node[leafUsedGreenDashed] (node63a) at (8.50, -6.50 - 13.20) {1};
					
					\draw [-, black!50] (node00) to [out=30,in=180] (4.875, 6.50) to [out=0,in=180] (7.00, 6.50);
					\draw [dashed, black!50] (7.00, 6.50) to (9.52, 6.50);
					\draw [-, black!50] (node00) to (node0);
					
					\draw [-, black!50] (node0) to (node1);
		
					\draw [-, black!50] (node1) to (node2);
					\draw [-, black!50] (node1) .. controls (8.50, 0.10) and (9.50, -0.10) .. (node3);
		
					\draw [-, black!50] (node2) to (node4);
					\draw [-, black!50] (node2) to (node5);
					\draw [-, black!50] (node3) to (node6);
					\draw [-, black!50] (node3) to (node7);
		
					\draw [-, black!50] (node4) to (node8);
					\draw [-, black!50] (node4) to (node9);
					\draw [-, black!50] (node5) to (node10);
					\draw [-, black!50] (node5) to (node11);
					\draw [-, black!50] (node6) to (node12);
					\draw [-, black!50] (node6) to (node13);
					\draw [-, black!50] (node7) to (node14);
					\draw [-, black!50] (node7) to (node15);
		
					\draw [-, black!50] (node8) to (node16);
					\draw [-, black!50] (node8) to (node17);
					\draw [-, black!50] (node9) to (node18);
					\draw [-, black!50] (node9) to (node19);
					\draw [-, black!50] (node10) to (node20);
					\draw [-, black!50] (node10) to (node21);
					\draw [-, black!50] (node11) to (node22);
					\draw [-, black!50] (node11) to (node23);
					\draw [-, black!50] (node12) to (node24);
					\draw [-, black!50] (node12) to (node25);
					\draw [-, black!50] (node13) to (node26);
					\draw [-, black!50] (node13) to (node27);
					\draw [-, black!50] (node14) to (node28);
					\draw [-, black!50] (node14) to (node29);
					\draw [-, black!50] (node15) to (node30);
					\draw [-, black!50] (node15) to (node31);
		
					\draw [-, black!50] (node16) to (node32);
					\draw [-, black!50] (node16) to (node33);
					\draw [-, black!50] (node17) to (node34);
					\draw [-, black!50] (node17) to (node35);
					\draw [-, black!50] (node18) to (node36);
					\draw [-, black!50] (node18) to (node37);
					\draw [-, black!50] (node19) to (node38);
					\draw [-, black!50] (node19) to (node39);
					\draw [-, black!50] (node20) to (node40);
					\draw [-, black!50] (node20) to (node41);
					\draw [-, black!50] (node21) to (node42);
					\draw [-, black!50] (node21) to (node43);
					\draw [-, black!50] (node22) to (node44);
					\draw [-, black!50] (node22) to (node45);
					\draw [-, black!50] (node23) to (node46);
					\draw [-, black!50] (node23) to (node47);
					\draw [-, black!50] (node24) to (node48);
					\draw [-, black!50] (node24) to (node49);
					\draw [-, black!50] (node25) to (node50);
					\draw [-, black!50] (node25) to (node51);
					\draw [-, black!50] (node26) to (node52);
					\draw [-, black!50] (node26) to (node53);
					\draw [-, black!50] (node27) to (node54);
					\draw [-, black!50] (node27) to (node55);
					\draw [-, black!50] (node28) to (node56);
					\draw [-, black!50] (node28) to (node57);
					\draw [-, black!50] (node29) to (node58);
					\draw [-, black!50] (node29) to (node59);
					\draw [-, black!50] (node30) to (node60);
					\draw [-, black!50] (node30) to (node61);
					\draw [-, black!50] (node31) to (node62);
					\draw [-, black!50] (node31) to (node63);
		
					\draw [-, black!50] (node0) to [out=-0,in=150] (8.00, 5.60) to [out=-30,in=100] (8.50, 4.50) to [out=-80,in=95] (9.00, 1.00) to [out=-85,in=85] (8.85, -6.50) to [out=-95,in=60] (node1a);				
					\draw [-, black!50] (node1a) to (node2a);
					\draw [-, black!50] (node1a) .. controls (8.50, 0.10 - 13.20) and (9.50, -0.10 - 13.20) .. (node3a);
		
					\draw [-, black!50] (node2a) to (node4a);
					\draw [-, black!50] (node2a) to (node5a);
					\draw [-, black!50] (node3a) to (node6a);
					\draw [-, black!50] (node3a) to (node7a);
		
					\draw [-, black!50] (node4a) to (node8a);
					\draw [-, black!50] (node4a) to (node9a);
					\draw [-, black!50] (node5a) to (node10a);
					\draw [-, black!50] (node5a) to (node11a);
					\draw [-, black!50] (node6a) to (node12a);
					\draw [-, black!50] (node6a) to (node13a);
					\draw [-, black!50] (node7a) to (node14a);
					\draw [-, black!50] (node7a) to (node15a);
		
					\draw [-, black!50] (node8a) to (node16a);
					\draw [-, black!50] (node8a) to (node17a);
					\draw [-, black!50] (node9a) to (node18a);
					\draw [-, black!50] (node9a) to (node19a);
					\draw [-, black!50] (node10a) to (node20a);
					\draw [-, black!50] (node10a) to (node21a);
					\draw [-, black!50] (node11a) to (node22a);
					\draw [-, black!50] (node11a) to (node23a);
					\draw [-, black!50] (node12a) to (node24a);
					\draw [-, black!50] (node12a) to (node25a);
					\draw [-, black!50] (node13a) to (node26a);
					\draw [-, black!50] (node13a) to (node27a);
					\draw [-, black!50] (node14a) to (node28a);
					\draw [-, black!50] (node14a) to (node29a);
					\draw [-, black!50] (node15a) to (node30a);
					\draw [-, black!50] (node15a) to (node31a);
		
					\draw [-, black!50] (node16a) to (node32a);
					\draw [-, black!50] (node16a) to (node33a);
					\draw [-, black!50] (node17a) to (node34a);
					\draw [-, black!50] (node17a) to (node35a);
					\draw [-, black!50] (node18a) to (node36a);
					\draw [-, black!50] (node18a) to (node37a);
					\draw [-, black!50] (node19a) to (node38a);
					\draw [-, black!50] (node19a) to (node39a);
					\draw [-, black!50] (node20a) to (node40a);
					\draw [-, black!50] (node20a) to (node41a);
					\draw [-, black!50] (node21a) to (node42a);
					\draw [-, black!50] (node21a) to (node43a);
					\draw [-, black!50] (node22a) to (node44a);
					\draw [-, black!50] (node22a) to (node45a);
					\draw [-, black!50] (node23a) to (node46a);
					\draw [-, black!50] (node23a) to (node47a);
					\draw [-, black!50] (node24a) to (node48a);
					\draw [-, black!50] (node24a) to (node49a);
					\draw [-, black!50] (node25a) to (node50a);
					\draw [-, black!50] (node25a) to (node51a);
					\draw [-, black!50] (node26a) to (node52a);
					\draw [-, black!50] (node26a) to (node53a);
					\draw [-, black!50] (node27a) to (node54a);
					\draw [-, black!50] (node27a) to (node55a);
					\draw [-, black!50] (node28a) to (node56a);
					\draw [-, black!50] (node28a) to (node57a);
					\draw [-, black!50] (node29a) to (node58a);
					\draw [-, black!50] (node29a) to (node59a);
					\draw [-, black!50] (node30a) to (node60a);
					\draw [-, black!50] (node30a) to (node61a);
					\draw [-, black!50] (node31a) to (node62a);
					\draw [-, black!50] (node31a) to (node63a);

					\draw [-] (node63a) to [out=90,in=-90] (9.00, -6.50) to [out=90,in=-90] (node46);  
					\draw [-] (node63a) to [out=90,in=-90] (9.125, -6.50);  
					\draw [dashed] (9.125, -6.50) to [out=90,in=180] (9.52, -2.00);  
					
					\draw [-] (node46) to [out=-90,in=-90] (node47);  
					\draw [-] (node46) to [out=-90,in=90] (8.875, -6.50) to [out=-90,in=90] (node47a);  

					\draw [-] (node47) to [out=-90,in=0] ([shift={(-0.50, -1.00)}]node47) to [out=180,in=0] ([shift={(+0.50, -1.00)}]node34) to [out=180,in=-90] (node34);  
					\draw [-] (node47) to [out=-90,in=0] (7.00,-1.75) to [out=180,in=0] (2.50, -1.75) to [out=180,in=90] (2.00, -2.75) to [out=-90,in=90] (node50);  
					\draw [-] (node47a) to [out=-90,in=0] ([shift={(-0.50, -1.00)}]node47a) to [out=180,in=0] ([shift={(+0.50, -1.00)}]node34a) to [out=180,in=-90] (node34a);  
					\draw [-] (node47a) to [out=-90,in=0] (7.00,-14.95) to [out=180,in=0] (2.50, -14.95) to [out=180,in=90] (2.00, -15.95) to [out=-90,in=90] (node50a);  
					
					\draw [-] (node34) to [out=-90,in=-90] (node35);  
					\draw [-] (node34) to [out=-90,in=180] ([shift={(+0.50, -1.25)}]node34) to [out=0,in=180] ([shift={(-0.50, -1.25)}]node43) to [out=0,in=-90] (node43);  
					\draw [-] (node50) to [out=-90,in=-90] (node51);  
					\draw [-] (node50) to [out=-90,in=180] ([shift={(+0.50, -1.00)}]node50) to [out=0,in=180] ([shift={(-0.50, -1.00)}]node59) to [out=0,in=-90] (node59);  
					\draw [-] (node34a) to [out=-90,in=-90] (node35a);  
					\draw [-] (node34a) to [out=-90,in=180] ([shift={(+0.50, -1.25)}]node34a) to [out=0,in=180] ([shift={(-0.50, -1.25)}]node43a) to [out=0,in=-90] (node43a);  
					\draw [-] (node50a) to [out=-90,in=-90] (node51a);  
					\draw [-] (node50a) to [out=-90,in=180] ([shift={(+0.50, -1.00)}]node50a) to [out=0,in=180] ([shift={(-0.50, -1.00)}]node59a) to [out=0,in=-90] (node59a);  

					\draw [-] (node35) to [out=-90,in=-90] (node33);  
					\draw [-] (node35) to [out=-90,in=-90] (node37);  
					\draw [-] (node43) to [out=-90,in=-90] (node41);  
					\draw [-] (node43) to [out=-90,in=-90] (node45);  
					\draw [-] (node51) to [out=-90,in=-90] (node49);  
					\draw [-] (node51) to [out=-90,in=-90] (node53);  
					\draw [-] (node59) to [out=-90,in=-90] (node57);  
					\draw [-] (node59) to [out=-90,in=-90] (node61);  
					\draw [-] (node35a) to [out=-90,in=-90] (node33a);  
					\draw [-] (node35a) to [out=-90,in=-90] (node37a);  
					\draw [-] (node43a) to [out=-90,in=-90] (node41a);  
					\draw [-] (node43a) to [out=-90,in=-90] (node45a);  
					\draw [-] (node51a) to [out=-90,in=-90] (node49a);  
					\draw [-] (node51a) to [out=-90,in=-90] (node53a);  
					\draw [-] (node59a) to [out=-90,in=-90] (node57a);  
					\draw [-] (node59a) to [out=-90,in=-90] (node61a);  
					
					\draw [-] (node33) to [out=-90,in=-90] (node32);  
					\draw [-] (node33) to [out=-90,in=180] ([shift={(+0.50, -1.50)}]node33) to [out=0,in=180] ([shift={(-0.50, -1.50)}]node39) to [out=0,in=-90] (node39);  
					\draw [-] (node37) to [out=-90,in=-90] (node36);  
					\draw [-] (node37) to [out=-90,in=-90] (node38);  
					\draw [-] (node41) to [out=-90,in=-90] (node40);  
					\draw [-] (node41) to [out=-90,in=-90] (node42);  
					\draw [-] (node45) to [out=-90,in=-90] (node44);  
					\draw [-] (node45) to [out=-90,in=90] (node63);  
					\draw [-] (node49) to [out=-90,in=-90] (node48);  
					\draw [-] (node49) to [out=-90,in=180] ([shift={(+0.50, -1.25)}]node49) to [out=0,in=180] ([shift={(-0.50, -1.25)}]node55) to [out=0,in=-90] (node55);  
					\draw [-] (node53) to [out=-90,in=-90] (node52);  
					\draw [-] (node53) to [out=-90,in=-90] (node54);  
					\draw [-] (node57) to [out=-90,in=-90] (node56);  
					\draw [-] (node57) to [out=-90,in=-90] (node58);  
					\draw [-] (node61) to [out=-90,in=-90] (node60);  
					\draw [-] (node61) to [out=-90,in=-90] (node62);  
					\draw [-] (node33a) to [out=-90,in=-90] (node32a);  
					\draw [-] (node33a) to [out=-90,in=180] ([shift={(+0.50, -1.50)}]node33a) to [out=0,in=180] ([shift={(-0.50, -1.50)}]node39a) to [out=0,in=-90] (node39a);  
					\draw [-] (node37a) to [out=-90,in=-90] (node36a);  
					\draw [-] (node37a) to [out=-90,in=-90] (node38a);  
					\draw [-] (node41a) to [out=-90,in=-90] (node40a);  
					\draw [-] (node41a) to [out=-90,in=-90] (node42a);  
					\draw [-] (node45a) to [out=-90,in=-90] (node44a);  
					\draw [-] (node45a) to [out=-90,in=-90] (node46a);  
					\draw [-] (node49a) to [out=-90,in=-90] (node48a);  
					\draw [-] (node49a) to [out=-90,in=180] ([shift={(+0.50, -1.25)}]node49a) to [out=0,in=180] ([shift={(-0.50, -1.25)}]node55a) to [out=0,in=-90] (node55a);  
					\draw [-] (node53a) to [out=-90,in=-90] (node52a);  
					\draw [-] (node53a) to [out=-90,in=-90] (node54a);  
					\draw [-] (node57a) to [out=-90,in=-90] (node56a);  
					\draw [-] (node57a) to [out=-90,in=-90] (node58a);  
					\draw [-] (node61a) to [out=-90,in=-90] (node60a);  
					\draw [-] (node61a) to [out=-90,in=-90] (node62a);  
				\end{tikzpictureHostGraph}
			\end{comment:upperBound}
			\vspace*{-0.27cm}
			\caption[Arrangement $\phi_A$ obtained from Algorithm~\ref{algorithm:arrangementPhiU} for the guest graph $G = (V, E)$ depicted in Figure~\ref{figure:guestGraphGForHG6DepictingTheArrangementPhiA} (binary regular tree of height $h_G = 6$) -- first part.]{Arrangement $\phi_A$ obtained from Algorithm~\ref{algorithm:arrangementPhiU} for the guest graph $G = (V, E)$ depicted in Figure~\ref{figure:guestGraphGForHG6DepictingTheArrangementPhiA} -- first part. Its objective function  value is $OV(G, 2, \phi_A) = 586$.}
			\label{figure:arrangementPhiAForHG6FirstPart}
		\end{figure}
		\begin{figure}[p]
			\centering
			\begin{comment:upperBound}
				\begin{tikzpictureHostGraph}
					\node[otherNode] (node0) at (6.00, 5.75) {};
					
					\node[otherNode] (node1) at (8.00, 5.00) {};
					
					\node[otherNode] (node2) at (3.75, 4.00) {};
		
					\node[otherNode] (node4) at (1.75, 3.00) {};
					\node[otherNode] (node5) at (5.75, 3.00) {};
		
					\node[otherNode] (node8) at (0.75, 2.00) {};
					\node[otherNode] (node9) at (2.75, 2.00) {};
					\node[otherNode] (node10) at (4.75, 2.00) {};
					\node[otherNode] (node11) at (6.75, 2.00) {};
		
					\node[otherNode] (node16) at (0.25, 1.00) {};
					\node[otherNode] (node17) at (1.25, 1.00) {};
					\node[otherNode] (node18) at (2.25, 1.00) {};
					\node[otherNode] (node19) at (3.25, 1.00) {};
					\node[otherNode] (node20) at (4.25, 1.00) {};
					\node[otherNode] (node21) at (5.25, 1.00) {};
					\node[otherNode] (node22) at (6.25, 1.00) {};
					\node[otherNode] (node23) at (7.25, 1.00) {};
		
					\node[leafUsedCyan] (node32) at (0.00, 0.00) {96};
					\node[leafUsedCyan] (node33) at (0.50, 0.00) {48};
					\node[leafUsedCyan] (node34) at (1.00, 0.00) {12};
					\node[leafUsedCyan] (node35) at (1.50, 0.00) {24};
					\node[leafUsedCyan] (node36) at (2.00, 0.00) {98};
					\node[leafUsedCyan] (node37) at (2.50, 0.00) {49};
					\node[leafUsedCyan] (node38) at (3.00, 0.00) {99};
					\node[leafUsedCyan] (node39) at (3.50, 0.00) {97};
					\node[leafUsedCyanDashed] (node40) at (4.00, 0.00) {100};
					\node[leafUsedCyanDashed] (node41) at (4.50, 0.00) {50};
					\node[leafUsedCyanDashed] (node42) at (5.00, 0.00) {101};
					\node[leafUsedCyanDashed] (node43) at (5.50, 0.00) {25};
					\node[leafUsedCyanDashed] (node44) at (6.00, 0.00) {102};
					\node[leafUsedCyanDashed] (node45) at (6.50, 0.00) {51};
					\node[leafUsedCyanDashed] (node46) at (7.00, 0.00) {3};
					\node[leafUsedCyanDashed] (node47) at (7.50, 0.00) {6};
		
					\node[otherNode] (node3) at (4.75, -2.50) {};
		
					\node[otherNode] (node6) at (2.75, -3.50) {};
					\node[otherNode] (node7) at (6.75, -3.50) {};
		
					\node[otherNode] (node12) at (1.75, -4.50) {};
					\node[otherNode] (node13) at (3.75, -4.50) {};
					\node[otherNode] (node14) at (5.75, -4.50) {};
					\node[otherNode] (node15) at (7.75, -4.50) {};
		
					\node[otherNode] (node24) at (1.25, -5.50) {};
					\node[otherNode] (node25) at (2.25, -5.50) {};
					\node[otherNode] (node26) at (3.25, -5.50) {};
					\node[otherNode] (node27) at (4.25, -5.50) {};
					\node[otherNode] (node28) at (5.25, -5.50) {};
					\node[otherNode] (node29) at (6.25, -5.50) {};
					\node[otherNode] (node30) at (7.25, -5.50) {};
					\node[otherNode] (node31) at (8.25, -5.50) {};
		
					\node[leafUsedMagenta] (node48) at (1.00, -6.50) {104};
					\node[leafUsedMagenta] (node49) at (1.50, -6.50) {52};
					\node[leafUsedMagenta] (node50) at (2.00, -6.50) {13};
					\node[leafUsedMagenta] (node51) at (2.50, -6.50) {26};
					\node[leafUsedMagenta] (node52) at (3.00, -6.50) {106};
					\node[leafUsedMagenta] (node53) at (3.50, -6.50) {53};
					\node[leafUsedMagenta] (node54) at (4.00, -6.50) {107};
					\node[leafUsedMagenta] (node55) at (4.50, -6.50) {105};
					\node[leafUsedMagentaDashed] (node56) at (5.00, -6.50) {108};
					\node[leafUsedMagentaDashed] (node57) at (5.50, -6.50) {54};
					\node[leafUsedMagentaDashed] (node58) at (6.00, -6.50) {109};
					\node[leafUsedMagentaDashed] (node59) at (6.50, -6.50) {27};
					\node[leafUsedMagentaDashed] (node60) at (7.00, -6.50) {110};
					\node[leafUsedMagentaDashed] (node61) at (7.50, -6.50) {55};
					\node[leafUsedMagentaDashed] (node62) at (8.00, -6.50) {111};
					\node[leafUsedMagentaDashed] (node63) at (8.50, -6.50) {103};
		
					\node[otherNode] (node1a) at (8.00, 5.00 - 13.20) {};
					
					\node[otherNode] (node2a) at (3.75, 4.00 - 13.20) {};
		
					\node[otherNode] (node4a) at (1.75, 3.00 - 13.20) {};
					\node[otherNode] (node5a) at (5.75, 3.00 - 13.20) {};
		
					\node[otherNode] (node8a) at (0.75, 2.00 - 13.20) {};
					\node[otherNode] (node9a) at (2.75, 2.00 - 13.20) {};
					\node[otherNode] (node10a) at (4.75, 2.00 - 13.20) {};
					\node[otherNode] (node11a) at (6.75, 2.00 - 13.20) {};
		
					\node[otherNode] (node16a) at (0.25, 1.00 - 13.20) {};
					\node[otherNode] (node17a) at (1.25, 1.00 - 13.20) {};
					\node[otherNode] (node18a) at (2.25, 1.00 - 13.20) {};
					\node[otherNode] (node19a) at (3.25, 1.00 - 13.20) {};
					\node[otherNode] (node20a) at (4.25, 1.00 - 13.20) {};
					\node[otherNode] (node21a) at (5.25, 1.00 - 13.20) {};
					\node[otherNode] (node22a) at (6.25, 1.00 - 13.20) {};
					\node[otherNode] (node23a) at (7.25, 1.00 - 13.20) {};
		
					\node[leafUsedYellow] (node32a) at (0.00, 0.00 - 13.20) {112};
					\node[leafUsedYellow] (node33a) at (0.50, 0.00 - 13.20) {56};
					\node[leafUsedYellow] (node34a) at (1.00, 0.00 - 13.20) {14};
					\node[leafUsedYellow] (node35a) at (1.50, 0.00 - 13.20) {28};
					\node[leafUsedYellow] (node36a) at (2.00, 0.00 - 13.20) {114};
					\node[leafUsedYellow] (node37a) at (2.50, 0.00 - 13.20) {57};
					\node[leafUsedYellow] (node38a) at (3.00, 0.00 - 13.20) {115};
					\node[leafUsedYellow] (node39a) at (3.50, 0.00 - 13.20) {113};
					\node[leafUsedYellowDashed] (node40a) at (4.00, 0.00 - 13.20) {116};
					\node[leafUsedYellowDashed] (node41a) at (4.50, 0.00 - 13.20) {58};
					\node[leafUsedYellowDashed] (node42a) at (5.00, 0.00 - 13.20) {117};
					\node[leafUsedYellowDashed] (node43a) at (5.50, 0.00 - 13.20) {29};
					\node[leafUsedYellowDashed] (node44a) at (6.00, 0.00 - 13.20) {118};
					\node[leafUsedYellowDashed] (node45a) at (6.50, 0.00 - 13.20) {59};
					\node[leafUsedYellowDashed] (node46a) at (7.00, 0.00 - 13.20) {119};
					\node[leafUsedYellowDashed] (node47a) at (7.50, 0.00 - 13.20) {7};
		
					\node[otherNode] (node3a) at (4.75, -2.50 - 13.20) {};
		
					\node[otherNode] (node6a) at (2.75, -3.50 - 13.20) {};
					\node[otherNode] (node7a) at (6.75, -3.50 - 13.20) {};
		
					\node[otherNode] (node12a) at (1.75, -4.50 - 13.20) {};
					\node[otherNode] (node13a) at (3.75, -4.50 - 13.20) {};
					\node[otherNode] (node14a) at (5.75, -4.50 - 13.20) {};
					\node[otherNode] (node15a) at (7.75, -4.50 - 13.20) {};
		
					\node[otherNode] (node24a) at (1.25, -5.50 - 13.20) {};
					\node[otherNode] (node25a) at (2.25, -5.50 - 13.20) {};
					\node[otherNode] (node26a) at (3.25, -5.50 - 13.20) {};
					\node[otherNode] (node27a) at (4.25, -5.50 - 13.20) {};
					\node[otherNode] (node28a) at (5.25, -5.50 - 13.20) {};
					\node[otherNode] (node29a) at (6.25, -5.50 - 13.20) {};
					\node[otherNode] (node30a) at (7.25, -5.50 - 13.20) {};
					\node[otherNode] (node31a) at (8.25, -5.50 - 13.20) {};
		
					\node[leafUsedGrey] (node48a) at (1.00, -6.50 - 13.20) {120};
					\node[leafUsedGrey] (node49a) at (1.50, -6.50 - 13.20) {60};
					\node[leafUsedGrey] (node50a) at (2.00, -6.50 - 13.20) {15};
					\node[leafUsedGrey] (node51a) at (2.50, -6.50 - 13.20) {30};
					\node[leafUsedGrey] (node52a) at (3.00, -6.50 - 13.20) {122};
					\node[leafUsedGrey] (node53a) at (3.50, -6.50 - 13.20) {61};
					\node[leafUsedGrey] (node54a) at (4.00, -6.50 - 13.20) {123};
					\node[leafUsedGrey] (node55a) at (4.50, -6.50 - 13.20) {121};
					\node[leafUsedGreyDashed] (node56a) at (5.00, -6.50 - 13.20) {124};
					\node[leafUsedGreyDashed] (node57a) at (5.50, -6.50 - 13.20) {62};
					\node[leafUsedGreyDashed] (node58a) at (6.00, -6.50 - 13.20) {125};
					\node[leafUsedGreyDashed] (node59a) at (6.50, -6.50 - 13.20) {31};
					\node[leafUsedGreyDashed] (node60a) at (7.00, -6.50 - 13.20) {126};
					\node[leafUsedGreyDashed] (node61a) at (7.50, -6.50 - 13.20) {63};
					\node[leafUsedGreyDashed] (node62a) at (8.00, -6.50 - 13.20) {127};
					\node[leafUnused] (node63a) at (8.50, -6.50 - 13.20) {};
					
					\draw [-, black!50] (node0) to [out=150,in=0] (4.875, 6.50) to [out=180,in=0] (2.23, 6.50);
					\draw [dashed, black!50] (2.23, 6.50) to (-0.29, 6.50);
					
					\draw [-, black!50] (node0) to (node1);
		
					\draw [-, black!50] (node1) to (node2);
					\draw [-, black!50] (node1) .. controls (8.50, 0.10) and (9.50, -0.10) .. (node3);
		
					\draw [-, black!50] (node2) to (node4);
					\draw [-, black!50] (node2) to (node5);
					\draw [-, black!50] (node3) to (node6);
					\draw [-, black!50] (node3) to (node7);
		
					\draw [-, black!50] (node4) to (node8);
					\draw [-, black!50] (node4) to (node9);
					\draw [-, black!50] (node5) to (node10);
					\draw [-, black!50] (node5) to (node11);
					\draw [-, black!50] (node6) to (node12);
					\draw [-, black!50] (node6) to (node13);
					\draw [-, black!50] (node7) to (node14);
					\draw [-, black!50] (node7) to (node15);
		
					\draw [-, black!50] (node8) to (node16);
					\draw [-, black!50] (node8) to (node17);
					\draw [-, black!50] (node9) to (node18);
					\draw [-, black!50] (node9) to (node19);
					\draw [-, black!50] (node10) to (node20);
					\draw [-, black!50] (node10) to (node21);
					\draw [-, black!50] (node11) to (node22);
					\draw [-, black!50] (node11) to (node23);
					\draw [-, black!50] (node12) to (node24);
					\draw [-, black!50] (node12) to (node25);
					\draw [-, black!50] (node13) to (node26);
					\draw [-, black!50] (node13) to (node27);
					\draw [-, black!50] (node14) to (node28);
					\draw [-, black!50] (node14) to (node29);
					\draw [-, black!50] (node15) to (node30);
					\draw [-, black!50] (node15) to (node31);
		
					\draw [-, black!50] (node16) to (node32);
					\draw [-, black!50] (node16) to (node33);
					\draw [-, black!50] (node17) to (node34);
					\draw [-, black!50] (node17) to (node35);
					\draw [-, black!50] (node18) to (node36);
					\draw [-, black!50] (node18) to (node37);
					\draw [-, black!50] (node19) to (node38);
					\draw [-, black!50] (node19) to (node39);
					\draw [-, black!50] (node20) to (node40);
					\draw [-, black!50] (node20) to (node41);
					\draw [-, black!50] (node21) to (node42);
					\draw [-, black!50] (node21) to (node43);
					\draw [-, black!50] (node22) to (node44);
					\draw [-, black!50] (node22) to (node45);
					\draw [-, black!50] (node23) to (node46);
					\draw [-, black!50] (node23) to (node47);
					\draw [-, black!50] (node24) to (node48);
					\draw [-, black!50] (node24) to (node49);
					\draw [-, black!50] (node25) to (node50);
					\draw [-, black!50] (node25) to (node51);
					\draw [-, black!50] (node26) to (node52);
					\draw [-, black!50] (node26) to (node53);
					\draw [-, black!50] (node27) to (node54);
					\draw [-, black!50] (node27) to (node55);
					\draw [-, black!50] (node28) to (node56);
					\draw [-, black!50] (node28) to (node57);
					\draw [-, black!50] (node29) to (node58);
					\draw [-, black!50] (node29) to (node59);
					\draw [-, black!50] (node30) to (node60);
					\draw [-, black!50] (node30) to (node61);
					\draw [-, black!50] (node31) to (node62);
					\draw [-, black!50] (node31) to (node63);
		
					\draw [-, black!50] (node0) to [out=-0,in=150] (8.00, 5.60) to [out=-30,in=100] (8.50, 4.50) to [out=-80,in=95] (9.00, 1.00) to [out=-85,in=85] (8.85, -6.50) to [out=-95,in=60] (node1a);				
					\draw [-, black!50] (node1a) to (node2a);
					\draw [-, black!50] (node1a) .. controls (8.50, 0.10 - 13.20) and (9.50, -0.10 - 13.20) .. (node3a);
		
					\draw [-, black!50] (node2a) to (node4a);
					\draw [-, black!50] (node2a) to (node5a);
					\draw [-, black!50] (node3a) to (node6a);
					\draw [-, black!50] (node3a) to (node7a);
		
					\draw [-, black!50] (node4a) to (node8a);
					\draw [-, black!50] (node4a) to (node9a);
					\draw [-, black!50] (node5a) to (node10a);
					\draw [-, black!50] (node5a) to (node11a);
					\draw [-, black!50] (node6a) to (node12a);
					\draw [-, black!50] (node6a) to (node13a);
					\draw [-, black!50] (node7a) to (node14a);
					\draw [-, black!50] (node7a) to (node15a);
		
					\draw [-, black!50] (node8a) to (node16a);
					\draw [-, black!50] (node8a) to (node17a);
					\draw [-, black!50] (node9a) to (node18a);
					\draw [-, black!50] (node9a) to (node19a);
					\draw [-, black!50] (node10a) to (node20a);
					\draw [-, black!50] (node10a) to (node21a);
					\draw [-, black!50] (node11a) to (node22a);
					\draw [-, black!50] (node11a) to (node23a);
					\draw [-, black!50] (node12a) to (node24a);
					\draw [-, black!50] (node12a) to (node25a);
					\draw [-, black!50] (node13a) to (node26a);
					\draw [-, black!50] (node13a) to (node27a);
					\draw [-, black!50] (node14a) to (node28a);
					\draw [-, black!50] (node14a) to (node29a);
					\draw [-, black!50] (node15a) to (node30a);
					\draw [-, black!50] (node15a) to (node31a);
		
					\draw [-, black!50] (node16a) to (node32a);
					\draw [-, black!50] (node16a) to (node33a);
					\draw [-, black!50] (node17a) to (node34a);
					\draw [-, black!50] (node17a) to (node35a);
					\draw [-, black!50] (node18a) to (node36a);
					\draw [-, black!50] (node18a) to (node37a);
					\draw [-, black!50] (node19a) to (node38a);
					\draw [-, black!50] (node19a) to (node39a);
					\draw [-, black!50] (node20a) to (node40a);
					\draw [-, black!50] (node20a) to (node41a);
					\draw [-, black!50] (node21a) to (node42a);
					\draw [-, black!50] (node21a) to (node43a);
					\draw [-, black!50] (node22a) to (node44a);
					\draw [-, black!50] (node22a) to (node45a);
					\draw [-, black!50] (node23a) to (node46a);
					\draw [-, black!50] (node23a) to (node47a);
					\draw [-, black!50] (node24a) to (node48a);
					\draw [-, black!50] (node24a) to (node49a);
					\draw [-, black!50] (node25a) to (node50a);
					\draw [-, black!50] (node25a) to (node51a);
					\draw [-, black!50] (node26a) to (node52a);
					\draw [-, black!50] (node26a) to (node53a);
					\draw [-, black!50] (node27a) to (node54a);
					\draw [-, black!50] (node27a) to (node55a);
					\draw [-, black!50] (node28a) to (node56a);
					\draw [-, black!50] (node28a) to (node57a);
					\draw [-, black!50] (node29a) to (node58a);
					\draw [-, black!50] (node29a) to (node59a);
					\draw [-, black!50] (node30a) to (node60a);
					\draw [-, black!50] (node30a) to (node61a);
					\draw [-, black!50] (node31a) to (node62a);
					\draw [-, black!50] (node31a) to (node63a);

						\draw [-] (node46) to [out=-90,in=90] (7.00, -0.75) to [out=-90,in=0] (6.50, -2.00) to [out=180,in=0] (2.23, -2.00);  
					\draw [dashed] (2.23, -2.00) -- (-0.29, -2.00);  

					\draw [-] (node46) to [out=-90,in=-90] (node47);  
					\draw [-] (node46) to [out=-90,in=90] (8.875, -6.50) to [out=-90,in=90] (node47a);  

						\draw [-] (node47) to [out=-90,in=0] ([shift={(-0.50, -1.00)}]node47) to [out=180,in=0] ([shift={(+0.50, -1.00)}]node34) to [out=180,in=-90] (node34);  
					\draw [-] (node47) to [out=-90,in=0] (7.00,-1.75) to [out=180,in=0] (2.50, -1.75) to [out=180,in=90] (2.00, -2.75) to [out=-90,in=90] (node50);  
					\draw [-] (node47a) to [out=-90,in=0] ([shift={(-0.50, -1.00)}]node47a) to [out=180,in=0] ([shift={(+0.50, -1.00)}]node34a) to [out=180,in=-90] (node34a);  
					\draw [-] (node47a) to [out=-90,in=0] (7.00,-14.95) to [out=180,in=0] (2.50, -14.95) to [out=180,in=90] (2.00, -15.95) to [out=-90,in=90] (node50a);  
					
					\draw [-] (node34) to [out=-90,in=-90] (node35);  
					\draw [-] (node34) to [out=-90,in=180] ([shift={(+0.50, -1.25)}]node34) to [out=0,in=180] ([shift={(-0.50, -1.25)}]node43) to [out=0,in=-90] (node43);  
					\draw [-] (node50) to [out=-90,in=-90] (node51);  
					\draw [-] (node50) to [out=-90,in=180] ([shift={(+0.50, -1.00)}]node50) to [out=0,in=180] ([shift={(-0.50, -1.00)}]node59) to [out=0,in=-90] (node59);  
					\draw [-] (node34a) to [out=-90,in=-90] (node35a);  
					\draw [-] (node34a) to [out=-90,in=180] ([shift={(+0.50, -1.25)}]node34a) to [out=0,in=180] ([shift={(-0.50, -1.25)}]node43a) to [out=0,in=-90] (node43a);  
					\draw [-] (node50a) to [out=-90,in=-90] (node51a);  
					\draw [-] (node50a) to [out=-90,in=180] ([shift={(+0.50, -1.00)}]node50a) to [out=0,in=180] ([shift={(-0.50, -1.00)}]node59a) to [out=0,in=-90] (node59a);  

					\draw [-] (node35) to [out=-90,in=-90] (node37);  
					\draw [-] (node35) to [out=-90,in=-90] (node33);  
					\draw [-] (node43) to [out=-90,in=-90] (node41);  
					\draw [-] (node43) to [out=-90,in=-90] (node45);  
					\draw [-] (node51) to [out=-90,in=-90] (node49);  
					\draw [-] (node51) to [out=-90,in=-90] (node53);  
					\draw [-] (node59) to [out=-90,in=-90] (node57);  
					\draw [-] (node59) to [out=-90,in=-90] (node61);  
					\draw [-] (node35a) to [out=-90,in=-90] (node33a);  
					\draw [-] (node35a) to [out=-90,in=-90] (node37a);  
					\draw [-] (node43a) to [out=-90,in=-90] (node41a);  
					\draw [-] (node43a) to [out=-90,in=-90] (node45a);  
					\draw [-] (node51a) to [out=-90,in=-90] (node49a);  
					\draw [-] (node51a) to [out=-90,in=-90] (node53a);  
					\draw [-] (node59a) to [out=-90,in=-90] (node57a);  
					\draw [-] (node59a) to [out=-90,in=-90] (node61a);  
					
					\draw [-] (node33) to [out=-90,in=-90] (node32);  
					\draw [-] (node33) to [out=-90,in=180] ([shift={(+0.50, -1.50)}]node33) to [out=0,in=180] ([shift={(-0.50, -1.50)}]node39) to [out=0,in=-90] (node39);  
					\draw [-] (node37) to [out=-90,in=-90] (node36);  
					\draw [-] (node37) to [out=-90,in=-90] (node38);  
					\draw [-] (node41) to [out=-90,in=-90] (node40);  
					\draw [-] (node41) to [out=-90,in=-90] (node42);  
					\draw [-] (node45) to [out=-90,in=-90] (node44);  
					\draw [-] (node45) to [out=-90,in=90] (node63);  
					\draw [-] (node49) to [out=-90,in=-90] (node48);  
					\draw [-] (node49) to [out=-90,in=180] ([shift={(+0.50, -1.25)}]node49) to [out=0,in=180] ([shift={(-0.50, -1.25)}]node55) to [out=0,in=-90] (node55);  
					\draw [-] (node53) to [out=-90,in=-90] (node52);  
					\draw [-] (node53) to [out=-90,in=-90] (node54);  
					\draw [-] (node57) to [out=-90,in=-90] (node56);  
					\draw [-] (node57) to [out=-90,in=-90] (node58);  
					\draw [-] (node61) to [out=-90,in=-90] (node60);  
					\draw [-] (node61) to [out=-90,in=-90] (node62);  
					\draw [-] (node33a) to [out=-90,in=-90] (node32a);  
					\draw [-] (node33a) to [out=-90,in=180] ([shift={(+0.50, -1.50)}]node33a) to [out=0,in=180] ([shift={(-0.50, -1.50)}]node39a) to [out=0,in=-90] (node39a);  
					\draw [-] (node37a) to [out=-90,in=-90] (node36a);  
					\draw [-] (node37a) to [out=-90,in=-90] (node38a);  
					\draw [-] (node41a) to [out=-90,in=-90] (node40a);  
					\draw [-] (node41a) to [out=-90,in=-90] (node42a);  
					\draw [-] (node45a) to [out=-90,in=-90] (node44a);  
					\draw [-] (node45a) to [out=-90,in=-90] (node46a);  
					\draw [-] (node49a) to [out=-90,in=-90] (node48a);  
					\draw [-] (node49a) to [out=-90,in=180] ([shift={(+0.50, -1.25)}]node49a) to [out=0,in=180] ([shift={(-0.50, -1.25)}]node55a) to [out=0,in=-90] (node55a);  
					\draw [-] (node53a) to [out=-90,in=-90] (node52a);  
					\draw [-] (node53a) to [out=-90,in=-90] (node54a);  
					\draw [-] (node57a) to [out=-90,in=-90] (node56a);  
					\draw [-] (node57a) to [out=-90,in=-90] (node58a);  
					\draw [-] (node61a) to [out=-90,in=-90] (node60a);  
					\draw [-] (node61a) to [out=-90,in=-90] (node62a);  
				\end{tikzpictureHostGraph}
			\end{comment:upperBound}
			\vspace*{-0.27cm}
			\caption[Arrangement $\phi_A$ obtained from Algorithm~\ref{algorithm:arrangementPhiU} for the guest graph $G = (V, E)$ depicted in Figure~\ref{figure:guestGraphGForHG6DepictingTheArrangementPhiA} (binary regular tree of height $h_G = 6$) -- second part.]{Arrangement $\phi_A$ obtained from Algorithm~\ref{algorithm:arrangementPhiU} for the guest graph $G = (V, E)$ depicted in Figure~\ref{figure:guestGraphGForHG6DepictingTheArrangementPhiA} -- second part. Its objective function value is $OV(G, 2, \phi_A) = 586$.}
			\label{figure:arrangementPhiAForHG6SecondPart}
		\end{figure}
		\begin{figure}[p]
			\centering
			\begin{comment:upperBound}
				\begin{tikzpictureGuestGraph}
					\node[nodeBlack] (node0) at (6.00, 5.75) {1};
					
					\node[nodeBlack] (node1) at (8.00, 5.00) {2};
					
					\node[nodeBlack] (node2) at (3.75, 4.00) {4};
		
					\node[nodeBlack] (node4) at (1.75, 3.00) {8};
					\node[nodeRed] (node5) at (5.75, 3.00) {9};
		
					\node[nodeBlack] (node8) at (0.75, 2.00) {16};
					\node[nodeBlackDashed] (node9) at (2.75, 2.00) {17};
					\node[nodeRed] (node10) at (4.75, 2.00) {18};
					\node[nodeRedDashed] (node11) at (6.75, 2.00) {19};
		
					\node[nodeBlack] (node16) at (0.25, 1.00) {32};
					\node[nodeBlackDashed] (node17) at (1.25, 1.00) {33};
					\node[nodeBlackDashed] (node18) at (2.25, 1.00) {34};
					\node[nodeBlackDashed] (node19) at (3.25, 1.00) {35};
					\node[nodeRed] (node20) at (4.25, 1.00) {36};
					\node[nodeRed] (node21) at (5.25, 1.00) {37};
					\node[nodeRedDashed] (node22) at (6.25, 1.00) {38};
					\node[nodeRedDashed] (node23) at (7.25, 1.00) {39};
		
					\node[nodeBlack] (node32) at (0.00, 0.00) {64};
					\node[nodeBlack] (node33) at (0.50, 0.00) {65};
					\node[nodeBlackDashed] (node34) at (1.00, 0.00) {66};
					\node[nodeBlackDashed] (node35) at (1.50, 0.00) {67};
					\node[nodeBlackDashed] (node36) at (2.00, 0.00) {68};
					\node[nodeBlackDashed] (node37) at (2.50, 0.00) {69};
					\node[nodeGreenDashed] (node38) at (3.00, 0.00) {70};
					\node[nodeRedDashed] (node39) at (3.50, 0.00) {71};
					\node[nodeRed] (node40) at (4.00, 0.00) {72};
					\node[nodeRed] (node41) at (4.50, 0.00) {73};
					\node[nodeRed] (node42) at (5.00, 0.00) {74};
					\node[nodeRed] (node43) at (5.50, 0.00) {75};
					\node[nodeRedDashed] (node44) at (6.00, 0.00) {76};
					\node[nodeRedDashed] (node45) at (6.50, 0.00) {77};
					\node[nodeRedDashed] (node46) at (7.00, 0.00) {78};
					\node[nodeRedDashed] (node47) at (7.50, 0.00) {79};
		
					\node[nodeBlueDashed] (node3) at (4.75, -2.50) {5};
		
					\node[nodeBlue] (node6) at (2.75, -3.50) {10};
					\node[nodeGreen] (node7) at (6.75, -3.50) {11};
		
					\node[nodeBlue] (node12) at (1.75, -4.50) {20};
					\node[nodeBlueDashed] (node13) at (3.75, -4.50) {21};
					\node[nodeGreen] (node14) at (5.75, -4.50) {22};
					\node[nodeGreenDashed] (node15) at (7.75, -4.50) {23};
		
					\node[nodeBlue] (node24) at (1.25, -5.50) {40};
					\node[nodeBlue] (node25) at (2.25, -5.50) {41};
					\node[nodeBlueDashed] (node26) at (3.25, -5.50) {42};
					\node[nodeBlueDashed] (node27) at (4.25, -5.50) {43};
					\node[nodeGreen] (node28) at (5.25, -5.50) {44};
					\node[nodeGreen] (node29) at (6.25, -5.50) {45};
					\node[nodeGreenDashed] (node30) at (7.25, -5.50) {46};
					\node[nodeGreenDashed] (node31) at (8.25, -5.50) {47};
		
					\node[nodeBlue] (node48) at (1.00, -6.50) {80};
					\node[nodeBlue] (node49) at (1.50, -6.50) {81};
					\node[nodeBlue] (node50) at (2.00, -6.50) {82};
					\node[nodeBlue] (node51) at (2.50, -6.50) {83};
					\node[nodeBlueDashed] (node52) at (3.00, -6.50) {84};
					\node[nodeBlueDashed] (node53) at (3.50, -6.50) {85};
					\node[nodeBlueDashed] (node54) at (4.00, -6.50) {86};
					\node[nodeBlueDashed] (node55) at (4.50, -6.50) {87};
					\node[nodeGreen] (node56) at (5.00, -6.50) {88};
					\node[nodeGreen] (node57) at (5.50, -6.50) {89};
					\node[nodeGreen] (node58) at (6.00, -6.50) {90};
					\node[nodeGreen] (node59) at (6.50, -6.50) {91};
					\node[nodeGreenDashed] (node60) at (7.00, -6.50) {92};
					\node[nodeGreenDashed] (node61) at (7.50, -6.50) {93};
					\node[nodeGreenDashed] (node62) at (8.00, -6.50) {94};
					\node[nodeGreenDashed] (node63) at (8.50, -6.50) {95};
		
					\node[nodeCyanDashed] (node1a) at (8.00, 5.00 - 13.20) {3};
					
					\node[nodeCyanDashed] (node2a) at (3.75, 4.00 - 13.20) {6};
		
					\node[nodeCyan] (node4a) at (1.75, 3.00 - 13.20) {12};
					\node[nodeMagenta] (node5a) at (5.75, 3.00 - 13.20) {13};
		
					\node[nodeCyan] (node8a) at (0.75, 2.00 - 13.20) {24};
					\node[nodeCyanDashed] (node9a) at (2.75, 2.00 - 13.20) {25};
					\node[nodeMagenta] (node10a) at (4.75, 2.00 - 13.20) {26};
					\node[nodeMagentaDashed] (node11a) at (6.75, 2.00 - 13.20) {27};
		
					\node[nodeCyan] (node16a) at (0.25, 1.00 - 13.20) {48};
					\node[nodeCyan] (node17a) at (1.25, 1.00 - 13.20) {49};
					\node[nodeCyanDashed] (node18a) at (2.25, 1.00 - 13.20) {50};
					\node[nodeCyanDashed] (node19a) at (3.25, 1.00 - 13.20) {51};
					\node[nodeMagenta] (node20a) at (4.25, 1.00 - 13.20) {52};
					\node[nodeMagenta] (node21a) at (5.25, 1.00 - 13.20) {53};
					\node[nodeMagentaDashed] (node22a) at (6.25, 1.00 - 13.20) {54};
					\node[nodeMagentaDashed] (node23a) at (7.25, 1.00 - 13.20) {55};
		
					\node[nodeCyan] (node32a) at (0.00, 0.00 - 13.20) {96};
					\node[nodeCyan] (node33a) at (0.50, 0.00 - 13.20) {97};
					\node[nodeCyan] (node34a) at (1.00, 0.00 - 13.20) {98};
					\node[nodeCyan] (node35a) at (1.50, 0.00 - 13.20) {99};
					\node[nodeCyanDashed] (node36a) at (2.00, 0.00 - 13.20) {100};
					\node[nodeCyanDashed] (node37a) at (2.50, 0.00 - 13.20) {101};
					\node[nodeCyanDashed] (node38a) at (3.00, 0.00 - 13.20) {102};
					\node[nodeMagentaDashed] (node39a) at (3.50, 0.00 - 13.20) {103};
					\node[nodeMagenta] (node40a) at (4.00, 0.00 - 13.20) {104};
					\node[nodeMagenta] (node41a) at (4.50, 0.00 - 13.20) {105};
					\node[nodeMagenta] (node42a) at (5.00, 0.00 - 13.20) {106};
					\node[nodeMagenta] (node43a) at (5.50, 0.00 - 13.20) {107};
					\node[nodeMagentaDashed] (node44a) at (6.00, 0.00 - 13.20) {108};
					\node[nodeMagentaDashed] (node45a) at (6.50, 0.00 - 13.20) {109};
					\node[nodeMagentaDashed] (node46a) at (7.00, 0.00 - 13.20) {110};
					\node[nodeMagentaDashed] (node47a) at (7.50, 0.00 - 13.20) {111};
		
					\node[nodeYellowDashed] (node3a) at (4.75, -2.50 - 13.20) {7};
		
					\node[nodeYellow] (node6a) at (2.75, -3.50 - 13.20) {14};
					\node[nodeGrey] (node7a) at (6.75, -3.50 - 13.20) {15};
		
					\node[nodeYellow] (node12a) at (1.75, -4.50 - 13.20) {28};
					\node[nodeYellowDashed] (node13a) at (3.75, -4.50 - 13.20) {29};
					\node[nodeGrey] (node14a) at (5.75, -4.50 - 13.20) {30};
					\node[nodeGreyDashed] (node15a) at (7.75, -4.50 - 13.20) {31};
		
					\node[nodeYellow] (node24a) at (1.25, -5.50 - 13.20) {56};
					\node[nodeYellow] (node25a) at (2.25, -5.50 - 13.20) {57};
					\node[nodeYellowDashed] (node26a) at (3.25, -5.50 - 13.20) {58};
					\node[nodeYellowDashed] (node27a) at (4.25, -5.50 - 13.20) {59};
					\node[nodeGrey] (node28a) at (5.25, -5.50 - 13.20) {60};
					\node[nodeGrey] (node29a) at (6.25, -5.50 - 13.20) {61};
					\node[nodeGreyDashed] (node30a) at (7.25, -5.50 - 13.20) {62};
					\node[nodeGreyDashed] (node31a) at (8.25, -5.50 - 13.20) {63};
		
					\node[nodeYellow] (node48a) at (1.00, -6.50 - 13.20) {112};
					\node[nodeYellow] (node49a) at (1.50, -6.50 - 13.20) {113};
					\node[nodeYellow] (node50a) at (2.00, -6.50 - 13.20) {114};
					\node[nodeYellow] (node51a) at (2.50, -6.50 - 13.20) {115};
					\node[nodeYellowDashed] (node52a) at (3.00, -6.50 - 13.20) {116};
					\node[nodeYellowDashed] (node53a) at (3.50, -6.50 - 13.20) {117};
					\node[nodeYellowDashed] (node54a) at (4.00, -6.50 - 13.20) {118};
					\node[nodeYellowDashed] (node55a) at (4.50, -6.50 - 13.20) {119};
					\node[nodeGrey] (node56a) at (5.00, -6.50 - 13.20) {120};
					\node[nodeGrey] (node57a) at (5.50, -6.50 - 13.20) {121};
					\node[nodeGrey] (node58a) at (6.00, -6.50 - 13.20) {122};
					\node[nodeGrey] (node59a) at (6.50, -6.50 - 13.20) {123};
					\node[nodeGreyDashed] (node60a) at (7.00, -6.50 - 13.20) {124};
					\node[nodeGreyDashed] (node61a) at (7.50, -6.50 - 13.20) {125};
					\node[nodeGreyDashed] (node62a) at (8.00, -6.50 - 13.20) {126};
					\node[nodeGreyDashed] (node63a) at (8.50, -6.50 - 13.20) {127};
					
					\draw (node0) to (node1);
		
					\draw (node1) to (node2);
					\draw (node1) .. controls (8.50, 0.10) and (9.50, -0.10) .. (node3);
		
					\draw (node2) to (node4);
					\draw (node2) to (node5);
					\draw (node3) to (node6);
					\draw (node3) to (node7);
		
					\draw (node4) to (node8);
					\draw (node4) to (node9);
					\draw (node5) to (node10);
					\draw (node5) to (node11);
					\draw (node6) to (node12);
					\draw (node6) to (node13);
					\draw (node7) to (node14);
					\draw (node7) to (node15);
		
					\draw (node8) to (node16);
					\draw (node8) to (node17);
					\draw (node9) to (node18);
					\draw (node9) to (node19);
					\draw (node10) to (node20);
					\draw (node10) to (node21);
					\draw (node11) to (node22);
					\draw (node11) to (node23);
					\draw (node12) to (node24);
					\draw (node12) to (node25);
					\draw (node13) to (node26);
					\draw (node13) to (node27);
					\draw (node14) to (node28);
					\draw (node14) to (node29);
					\draw (node15) to (node30);
					\draw (node15) to (node31);
		
					\draw (node16) to (node32);
					\draw (node16) to (node33);
					\draw (node17) to (node34);
					\draw (node17) to (node35);
					\draw (node18) to (node36);
					\draw (node18) to (node37);
					\draw (node19) to (node38);
					\draw (node19) to (node39);
					\draw (node20) to (node40);
					\draw (node20) to (node41);
					\draw (node21) to (node42);
					\draw (node21) to (node43);
					\draw (node22) to (node44);
					\draw (node22) to (node45);
					\draw (node23) to (node46);
					\draw (node23) to (node47);
					\draw (node24) to (node48);
					\draw (node24) to (node49);
					\draw (node25) to (node50);
					\draw (node25) to (node51);
					\draw (node26) to (node52);
					\draw (node26) to (node53);
					\draw (node27) to (node54);
					\draw (node27) to (node55);
					\draw (node28) to (node56);
					\draw (node28) to (node57);
					\draw (node29) to (node58);
					\draw (node29) to (node59);
					\draw (node30) to (node60);
					\draw (node30) to (node61);
					\draw (node31) to (node62);
					\draw (node31) to (node63);
		
					\draw (node0) to [out=-0,in=150] (8.00, 5.60) to [out=-30,in=100] (8.50, 4.50) to [out=-80,in=95] (9.00, 1.00) to [out=-85,in=85] (8.85, -6.50) to [out=-95,in=60] (node1a);				
					\draw (node1a) to (node2a);
					\draw (node1a) .. controls (8.50, 0.10 - 13.20) and (9.50, -0.10 - 13.20) .. (node3a);
		
					\draw (node2a) to (node4a);
					\draw (node2a) to (node5a);
					\draw (node3a) to (node6a);
					\draw (node3a) to (node7a);
		
					\draw (node4a) to (node8a);
					\draw (node4a) to (node9a);
					\draw (node5a) to (node10a);
					\draw (node5a) to (node11a);
					\draw (node6a) to (node12a);
					\draw (node6a) to (node13a);
					\draw (node7a) to (node14a);
					\draw (node7a) to (node15a);
		
					\draw (node8a) to (node16a);
					\draw (node8a) to (node17a);
					\draw (node9a) to (node18a);
					\draw (node9a) to (node19a);
					\draw (node10a) to (node20a);
					\draw (node10a) to (node21a);
					\draw (node11a) to (node22a);
					\draw (node11a) to (node23a);
					\draw (node12a) to (node24a);
					\draw (node12a) to (node25a);
					\draw (node13a) to (node26a);
					\draw (node13a) to (node27a);
					\draw (node14a) to (node28a);
					\draw (node14a) to (node29a);
					\draw (node15a) to (node30a);
					\draw (node15a) to (node31a);
		
					\draw (node16a) to (node32a);
					\draw (node16a) to (node33a);
					\draw (node17a) to (node34a);
					\draw (node17a) to (node35a);
					\draw (node18a) to (node36a);
					\draw (node18a) to (node37a);
					\draw (node19a) to (node38a);
					\draw (node19a) to (node39a);
					\draw (node20a) to (node40a);
					\draw (node20a) to (node41a);
					\draw (node21a) to (node42a);
					\draw (node21a) to (node43a);
					\draw (node22a) to (node44a);
					\draw (node22a) to (node45a);
					\draw (node23a) to (node46a);
					\draw (node23a) to (node47a);
					\draw (node24a) to (node48a);
					\draw (node24a) to (node49a);
					\draw (node25a) to (node50a);
					\draw (node25a) to (node51a);
					\draw (node26a) to (node52a);
					\draw (node26a) to (node53a);
					\draw (node27a) to (node54a);
					\draw (node27a) to (node55a);
					\draw (node28a) to (node56a);
					\draw (node28a) to (node57a);
					\draw (node29a) to (node58a);
					\draw (node29a) to (node59a);
					\draw (node30a) to (node60a);
					\draw (node30a) to (node61a);
					\draw (node31a) to (node62a);
					\draw (node31a) to (node63a);
				\end{tikzpictureGuestGraph}
			\end{comment:upperBound}
			\vspace*{-0.27cm}
			\caption[Guest graph $G = (V, E)$ (binary regular tree of height $h_G = 6$). The colours are related to   the arrangement $\phi$ depicted in Figures~\ref{figure:arrangementPhiForHG6FirstPart} and \ref{figure:arrangementPhiForHG6SecondPart}.]{Guest graph $G = (V, E)$ (binary regular tree of height $h_G = 6$). The colours are related to   the arrangement $\phi$ depicted  in Figures~\ref{figure:arrangementPhiForHG6FirstPart} and \ref{figure:arrangementPhiForHG6SecondPart}.}
			\label{figure:guestGraphGForHG6DepictingTheArrangementPhi}
		\end{figure}
		\begin{figure}[p]
			\centering
			\begin{comment:upperBound}
				\begin{tikzpictureHostGraph}
					\node[otherNode] (node00) at (3.75, 6.00) {};
				
					\node[otherNode] (node0) at (6.00, 5.75) {};
					
					\node[otherNode] (node1) at (8.00, 5.00) {};
					
					\node[otherNode] (node2) at (3.75, 4.00) {};
		
					\node[otherNode] (node4) at (1.75, 3.00) {};
					\node[otherNode] (node5) at (5.75, 3.00) {};
		
					\node[otherNode] (node8) at (0.75, 2.00) {};
					\node[otherNode] (node9) at (2.75, 2.00) {};
					\node[otherNode] (node10) at (4.75, 2.00) {};
					\node[otherNode] (node11) at (6.75, 2.00) {};
		
					\node[otherNode] (node16) at (0.25, 1.00) {};
					\node[otherNode] (node17) at (1.25, 1.00) {};
					\node[otherNode] (node18) at (2.25, 1.00) {};
					\node[otherNode] (node19) at (3.25, 1.00) {};
					\node[otherNode] (node20) at (4.25, 1.00) {};
					\node[otherNode] (node21) at (5.25, 1.00) {};
					\node[otherNode] (node22) at (6.25, 1.00) {};
					\node[otherNode] (node23) at (7.25, 1.00) {};
		
					\node[leafUsedBlack] (node32) at (0.00, 0.00) {1};
					\node[leafUsedBlack] (node33) at (0.50, 0.00) {2};
					\node[leafUsedBlack] (node34) at (1.00, 0.00) {4};
					\node[leafUsedBlack] (node35) at (1.50, 0.00) {8};
					\node[leafUsedBlack] (node36) at (2.00, 0.00) {16};
					\node[leafUsedBlack] (node37) at (2.50, 0.00) {64};
					\node[leafUsedBlack] (node38) at (3.00, 0.00) {32};
					\node[leafUsedBlack] (node39) at (3.50, 0.00) {65};
					\node[leafUsedBlackDashed] (node40) at (4.00, 0.00) {66};
					\node[leafUsedBlackDashed] (node41) at (4.50, 0.00) {33};
					\node[leafUsedBlackDashed] (node42) at (5.00, 0.00) {67};
					\node[leafUsedBlackDashed] (node43) at (5.50, 0.00) {68};
					\node[leafUsedBlackDashed] (node44) at (6.00, 0.00) {34};
					\node[leafUsedBlackDashed] (node45) at (6.50, 0.00) {69};
					\node[leafUsedBlackDashed] (node46) at (7.00, 0.00) {17};
					\node[leafUsedBlackDashed] (node47) at (7.50, 0.00) {35};
		
					\node[otherNode] (node3) at (4.75, -2.50) {};
		
					\node[otherNode] (node6) at (2.75, -3.50) {};
					\node[otherNode] (node7) at (6.75, -3.50) {};
		
					\node[otherNode] (node12) at (1.75, -4.50) {};
					\node[otherNode] (node13) at (3.75, -4.50) {};
					\node[otherNode] (node14) at (5.75, -4.50) {};
					\node[otherNode] (node15) at (7.75, -4.50) {};
		
					\node[otherNode] (node24) at (1.25, -5.50) {};
					\node[otherNode] (node25) at (2.25, -5.50) {};
					\node[otherNode] (node26) at (3.25, -5.50) {};
					\node[otherNode] (node27) at (4.25, -5.50) {};
					\node[otherNode] (node28) at (5.25, -5.50) {};
					\node[otherNode] (node29) at (6.25, -5.50) {};
					\node[otherNode] (node30) at (7.25, -5.50) {};
					\node[otherNode] (node31) at (8.25, -5.50) {};
		
					\node[leafUsedRed] (node48) at (1.00, -6.50) {9};
					\node[leafUsedRed] (node49) at (1.50, -6.50) {18};
					\node[leafUsedRed] (node50) at (2.00, -6.50) {36};
					\node[leafUsedRed] (node51) at (2.50, -6.50) {72};
					\node[leafUsedRed] (node52) at (3.00, -6.50) {73};
					\node[leafUsedRed] (node53) at (3.50, -6.50) {74};
					\node[leafUsedRed] (node54) at (4.00, -6.50) {37};
					\node[leafUsedRed] (node55) at (4.50, -6.50) {75};
					\node[leafUsedRedDashed] (node56) at (5.00, -6.50) {76};
					\node[leafUsedRedDashed] (node57) at (5.50, -6.50) {38};
					\node[leafUsedRedDashed] (node58) at (6.00, -6.50) {77};
					\node[leafUsedRedDashed] (node59) at (6.50, -6.50) {19};
					\node[leafUsedRedDashed] (node60) at (7.00, -6.50) {78};
					\node[leafUsedRedDashed] (node61) at (7.50, -6.50) {39};
					\node[leafUsedRedDashed] (node62) at (8.00, -6.50) {79};
					\node[leafUsedRedDashed] (node63) at (8.50, -6.50) {71};
		
					\node[otherNode] (node1a) at (8.00, 5.00 - 13.20) {};
					
					\node[otherNode] (node2a) at (3.75, 4.00 - 13.20) {};
		
					\node[otherNode] (node4a) at (1.75, 3.00 - 13.20) {};
					\node[otherNode] (node5a) at (5.75, 3.00 - 13.20) {};
		
					\node[otherNode] (node8a) at (0.75, 2.00 - 13.20) {};
					\node[otherNode] (node9a) at (2.75, 2.00 - 13.20) {};
					\node[otherNode] (node10a) at (4.75, 2.00 - 13.20) {};
					\node[otherNode] (node11a) at (6.75, 2.00 - 13.20) {};
		
					\node[otherNode] (node16a) at (0.25, 1.00 - 13.20) {};
					\node[otherNode] (node17a) at (1.25, 1.00 - 13.20) {};
					\node[otherNode] (node18a) at (2.25, 1.00 - 13.20) {};
					\node[otherNode] (node19a) at (3.25, 1.00 - 13.20) {};
					\node[otherNode] (node20a) at (4.25, 1.00 - 13.20) {};
					\node[otherNode] (node21a) at (5.25, 1.00 - 13.20) {};
					\node[otherNode] (node22a) at (6.25, 1.00 - 13.20) {};
					\node[otherNode] (node23a) at (7.25, 1.00 - 13.20) {};
		
					\node[leafUsedBlue] (node32a) at (0.00, 0.00 - 13.20) {10};
					\node[leafUsedBlue] (node33a) at (0.50, 0.00 - 13.20) {20};
					\node[leafUsedBlue] (node34a) at (1.00, 0.00 - 13.20) {40};
					\node[leafUsedBlue] (node35a) at (1.50, 0.00 - 13.20) {80};
					\node[leafUsedBlue] (node36a) at (2.00, 0.00 - 13.20) {81};
					\node[leafUsedBlue] (node37a) at (2.50, 0.00 - 13.20) {82};
					\node[leafUsedBlue] (node38a) at (3.00, 0.00 - 13.20) {41};
					\node[leafUsedBlue] (node39a) at (3.50, 0.00 - 13.20) {83};
					\node[leafUsedBlueDashed] (node40a) at (4.00, 0.00 - 13.20) {84};
					\node[leafUsedBlueDashed] (node41a) at (4.50, 0.00 - 13.20) {42};
					\node[leafUsedBlueDashed] (node42a) at (5.00, 0.00 - 13.20) {85};
					\node[leafUsedBlueDashed] (node43a) at (5.50, 0.00 - 13.20) {21};
					\node[leafUsedBlueDashed] (node44a) at (6.00, 0.00 - 13.20) {86};
					\node[leafUsedBlueDashed] (node45a) at (6.50, 0.00 - 13.20) {43};
					\node[leafUsedBlueDashed] (node46a) at (7.00, 0.00 - 13.20) {87};
					\node[leafUsedBlueDashed] (node47a) at (7.50, 0.00 - 13.20) {5};
		
					\node[otherNode] (node3a) at (4.75, -2.50 - 13.20) {};
		
					\node[otherNode] (node6a) at (2.75, -3.50 - 13.20) {};
					\node[otherNode] (node7a) at (6.75, -3.50 - 13.20) {};
		
					\node[otherNode] (node12a) at (1.75, -4.50 - 13.20) {};
					\node[otherNode] (node13a) at (3.75, -4.50 - 13.20) {};
					\node[otherNode] (node14a) at (5.75, -4.50 - 13.20) {};
					\node[otherNode] (node15a) at (7.75, -4.50 - 13.20) {};
		
					\node[otherNode] (node24a) at (1.25, -5.50 - 13.20) {};
					\node[otherNode] (node25a) at (2.25, -5.50 - 13.20) {};
					\node[otherNode] (node26a) at (3.25, -5.50 - 13.20) {};
					\node[otherNode] (node27a) at (4.25, -5.50 - 13.20) {};
					\node[otherNode] (node28a) at (5.25, -5.50 - 13.20) {};
					\node[otherNode] (node29a) at (6.25, -5.50 - 13.20) {};
					\node[otherNode] (node30a) at (7.25, -5.50 - 13.20) {};
					\node[otherNode] (node31a) at (8.25, -5.50 - 13.20) {};
		
					\node[leafUsedGreen] (node48a) at (1.00, -6.50 - 13.20) {11};
					\node[leafUsedGreen] (node49a) at (1.50, -6.50 - 13.20) {22};
					\node[leafUsedGreen] (node50a) at (2.00, -6.50 - 13.20) {44};
					\node[leafUsedGreen] (node51a) at (2.50, -6.50 - 13.20) {88};
					\node[leafUsedGreen] (node52a) at (3.00, -6.50 - 13.20) {89};
					\node[leafUsedGreen] (node53a) at (3.50, -6.50 - 13.20) {90};
					\node[leafUsedGreen] (node54a) at (4.00, -6.50 - 13.20) {45};
					\node[leafUsedGreen] (node55a) at (4.50, -6.50 - 13.20) {91};
					\node[leafUsedGreenDashed] (node56a) at (5.00, -6.50 - 13.20) {92};
					\node[leafUsedGreenDashed] (node57a) at (5.50, -6.50 - 13.20) {46};
					\node[leafUsedGreenDashed] (node58a) at (6.00, -6.50 - 13.20) {93};
					\node[leafUsedGreenDashed] (node59a) at (6.50, -6.50 - 13.20) {23};
					\node[leafUsedGreenDashed] (node60a) at (7.00, -6.50 - 13.20) {94};
					\node[leafUsedGreenDashed] (node61a) at (7.50, -6.50 - 13.20) {47};
					\node[leafUsedGreenDashed] (node62a) at (8.00, -6.50 - 13.20) {95};
					\node[leafUsedGreenDashed] (node63a) at (8.50, -6.50 - 13.20) {70};
					
					\draw [-, black!50] (node00) to [out=30,in=180] (4.875, 6.50) to [out=0,in=180] (7.00, 6.50);
					\draw [dashed, black!50] (7.00, 6.50) to (9.52, 6.50);
					\draw [-, black!50] (node00) to (node0);
					
					\draw [-, black!50] (node0) to (node1);
		
					\draw [-, black!50] (node1) to (node2);
					\draw [-, black!50] (node1) .. controls (8.50, 0.10) and (9.50, -0.10) .. (node3);
		
					\draw [-, black!50] (node2) to (node4);
					\draw [-, black!50] (node2) to (node5);
					\draw [-, black!50] (node3) to (node6);
					\draw [-, black!50] (node3) to (node7);
		
					\draw [-, black!50] (node4) to (node8);
					\draw [-, black!50] (node4) to (node9);
					\draw [-, black!50] (node5) to (node10);
					\draw [-, black!50] (node5) to (node11);
					\draw [-, black!50] (node6) to (node12);
					\draw [-, black!50] (node6) to (node13);
					\draw [-, black!50] (node7) to (node14);
					\draw [-, black!50] (node7) to (node15);
		
					\draw [-, black!50] (node8) to (node16);
					\draw [-, black!50] (node8) to (node17);
					\draw [-, black!50] (node9) to (node18);
					\draw [-, black!50] (node9) to (node19);
					\draw [-, black!50] (node10) to (node20);
					\draw [-, black!50] (node10) to (node21);
					\draw [-, black!50] (node11) to (node22);
					\draw [-, black!50] (node11) to (node23);
					\draw [-, black!50] (node12) to (node24);
					\draw [-, black!50] (node12) to (node25);
					\draw [-, black!50] (node13) to (node26);
					\draw [-, black!50] (node13) to (node27);
					\draw [-, black!50] (node14) to (node28);
					\draw [-, black!50] (node14) to (node29);
					\draw [-, black!50] (node15) to (node30);
					\draw [-, black!50] (node15) to (node31);
		
					\draw [-, black!50] (node16) to (node32);
					\draw [-, black!50] (node16) to (node33);
					\draw [-, black!50] (node17) to (node34);
					\draw [-, black!50] (node17) to (node35);
					\draw [-, black!50] (node18) to (node36);
					\draw [-, black!50] (node18) to (node37);
					\draw [-, black!50] (node19) to (node38);
					\draw [-, black!50] (node19) to (node39);
					\draw [-, black!50] (node20) to (node40);
					\draw [-, black!50] (node20) to (node41);
					\draw [-, black!50] (node21) to (node42);
					\draw [-, black!50] (node21) to (node43);
					\draw [-, black!50] (node22) to (node44);
					\draw [-, black!50] (node22) to (node45);
					\draw [-, black!50] (node23) to (node46);
					\draw [-, black!50] (node23) to (node47);
					\draw [-, black!50] (node24) to (node48);
					\draw [-, black!50] (node24) to (node49);
					\draw [-, black!50] (node25) to (node50);
					\draw [-, black!50] (node25) to (node51);
					\draw [-, black!50] (node26) to (node52);
					\draw [-, black!50] (node26) to (node53);
					\draw [-, black!50] (node27) to (node54);
					\draw [-, black!50] (node27) to (node55);
					\draw [-, black!50] (node28) to (node56);
					\draw [-, black!50] (node28) to (node57);
					\draw [-, black!50] (node29) to (node58);
					\draw [-, black!50] (node29) to (node59);
					\draw [-, black!50] (node30) to (node60);
					\draw [-, black!50] (node30) to (node61);
					\draw [-, black!50] (node31) to (node62);
					\draw [-, black!50] (node31) to (node63);
		
					\draw [-, black!50] (node0) to [out=-0,in=150] (8.00, 5.60) to [out=-30,in=100] (8.50, 4.50) to [out=-80,in=95] (9.00, 1.00) to [out=-85,in=85] (8.85, -6.50) to [out=-95,in=60] (node1a);				
					\draw [-, black!50] (node1a) to (node2a);
					\draw [-, black!50] (node1a) .. controls (8.50, 0.10 - 13.20) and (9.50, -0.10 - 13.20) .. (node3a);
		
					\draw [-, black!50] (node2a) to (node4a);
					\draw [-, black!50] (node2a) to (node5a);
					\draw [-, black!50] (node3a) to (node6a);
					\draw [-, black!50] (node3a) to (node7a);
		
					\draw [-, black!50] (node4a) to (node8a);
					\draw [-, black!50] (node4a) to (node9a);
					\draw [-, black!50] (node5a) to (node10a);
					\draw [-, black!50] (node5a) to (node11a);
					\draw [-, black!50] (node6a) to (node12a);
					\draw [-, black!50] (node6a) to (node13a);
					\draw [-, black!50] (node7a) to (node14a);
					\draw [-, black!50] (node7a) to (node15a);
		
					\draw [-, black!50] (node8a) to (node16a);
					\draw [-, black!50] (node8a) to (node17a);
					\draw [-, black!50] (node9a) to (node18a);
					\draw [-, black!50] (node9a) to (node19a);
					\draw [-, black!50] (node10a) to (node20a);
					\draw [-, black!50] (node10a) to (node21a);
					\draw [-, black!50] (node11a) to (node22a);
					\draw [-, black!50] (node11a) to (node23a);
					\draw [-, black!50] (node12a) to (node24a);
					\draw [-, black!50] (node12a) to (node25a);
					\draw [-, black!50] (node13a) to (node26a);
					\draw [-, black!50] (node13a) to (node27a);
					\draw [-, black!50] (node14a) to (node28a);
					\draw [-, black!50] (node14a) to (node29a);
					\draw [-, black!50] (node15a) to (node30a);
					\draw [-, black!50] (node15a) to (node31a);
		
					\draw [-, black!50] (node16a) to (node32a);
					\draw [-, black!50] (node16a) to (node33a);
					\draw [-, black!50] (node17a) to (node34a);
					\draw [-, black!50] (node17a) to (node35a);
					\draw [-, black!50] (node18a) to (node36a);
					\draw [-, black!50] (node18a) to (node37a);
					\draw [-, black!50] (node19a) to (node38a);
					\draw [-, black!50] (node19a) to (node39a);
					\draw [-, black!50] (node20a) to (node40a);
					\draw [-, black!50] (node20a) to (node41a);
					\draw [-, black!50] (node21a) to (node42a);
					\draw [-, black!50] (node21a) to (node43a);
					\draw [-, black!50] (node22a) to (node44a);
					\draw [-, black!50] (node22a) to (node45a);
					\draw [-, black!50] (node23a) to (node46a);
					\draw [-, black!50] (node23a) to (node47a);
					\draw [-, black!50] (node24a) to (node48a);
					\draw [-, black!50] (node24a) to (node49a);
					\draw [-, black!50] (node25a) to (node50a);
					\draw [-, black!50] (node25a) to (node51a);
					\draw [-, black!50] (node26a) to (node52a);
					\draw [-, black!50] (node26a) to (node53a);
					\draw [-, black!50] (node27a) to (node54a);
					\draw [-, black!50] (node27a) to (node55a);
					\draw [-, black!50] (node28a) to (node56a);
					\draw [-, black!50] (node28a) to (node57a);
					\draw [-, black!50] (node29a) to (node58a);
					\draw [-, black!50] (node29a) to (node59a);
					\draw [-, black!50] (node30a) to (node60a);
					\draw [-, black!50] (node30a) to (node61a);
					\draw [-, black!50] (node31a) to (node62a);
					\draw [-, black!50] (node31a) to (node63a);

					\draw [-] (node32) to [out=-90,in=-90] (node33);  
					\draw [-] (node32) to [out=-90,in=90] (0.00, -0.75) to [out=-90,in=180] (0.50, -2.00) to [out=0,in=180] (7.00, -2.00);  
					\draw [dashed] (7.00, -2.00) -- (9.52, -2.00);  
					
					\draw [-] (node33) to [out=-90,in=-90] (node34);  
					\draw [-] (node33) to [out=-90,in=90] (0.50, -7.00) to [out=-90,in=180] (1.00, -8.00) to [out=0,in=180] (7.00, -8.00) to [out=0,in=90] (7.50, -9.00) to [out=-90,in=90] (node47a);  

					\draw [-] (node34) to [out=-90,in=-90] (node35);  
					\draw [-] (node34) to [out=-90,in=90] (node48);  
					\draw [-] (node47a) to [out=-90,in=0] ([shift={(-0.50, -1.50)}]node47a) to [out=180,in=0]([shift={(+0.50, -1.50)}]node32a) to [out=180,in=-90] (node32a);  
					\draw [-] (node47a) to [out=-90,in=0] (7.00,-14.95) to [out=180,in=0] (1.50, -14.95) to [out=180,in=90] (1.00, -15.95) to [out=-90,in=90] (node48a);  

					\draw [-] (node35) to [out=-90,in=-90] (node36);  
					\draw [-] (node35) to [out=-90,in=180] ([shift={(+0.50, -1.25)}]node35) to [out=0,in=180] ([shift={(-0.50, -1.25)}]node46) to [out=0,in=-90] (node46);  
					\draw [-] (node48) to [out=-90,in=-90] (node49);  
					\draw [-] (node48) to [out=-90,in=180] ([shift={(+0.50, -1.25)}]node48) to [out=0,in=180] ([shift={(-0.50, -1.25)}]node59) to [out=0,in=-90] (node59);  
					\draw [-] (node32a) to [out=-90,in=-90] (node33a);  
					\draw [-] (node32a) to [out=-90,in=180] ([shift={(+0.50, -1.25)}]node32a) to [out=0,in=180] ([shift={(-0.50, -1.25)}]node43a) to [out=0,in=-90] (node43a);  
					\draw [-] (node48a) to [out=-90,in=-90] (node49a);  
					\draw [-] (node48a) to [out=-90,in=180] ([shift={(+0.50, -1.25)}]node48a) to [out=0,in=180] ([shift={(-0.50, -1.25)}]node59a) to [out=0,in=-90] (node59a);  
					
					\draw [-] (node36) to [out=-90,in=-90] (node38);  
					\draw [-] (node36) to [out=-90,in=180] ([shift={(+0.50, -1.00)}]node36) to [out=0,in=180] ([shift={(-0.50, -1.00)}]node41) to [out=0,in=-90] (node41);  
					\draw [-] (node46) to [out=-90,in=-90] (node44);  
					\draw [-] (node46) to [out=-90,in=-90] (node47);  
					\draw [-] (node49) to [out=-90,in=-90] (node50);  
					\draw [-] (node49) to [out=-90,in=180] ([shift={(+0.50, -1.00)}]node49) to [out=0,in=180] ([shift={(-0.50, -1.00)}]node54) to [out=0,in=-90] (node54);  
					\draw [-] (node59) to [out=-90,in=-90] (node57);  
					\draw [-] (node59) to [out=-90,in=-90] (node61);  
					\draw [-] (node33a) to [out=-90,in=-90] (node34a);  
					\draw [-] (node33a) to [out=-90,in=180] ([shift={(+0.50, -1.00)}]node33a) to [out=0,in=180] ([shift={(-0.50, -1.00)}]node38a) to [out=0,in=-90] (node38a);  
					\draw [-] (node43a) to [out=-90,in=-90] (node41a);  
					\draw [-] (node43a) to [out=-90,in=-90] (node45a);  
					\draw [-] (node49a) to [out=-90,in=-90] (node50a);  
					\draw [-] (node49a) to [out=-90,in=180] ([shift={(+0.50, -1.00)}]node49a) to [out=0,in=180] ([shift={(-0.50, -1.00)}]node54a) to [out=0,in=-90] (node54a);  
					\draw [-] (node59a) to [out=-90,in=-90] (node57a);  
					\draw [-] (node59a) to [out=-90,in=-90] (node61a);  

					\draw [-] (node38) to [out=-90,in=-90] (node37);  
					\draw [-] (node38) to [out=-90,in=-90] (node39);  
					\draw [-] (node41) to [out=-90,in=-90] (node40);  
					\draw [-] (node41) to [out=-90,in=-90] (node42);  
					\draw [-] (node44) to [out=-90,in=-90] (node43);  
					\draw [-] (node44) to [out=-90,in=-90] (node45);  
					\draw [-] (node47) to [out=-90,in=90] (8.875, -6.50) to [out=-90,in=90] (node63a);  
					\draw [-] (node47) to [out=-90,in=90] (node63);  
					\draw [-] (node50) to [out=-90,in=-90] (node51);  
					\draw [-] (node50) to [out=-90,in=-90] (node52);  
					\draw [-] (node54) to [out=-90,in=-90] (node53);  
					\draw [-] (node54) to [out=-90,in=-90] (node55);  
					\draw [-] (node57) to [out=-90,in=-90] (node56);  
					\draw [-] (node57) to [out=-90,in=-90] (node58);  
					\draw [-] (node61) to [out=-90,in=-90] (node60);  
					\draw [-] (node61) to [out=-90,in=-90] (node62);  
					\draw [-] (node34a) to [out=-90,in=-90] (node35a);  
					\draw [-] (node34a) to [out=-90,in=-90] (node36a);  
					\draw [-] (node38a) to [out=-90,in=-90] (node37a);  
					\draw [-] (node38a) to [out=-90,in=-90] (node39a);  
					\draw [-] (node41a) to [out=-90,in=-90] (node40a);  
					\draw [-] (node41a) to [out=-90,in=-90] (node42a);  
					\draw [-] (node45a) to [out=-90,in=-90] (node44a);  
					\draw [-] (node45a) to [out=-90,in=-90] (node46a);  
					\draw [-] (node50a) to [out=-90,in=-90] (node51a);  
					\draw [-] (node50a) to [out=-90,in=-90] (node52a);  
					\draw [-] (node54a) to [out=-90,in=-90] (node53a);  
					\draw [-] (node54a) to [out=-90,in=-90] (node55a);  
					\draw [-] (node57a) to [out=-90,in=-90] (node56a);  
					\draw [-] (node57a) to [out=-90,in=-90] (node58a);  
					\draw [-] (node61a) to [out=-90,in=-90] (node60a);  
					\draw [-] (node61a) to [out=-90,in=-90] (node62a);  
				\end{tikzpictureHostGraph}
			\end{comment:upperBound}
			\vspace*{-0.27cm}
			\caption[Arrangement $\phi$ for the guest graph $G = (V, E)$ depicted in Figure~\ref{figure:guestGraphGForHG6DepictingTheArrangementPhi} (binary regular tree of height $h_G = 6$) -- first part.]{Arrangement $\phi$ for the guest graph $G = (V, E)$ depicted in Figure~\ref{figure:guestGraphGForHG6DepictingTheArrangementPhi} -- first part. Its objective function value is $OV(G, 2, \phi_A) = 584$.}
			\label{figure:arrangementPhiForHG6FirstPart}
		\end{figure}
		\begin{figure}[p]
			\centering
			\begin{comment:upperBound}
				\begin{tikzpictureHostGraph}
					\node[otherNode] (node0) at (6.00, 5.75) {};
					
					\node[otherNode] (node1) at (8.00, 5.00) {};
					
					\node[otherNode] (node2) at (3.75, 4.00) {};
		
					\node[otherNode] (node4) at (1.75, 3.00) {};
					\node[otherNode] (node5) at (5.75, 3.00) {};
		
					\node[otherNode] (node8) at (0.75, 2.00) {};
					\node[otherNode] (node9) at (2.75, 2.00) {};
					\node[otherNode] (node10) at (4.75, 2.00) {};
					\node[otherNode] (node11) at (6.75, 2.00) {};
		
					\node[otherNode] (node16) at (0.25, 1.00) {};
					\node[otherNode] (node17) at (1.25, 1.00) {};
					\node[otherNode] (node18) at (2.25, 1.00) {};
					\node[otherNode] (node19) at (3.25, 1.00) {};
					\node[otherNode] (node20) at (4.25, 1.00) {};
					\node[otherNode] (node21) at (5.25, 1.00) {};
					\node[otherNode] (node22) at (6.25, 1.00) {};
					\node[otherNode] (node23) at (7.25, 1.00) {};
		
					\node[leafUsedCyan] (node32) at (0.00, 0.00) {12};
					\node[leafUsedCyan] (node33) at (0.50, 0.00) {24};
					\node[leafUsedCyan] (node34) at (1.00, 0.00) {48};
					\node[leafUsedCyan] (node35) at (1.50, 0.00) {96};
					\node[leafUsedCyan] (node36) at (2.00, 0.00) {97};
					\node[leafUsedCyan] (node37) at (2.50, 0.00) {98};
					\node[leafUsedCyan] (node38) at (3.00, 0.00) {49};
					\node[leafUsedCyan] (node39) at (3.50, 0.00) {99};
					\node[leafUsedCyanDashed] (node40) at (4.00, 0.00) {100};
					\node[leafUsedCyanDashed] (node41) at (4.50, 0.00) {50};
					\node[leafUsedCyanDashed] (node42) at (5.00, 0.00) {101};
					\node[leafUsedCyanDashed] (node43) at (5.50, 0.00) {25};
					\node[leafUsedCyanDashed] (node44) at (6.00, 0.00) {6};
					\node[leafUsedCyanDashed] (node45) at (6.50, 0.00) {3};
					\node[leafUsedCyanDashed] (node46) at (7.00, 0.00) {102};
					\node[leafUsedCyanDashed] (node47) at (7.50, 0.00) {51};
		
					\node[otherNode] (node3) at (4.75, -2.50) {};
		
					\node[otherNode] (node6) at (2.75, -3.50) {};
					\node[otherNode] (node7) at (6.75, -3.50) {};
		
					\node[otherNode] (node12) at (1.75, -4.50) {};
					\node[otherNode] (node13) at (3.75, -4.50) {};
					\node[otherNode] (node14) at (5.75, -4.50) {};
					\node[otherNode] (node15) at (7.75, -4.50) {};
		
					\node[otherNode] (node24) at (1.25, -5.50) {};
					\node[otherNode] (node25) at (2.25, -5.50) {};
					\node[otherNode] (node26) at (3.25, -5.50) {};
					\node[otherNode] (node27) at (4.25, -5.50) {};
					\node[otherNode] (node28) at (5.25, -5.50) {};
					\node[otherNode] (node29) at (6.25, -5.50) {};
					\node[otherNode] (node30) at (7.25, -5.50) {};
					\node[otherNode] (node31) at (8.25, -5.50) {};
		
					\node[leafUsedMagenta] (node48) at (1.00, -6.50) {13};
					\node[leafUsedMagenta] (node49) at (1.50, -6.50) {26};
					\node[leafUsedMagenta] (node50) at (2.00, -6.50) {52};
					\node[leafUsedMagenta] (node51) at (2.50, -6.50) {104};
					\node[leafUsedMagenta] (node52) at (3.00, -6.50) {105};
					\node[leafUsedMagenta] (node53) at (3.50, -6.50) {106};
					\node[leafUsedMagenta] (node54) at (4.00, -6.50) {53};
					\node[leafUsedMagenta] (node55) at (4.50, -6.50) {107};
					\node[leafUsedMagentaDashed] (node56) at (5.00, -6.50) {108};
					\node[leafUsedMagentaDashed] (node57) at (5.50, -6.50) {54};
					\node[leafUsedMagentaDashed] (node58) at (6.00, -6.50) {109};
					\node[leafUsedMagentaDashed] (node59) at (6.50, -6.50) {27};
					\node[leafUsedMagentaDashed] (node60) at (7.00, -6.50) {110};
					\node[leafUsedMagentaDashed] (node61) at (7.50, -6.50) {55};
					\node[leafUsedMagentaDashed] (node62) at (8.00, -6.50) {111};
					\node[leafUsedMagentaDashed] (node63) at (8.50, -6.50) {103};
		
					\node[otherNode] (node1a) at (8.00, 5.00 - 13.20) {};
					
					\node[otherNode] (node2a) at (3.75, 4.00 - 13.20) {};
		
					\node[otherNode] (node4a) at (1.75, 3.00 - 13.20) {};
					\node[otherNode] (node5a) at (5.75, 3.00 - 13.20) {};
		
					\node[otherNode] (node8a) at (0.75, 2.00 - 13.20) {};
					\node[otherNode] (node9a) at (2.75, 2.00 - 13.20) {};
					\node[otherNode] (node10a) at (4.75, 2.00 - 13.20) {};
					\node[otherNode] (node11a) at (6.75, 2.00 - 13.20) {};
		
					\node[otherNode] (node16a) at (0.25, 1.00 - 13.20) {};
					\node[otherNode] (node17a) at (1.25, 1.00 - 13.20) {};
					\node[otherNode] (node18a) at (2.25, 1.00 - 13.20) {};
					\node[otherNode] (node19a) at (3.25, 1.00 - 13.20) {};
					\node[otherNode] (node20a) at (4.25, 1.00 - 13.20) {};
					\node[otherNode] (node21a) at (5.25, 1.00 - 13.20) {};
					\node[otherNode] (node22a) at (6.25, 1.00 - 13.20) {};
					\node[otherNode] (node23a) at (7.25, 1.00 - 13.20) {};
		
					\node[leafUsedYellow] (node32a) at (0.00, 0.00 - 13.20) {14};
					\node[leafUsedYellow] (node33a) at (0.50, 0.00 - 13.20) {28};
					\node[leafUsedYellow] (node34a) at (1.00, 0.00 - 13.20) {56};
					\node[leafUsedYellow] (node35a) at (1.50, 0.00 - 13.20) {112};
					\node[leafUsedYellow] (node36a) at (2.00, 0.00 - 13.20) {113};
					\node[leafUsedYellow] (node37a) at (2.50, 0.00 - 13.20) {114};
					\node[leafUsedYellow] (node38a) at (3.00, 0.00 - 13.20) {57};
					\node[leafUsedYellow] (node39a) at (3.50, 0.00 - 13.20) {115};
					\node[leafUsedYellowDashed] (node40a) at (4.00, 0.00 - 13.20) {116};
					\node[leafUsedYellowDashed] (node41a) at (4.50, 0.00 - 13.20) {58};
					\node[leafUsedYellowDashed] (node42a) at (5.00, 0.00 - 13.20) {117};
					\node[leafUsedYellowDashed] (node43a) at (5.50, 0.00 - 13.20) {29};
					\node[leafUsedYellowDashed] (node44a) at (6.00, 0.00 - 13.20) {118};
					\node[leafUsedYellowDashed] (node45a) at (6.50, 0.00 - 13.20) {59};
					\node[leafUsedYellowDashed] (node46a) at (7.00, 0.00 - 13.20) {119};
					\node[leafUsedYellowDashed] (node47a) at (7.50, 0.00 - 13.20) {7};
		
					\node[otherNode] (node3a) at (4.75, -2.50 - 13.20) {};
		
					\node[otherNode] (node6a) at (2.75, -3.50 - 13.20) {};
					\node[otherNode] (node7a) at (6.75, -3.50 - 13.20) {};
		
					\node[otherNode] (node12a) at (1.75, -4.50 - 13.20) {};
					\node[otherNode] (node13a) at (3.75, -4.50 - 13.20) {};
					\node[otherNode] (node14a) at (5.75, -4.50 - 13.20) {};
					\node[otherNode] (node15a) at (7.75, -4.50 - 13.20) {};
		
					\node[otherNode] (node24a) at (1.25, -5.50 - 13.20) {};
					\node[otherNode] (node25a) at (2.25, -5.50 - 13.20) {};
					\node[otherNode] (node26a) at (3.25, -5.50 - 13.20) {};
					\node[otherNode] (node27a) at (4.25, -5.50 - 13.20) {};
					\node[otherNode] (node28a) at (5.25, -5.50 - 13.20) {};
					\node[otherNode] (node29a) at (6.25, -5.50 - 13.20) {};
					\node[otherNode] (node30a) at (7.25, -5.50 - 13.20) {};
					\node[otherNode] (node31a) at (8.25, -5.50 - 13.20) {};
		
					\node[leafUsedGrey] (node48a) at (1.00, -6.50 - 13.20) {15};
					\node[leafUsedGrey] (node49a) at (1.50, -6.50 - 13.20) {30};
					\node[leafUsedGrey] (node50a) at (2.00, -6.50 - 13.20) {60};
					\node[leafUsedGrey] (node51a) at (2.50, -6.50 - 13.20) {120};
					\node[leafUsedGrey] (node52a) at (3.00, -6.50 - 13.20) {121};
					\node[leafUsedGrey] (node53a) at (3.50, -6.50 - 13.20) {122};
					\node[leafUsedGrey] (node54a) at (4.00, -6.50 - 13.20) {61};
					\node[leafUsedGrey] (node55a) at (4.50, -6.50 - 13.20) {123};
					\node[leafUsedGreyDashed] (node56a) at (5.00, -6.50 - 13.20) {124};
					\node[leafUsedGreyDashed] (node57a) at (5.50, -6.50 - 13.20) {62};
					\node[leafUsedGreyDashed] (node58a) at (6.00, -6.50 - 13.20) {125};
					\node[leafUsedGreyDashed] (node59a) at (6.50, -6.50 - 13.20) {31};
					\node[leafUsedGreyDashed] (node60a) at (7.00, -6.50 - 13.20) {126};
					\node[leafUsedGreyDashed] (node61a) at (7.50, -6.50 - 13.20) {63};
					\node[leafUsedGreyDashed] (node62a) at (8.00, -6.50 - 13.20) {127};
					\node[leafUnused] (node63a) at (8.50, -6.50 - 13.20) {};
					
					\draw [-, black!50] (node0) to [out=150,in=0] (4.875, 6.50) to [out=180,in=0] (2.23, 6.50);
					\draw [dashed, black!50] (2.23, 6.50) to (-0.29, 6.50);
					
					\draw [-, black!50] (node0) to (node1);
		
					\draw [-, black!50] (node1) to (node2);
					\draw [-, black!50] (node1) .. controls (8.50, 0.10) and (9.50, -0.10) .. (node3);
		
					\draw [-, black!50] (node2) to (node4);
					\draw [-, black!50] (node2) to (node5);
					\draw [-, black!50] (node3) to (node6);
					\draw [-, black!50] (node3) to (node7);
		
					\draw [-, black!50] (node4) to (node8);
					\draw [-, black!50] (node4) to (node9);
					\draw [-, black!50] (node5) to (node10);
					\draw [-, black!50] (node5) to (node11);
					\draw [-, black!50] (node6) to (node12);
					\draw [-, black!50] (node6) to (node13);
					\draw [-, black!50] (node7) to (node14);
					\draw [-, black!50] (node7) to (node15);
		
					\draw [-, black!50] (node8) to (node16);
					\draw [-, black!50] (node8) to (node17);
					\draw [-, black!50] (node9) to (node18);
					\draw [-, black!50] (node9) to (node19);
					\draw [-, black!50] (node10) to (node20);
					\draw [-, black!50] (node10) to (node21);
					\draw [-, black!50] (node11) to (node22);
					\draw [-, black!50] (node11) to (node23);
					\draw [-, black!50] (node12) to (node24);
					\draw [-, black!50] (node12) to (node25);
					\draw [-, black!50] (node13) to (node26);
					\draw [-, black!50] (node13) to (node27);
					\draw [-, black!50] (node14) to (node28);
					\draw [-, black!50] (node14) to (node29);
					\draw [-, black!50] (node15) to (node30);
					\draw [-, black!50] (node15) to (node31);
		
					\draw [-, black!50] (node16) to (node32);
					\draw [-, black!50] (node16) to (node33);
					\draw [-, black!50] (node17) to (node34);
					\draw [-, black!50] (node17) to (node35);
					\draw [-, black!50] (node18) to (node36);
					\draw [-, black!50] (node18) to (node37);
					\draw [-, black!50] (node19) to (node38);
					\draw [-, black!50] (node19) to (node39);
					\draw [-, black!50] (node20) to (node40);
					\draw [-, black!50] (node20) to (node41);
					\draw [-, black!50] (node21) to (node42);
					\draw [-, black!50] (node21) to (node43);
					\draw [-, black!50] (node22) to (node44);
					\draw [-, black!50] (node22) to (node45);
					\draw [-, black!50] (node23) to (node46);
					\draw [-, black!50] (node23) to (node47);
					\draw [-, black!50] (node24) to (node48);
					\draw [-, black!50] (node24) to (node49);
					\draw [-, black!50] (node25) to (node50);
					\draw [-, black!50] (node25) to (node51);
					\draw [-, black!50] (node26) to (node52);
					\draw [-, black!50] (node26) to (node53);
					\draw [-, black!50] (node27) to (node54);
					\draw [-, black!50] (node27) to (node55);
					\draw [-, black!50] (node28) to (node56);
					\draw [-, black!50] (node28) to (node57);
					\draw [-, black!50] (node29) to (node58);
					\draw [-, black!50] (node29) to (node59);
					\draw [-, black!50] (node30) to (node60);
					\draw [-, black!50] (node30) to (node61);
					\draw [-, black!50] (node31) to (node62);
					\draw [-, black!50] (node31) to (node63);
		
					\draw [-, black!50] (node0) to [out=-0,in=150] (8.00, 5.60) to [out=-30,in=100] (8.50, 4.50) to [out=-80,in=95] (9.00, 1.00) to [out=-85,in=85] (8.85, -6.50) to [out=-95,in=60] (node1a);				
					\draw [-, black!50] (node1a) to (node2a);
					\draw [-, black!50] (node1a) .. controls (8.50, 0.10 - 13.20) and (9.50, -0.10 - 13.20) .. (node3a);
		
					\draw [-, black!50] (node2a) to (node4a);
					\draw [-, black!50] (node2a) to (node5a);
					\draw [-, black!50] (node3a) to (node6a);
					\draw [-, black!50] (node3a) to (node7a);
		
					\draw [-, black!50] (node4a) to (node8a);
					\draw [-, black!50] (node4a) to (node9a);
					\draw [-, black!50] (node5a) to (node10a);
					\draw [-, black!50] (node5a) to (node11a);
					\draw [-, black!50] (node6a) to (node12a);
					\draw [-, black!50] (node6a) to (node13a);
					\draw [-, black!50] (node7a) to (node14a);
					\draw [-, black!50] (node7a) to (node15a);
		
					\draw [-, black!50] (node8a) to (node16a);
					\draw [-, black!50] (node8a) to (node17a);
					\draw [-, black!50] (node9a) to (node18a);
					\draw [-, black!50] (node9a) to (node19a);
					\draw [-, black!50] (node10a) to (node20a);
					\draw [-, black!50] (node10a) to (node21a);
					\draw [-, black!50] (node11a) to (node22a);
					\draw [-, black!50] (node11a) to (node23a);
					\draw [-, black!50] (node12a) to (node24a);
					\draw [-, black!50] (node12a) to (node25a);
					\draw [-, black!50] (node13a) to (node26a);
					\draw [-, black!50] (node13a) to (node27a);
					\draw [-, black!50] (node14a) to (node28a);
					\draw [-, black!50] (node14a) to (node29a);
					\draw [-, black!50] (node15a) to (node30a);
					\draw [-, black!50] (node15a) to (node31a);
		
					\draw [-, black!50] (node16a) to (node32a);
					\draw [-, black!50] (node16a) to (node33a);
					\draw [-, black!50] (node17a) to (node34a);
					\draw [-, black!50] (node17a) to (node35a);
					\draw [-, black!50] (node18a) to (node36a);
					\draw [-, black!50] (node18a) to (node37a);
					\draw [-, black!50] (node19a) to (node38a);
					\draw [-, black!50] (node19a) to (node39a);
					\draw [-, black!50] (node20a) to (node40a);
					\draw [-, black!50] (node20a) to (node41a);
					\draw [-, black!50] (node21a) to (node42a);
					\draw [-, black!50] (node21a) to (node43a);
					\draw [-, black!50] (node22a) to (node44a);
					\draw [-, black!50] (node22a) to (node45a);
					\draw [-, black!50] (node23a) to (node46a);
					\draw [-, black!50] (node23a) to (node47a);
					\draw [-, black!50] (node24a) to (node48a);
					\draw [-, black!50] (node24a) to (node49a);
					\draw [-, black!50] (node25a) to (node50a);
					\draw [-, black!50] (node25a) to (node51a);
					\draw [-, black!50] (node26a) to (node52a);
					\draw [-, black!50] (node26a) to (node53a);
					\draw [-, black!50] (node27a) to (node54a);
					\draw [-, black!50] (node27a) to (node55a);
					\draw [-, black!50] (node28a) to (node56a);
					\draw [-, black!50] (node28a) to (node57a);
					\draw [-, black!50] (node29a) to (node58a);
					\draw [-, black!50] (node29a) to (node59a);
					\draw [-, black!50] (node30a) to (node60a);
					\draw [-, black!50] (node30a) to (node61a);
					\draw [-, black!50] (node31a) to (node62a);
					\draw [-, black!50] (node31a) to (node63a);

					\draw [-] (node45) to [out=-90,in=90] (6.50, -0.75) to [out=-90,in=0] (6.00, -2.00) to [out=180,in=0] (2.23, -2.00);  
					\draw [dashed] (2.23, -2.00) -- (-0.29, -2.00);  

					\draw [-] (node45) to [out=-90,in=-90] (node44);  
					\draw [-] (node45) to [out=-90,in=90] (8.875, -6.50) to [out=-90,in=90] (node47a);  

					\draw [-] (node44) to [out=-90,in=0] ([shift={(-0.50, -1.50)}]node44) to [out=180,in=0] ([shift={(+0.50, -1.50)}]node32) to [out=180,in=-90] (node32);  
					\draw [-] (node44) to [out=-90,in=90] (6.00, -0.75) to [out=-90,in=0] (5.50, -1.75) to [out=180,in=0] (1.50, -1.75) to [out=180,in=90] (1.00, -2.75) to [out=-90,in=90] (node48);  
					\draw [-] (node47a) to [out=-90,in=0] ([shift={(-0.50, -1.50)}]node47a) to [out=180,in=0] ([shift={(+0.50, -1.50)}]node32a) to [out=180,in=-90] (node32a);  
					\draw [-] (node47a) to [out=-90,in=90] (7.50, -13.95) to [out=-90,in=0] (7.00, -14.95) to [out=180,in=0] (1.50, -14.95) to [out=180,in=90] (1.00, -15.95) to [out=-90,in=90] (node48a);  

					\draw [-] (node32) to [out=-90,in=-90] (node33);  
					\draw [-] (node32) to [out=-90,in=180] ([shift={(+0.50, -1.25)}]node32) to [out=0,in=180] ([shift={(-0.50, -1.25)}]node43) to [out=0,in=-90] (node43);  
					\draw [-] (node48) to [out=-90,in=-90] (node49);  
					\draw [-] (node48) to [out=-90,in=180] ([shift={(+0.50, -1.25)}]node48) to [out=0,in=180] ([shift={(-0.50, -1.25)}]node59) to [out=0,in=-90] (node59);  
					\draw [-] (node32a) to [out=-90,in=-90] (node33a);  
					\draw [-] (node32a) to [out=-90,in=180] ([shift={(+0.50, -1.25)}]node32a) to [out=0,in=180] ([shift={(-0.50, -1.25)}]node43a) to [out=0,in=-90] (node43a);  
					\draw [-] (node48a) to [out=-90,in=-90] (node49a);  
					\draw [-] (node48a) to [out=-90,in=180] ([shift={(+0.50, -1.25)}]node48a) to [out=0,in=180] ([shift={(-0.50, -1.25)}]node59a) to [out=0,in=-90] (node59a);  

					\draw [-] (node33) to [out=-90,in=-90] (node34);  
					\draw [-] (node33) to [out=-90,in=180] ([shift={(+0.50, -1.00)}]node33) to [out=0,in=180] ([shift={(-0.50, -1.00)}]node38) to [out=0,in=-90] (node38);  
					\draw [-] (node43) to [out=-90,in=-90] (node41);  
					\draw [-] (node43) to [out=-90,in=-90] (node47);  
					\draw [-] (node49) to [out=-90,in=-90] (node50);  
					\draw [-] (node49) to [out=-90,in=180] ([shift={(+0.50, -1.00)}]node49) to [out=0,in=180] ([shift={(-0.50, -1.00)}]node54) to [out=0,in=-90] (node54);  
					\draw [-] (node59) to [out=-90,in=-90] (node57);  
					\draw [-] (node59) to [out=-90,in=-90] (node61);  
					\draw [-] (node33a) to [out=-90,in=-90] (node34a);  
					\draw [-] (node33a) to [out=-90,in=180] ([shift={(+0.50, -1.00)}]node33a) to [out=0,in=180] ([shift={(-0.50, -1.00)}]node38a) to [out=0,in=-90] (node38a);  
					\draw [-] (node43a) to [out=-90,in=-90] (node41a);  
					\draw [-] (node43a) to [out=-90,in=-90] (node45a);  
					\draw [-] (node49a) to [out=-90,in=-90] (node50a);  
					\draw [-] (node49a) to [out=-90,in=180] ([shift={(+0.50, -1.00)}]node49a) to [out=0,in=180] ([shift={(-0.50, -1.00)}]node54a) to [out=0,in=-90] (node54a);  
					\draw [-] (node59a) to [out=-90,in=-90] (node57a);  
					\draw [-] (node59a) to [out=-90,in=-90] (node61a);  

					\draw [-] (node34) to [out=-90,in=-90] (node35);  
					\draw [-] (node34) to [out=-90,in=-90] (node36);  
					\draw [-] (node38) to [out=-90,in=-90] (node37);  
					\draw [-] (node38) to [out=-90,in=-90] (node39);  
					\draw [-] (node41) to [out=-90,in=-90] (node40);  
					\draw [-] (node41) to [out=-90,in=-90] (node42);  
					\draw [-] (node47) to [out=-90,in=-90] (node46);  
					\draw [-] (node47) to [out=-76,in=90] (node63);  
					\draw [-] (node50) to [out=-90,in=-90] (node51);  
					\draw [-] (node50) to [out=-90,in=-90] (node52);  
					\draw [-] (node54) to [out=-90,in=-90] (node53);  
					\draw [-] (node54) to [out=-90,in=-90] (node55);  
					\draw [-] (node57) to [out=-90,in=-90] (node56);  
					\draw [-] (node57) to [out=-90,in=-90] (node58);  
					\draw [-] (node61) to [out=-90,in=-90] (node60);  
					\draw [-] (node61) to [out=-90,in=-90] (node62);  
					\draw [-] (node34a) to [out=-90,in=-90] (node35a);  
					\draw [-] (node34a) to [out=-90,in=-90] (node36a);  
					\draw [-] (node38a) to [out=-90,in=-90] (node37a);  
					\draw [-] (node38a) to [out=-90,in=-90] (node39a);  
					\draw [-] (node41a) to [out=-90,in=-90] (node40a);  
					\draw [-] (node41a) to [out=-90,in=-90] (node42a);  
					\draw [-] (node45a) to [out=-90,in=-90] (node44a);  
					\draw [-] (node45a) to [out=-90,in=-90] (node46a);  
					\draw [-] (node50a) to [out=-90,in=-90] (node51a);  
					\draw [-] (node50a) to [out=-90,in=-90] (node52a);  
					\draw [-] (node54a) to [out=-90,in=-90] (node53a);  
					\draw [-] (node54a) to [out=-90,in=-90] (node55a);  
					\draw [-] (node57a) to [out=-90,in=-90] (node56a);  
					\draw [-] (node57a) to [out=-90,in=-90] (node58a);  
					\draw [-] (node61a) to [out=-90,in=-90] (node60a);  
					\draw [-] (node61a) to [out=-90,in=-90] (node62a);  
				\end{tikzpictureHostGraph}
			\end{comment:upperBound}
			\vspace*{-0.27cm}
			\caption[Arrangement $\phi$ for the guest graph $G = (V, E)$ depicted in Figure~\ref{figure:guestGraphGForHG6DepictingTheArrangementPhi} (binary regular tree of height $h_G = 6$) -- second part.]{Arrangement $\phi$ for the guest graph $G = (V, E)$ depicted in Figure~\ref{figure:guestGraphGForHG6DepictingTheArrangementPhi} -- second part. Its objective function value is $OV(G, 2, \phi_A) = 584$.}
			\label{figure:arrangementPhiForHG6SecondPart}
		\end{figure}
	\end{example}
	
\medskip	
	
	\section{The {\em k}-balanced partitioning problem}
		\label{section:approximationRatio:kBalancedPartitioningProblem}
In this section we introduce the $k$-balanced partitioning problem and a special case of it which 
will be involved in the analysis of the approximation algorithm for the $DAPT(G,2)$ with a binary regular tree $G$. 
\medskip
 
		\begin{definition}[{\bf {\em k}-balanced partitioning problem}]
			\label{definition:kBPP}
Given a graph $G = (V, E)$ with $|V| = n$ and $k \geq 2$, a {\bf {\em k}-balanced partition} is a partition 
of the vertex set $V$ into $k$ non-empty {\bf partition sets} $V_1 \neq \emptyset$, $V_2 \neq \emptyset$, \ldots, 
$V_k \neq \emptyset$, where $\cup_{i = 1}^{k}{V_k} = V$, $V_i \cap V_j = \emptyset$ for every $i \neq j$ and 
$|V_i| \leq \left\lceil \frac{n}{k} \right\rceil$ for all $1 \leq i \leq k$. 
The {\bf {\em k}-balanced partitioning problem ({\em k}-BPP)} asks for a $k$-balanced partition $\mathscr{V}$ which minimises
\begin{equation}
\label{equation:definition:kBPP:oV}
c(G, \mathscr{V}) \defeq \Big|\big\{(u, v) \in E | u \in V_i \text{, } v \in V_j \text{, } i \neq j\big\}\Big|,
\end{equation}
where $\mathscr{V} \defeq \{V_i | 1 \leq i \leq k\}$.
\end{definition}

$k$-BPP is a well known ${\cal NP}$-hard problem (for $k = 2$ we get the {\em minimum bisection problem} which 
is ${\cal NP}$-hard, see {\sc Garey} and 
{\sc Johnson}~\cite{GareyJohnson:ComputersAndIntractabilityAGuideToTheTheoryOfNPCompleteness}). 
A lot of work has been done focusing on the computational complexity of the $k$-BPP. 
{\sc Andreev} and {\sc R\"{a}cke} proved further complexity results for a generalization allowing near-balanced 
partitions~\cite{AndreevRacke:BalancedGraphPartitioning}. {\sc Krauthgamer}, {\sc Naor} and {\sc Schwartz}
 provide an approximation algorithm achieving an approximation of 
$O(\sqrt{\log{n} \log{k}})$~\cite{KrauthgamerNaorSchwartz:PartitioningGraphsIntoBalancedComponents}. 
And finally, {\sc Feldmann} and {\sc Foschini} proved that the $k$-BPP remains $APX$-hard even if the graph $G$ is 
restricted to be an unweighted tree with constant maximum degree~\cite{FeldmannFoschini:BalancedPartitionsOfTreesAndApplications}.
		
We deal with a special case of this problem where $G = (V, E)$ is a binary regular tree of height $h \geq 1$ and 
where $k = 2^{k^\prime}$ and $1 \leq k^\prime \leq h$. 
The following facts are obvious:
\begin{observation}	\label{observation:kBalancedPartitioningProblemOnBinaryRegularTrees:someObviousEqualities}
Let $G = (V, E)$ be a binary regular tree of height $h \geq 1$ with $n = 2^{h + 1} - 1$ vertices. 
Let
$\mathscr{V} = \{V_1, V_2, \ldots, V_k\}$ be a $k$-balanced partition with $k = 2^{k^\prime}$ and  $1 \leq k^\prime \leq h$. 
Then		
one of the partition sets in $\mathscr{V}$    has   $n_s \defeq  \frac{n + 1}{k} - 1$ elements and is called 
{\bf the small partition set}. 
All other partition sets have  $n_b \defeq \frac{n + 1}{k}$ elements and are called {\bf big partition sets}. Moreover the following equalities clearly hold
\begin{equation} \label{equation:kBalancedPartitioningProblemOnBinaryRegularTrees:someObviousEqualities}
n_s= 2^{h - k^\prime + 1} - 1 \mbox{ and }	n_b= 2^{h - k^\prime + 1}\, .
\end{equation}
\end{observation}
%

The rest of this section is structured as follows: 
In Section~\ref{subsection:section:approximationRatio:kBalancedPartitioningProblem:solutionAlgorithmKBPP} 
we introduce
 an algorithm to construct  an optimal $2^{k^{\prime}}$-balanced partition $\mathscr{V}^*$ in a regular binary tree.  The  optimality is proven  in 
Section~\ref{subsection:section:approximationRatio:kBalancedPartitioningProblem:optimalityProofKBPP}. 
 
Section~\ref{subsection:section:approximationRatio:kBalancedPartitioningProblem:objectiveValueKBPP}
 provides a lower bound on the optimal value   $c(G, k, \mathscr{V}^*)$ of the objective function of the $2^{k^{\prime}}$-BPP in a regular binary tree.
		
\subsection{\boldmath A solution algorithm for the $2^{k^{\prime}}$-BPP on  regular binary trees\unboldmath}
\label{subsection:section:approximationRatio:kBalancedPartitioningProblem:solutionAlgorithmKBPP}
The  algorithm consists of three simple steps. 
Let $t \defeq h  - k^\prime + 2$ and $e \defeq \left\lfloor\frac{h + 1}{t}\right\rfloor - 1$.
\begin{enumerate}
\item First, we partition the tree $G$ by cutting all edges $(u, v) \in E$ with $\mbox{level}(u) = h - i t$ and 
$\mbox{level}(v) = h - i t + 1$, where $1 \leq i \leq e$. Roughly spoken, we separate $e$ horizontal bands of height $t - 1$ 
from the input tree $G$, from the bottom to the top. The height of  the remaining top part is then $\widehat{h}$  with 
 $t - 1 \leq \widehat{h} \leq 2 t - 2$. 
Let $p$ be the number of binary regular tress of height $t - 1$ contained in these bands.
\item Next consider the binary regular trees contained in the above mentioned bands  and cut all edges connecting 
their roots with their  right children, respectively.
After that each root remains connected to the corresponding left basic subtree, thus forming a big partition set,  since each  root and its left basic subbtree 
 tree  have    $2^{t - 2 + 1} - 1 + 1 = 
2^{h - k^\prime + 1} = n_b$ 
vertices altogether. 

On the other hand each of the   right  basic subtrees mentioned above  has  $2^{t - 2 + 1} - 1 = 
2^{h - k^\prime + 1} - 1 = n_b - 1$ vertices 
and   needs one more vertex in order to form a big partition set. 
Let $q \defeq \frac{2^{h - k^\prime + 1} - 1}{2^{h - k^\prime + 1}} p$. We split $p - q$ of the right basic subtrees into isolated vertices, thus obtaining $(p-q)(2^{h-k'+1} - 1)=\left(p-\frac{2^{h-k'+1}-1}{2^{h-k'+1}}p\right)\\(2^{h-k'+1}-1)=q$ isolated vertices. 
Each of them is paired  with the remaining $q$ non-split right basic subtrees in order to obtain further big partition sets. 
It is not difficult to check that $q\in \nz$.
\item Finally, let us consider the top part consisting of a binary regular tree of height $\widehat{h}$, 
$t - 1 \leq \widehat{h} \leq 2 t - 2$. 
We cut all edges $(u, v) \in E$ with $\mbox{level}(u) = h - (e + 1) t + 1$ and   $v$ being  the right child of $u$. 
Analogously as above we  obtain one big partition set for every vertex $u \in V$ with $\mbox{level}(u) = h - (e + 1) t + 1$
 together with its corresponding  left basic subtree. Moreover, each of the remaining right basic 
subtrees of the vertices $u$ as above  can be paired with one of the  vertices $u^\prime \in V$ with 
$\mbox{level}(u^\prime) < h - (e + 1) t + 1$,
(roughly spoken, these are  the vertices lying on the very top of the tree) in order to obtain further big partition sets. 
Again  simple computations show  that the  number of the  right basic subtrees mentioned above exceeds the number of the remaining 
vertices by exactly one.
Hence just  one right basic subtree of one vertex $u \in V$ with $\mbox{level}(u) = h - (e + 1) t + 1$ remains unpaired; 
this subtree builds the small partition set.
\end{enumerate}
\medskip

The following example illustrates this algorithm.
\begin{example}
				\label{example:kBPP:solutionAlgorithm}
Let us consider a binary regular tree $G = (V, E)$ of height $h = 5$ depicted in Figure~\ref{figure:16BalancedPartitionForH5} 
and let $k = 2^4 = 16$, i.e.\ $k^\prime = 4$.
\begin{figure}[htb]
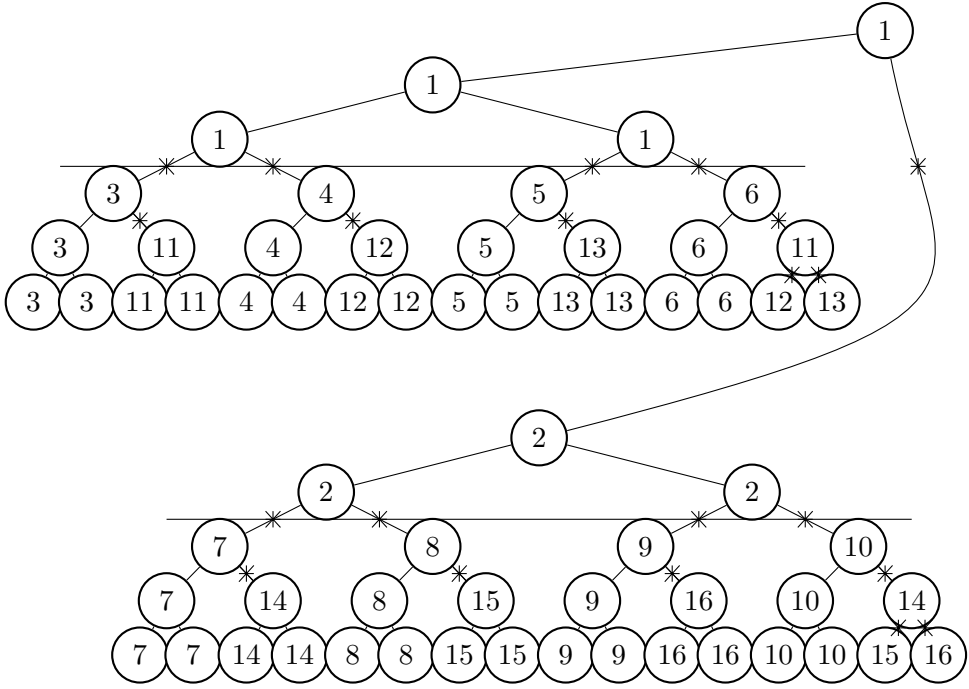

\centering
					\begin{comment:kBPP:solutionAlgorithm}
						\begin{tikzpictureGuestGraph}
							\node[node] (node1) at (8.00, 5.00) {1};
							
							\node[node] (node2) at (3.75, 4.00) {1};
				
							\node[node] (node4) at (1.75, 3.00) {1};
							\node[node] (node5) at (5.75, 3.00) {1};
				
							\node[node] (node8) at (0.75, 2.00) {3};
							\node[node] (node9) at (2.75, 2.00) {4};
							\node[node] (node10) at (4.75, 2.00) {5};
							\node[node] (node11) at (6.75, 2.00) {6};
				
							\node[node] (node16) at (0.25, 1.00) {3};
							\node[node] (node17) at (1.25, 1.00) {11};
							\node[node] (node18) at (2.25, 1.00) {4};
							\node[node] (node19) at (3.25, 1.00) {12};
							\node[node] (node20) at (4.25, 1.00) {5};
							\node[node] (node21) at (5.25, 1.00) {13};
							\node[node] (node22) at (6.25, 1.00) {6};
							\node[node] (node23) at (7.25, 1.00) {11};
				
							\node[node] (node32) at (0.00, 0.00) {3};
							\node[node] (node33) at (0.50, 0.00) {3};
							\node[node] (node34) at (1.00, 0.00) {11};
							\node[node] (node35) at (1.50, 0.00) {11};
							\node[node] (node36) at (2.00, 0.00) {4};
							\node[node] (node37) at (2.50, 0.00) {4};
							\node[node] (node38) at (3.00, 0.00) {12};
							\node[node] (node39) at (3.50, 0.00) {12};
							\node[node] (node40) at (4.00, 0.00) {5};
							\node[node] (node41) at (4.50, 0.00) {5};
							\node[node] (node42) at (5.00, 0.00) {13};
							\node[node] (node43) at (5.50, 0.00) {13};
							\node[node] (node44) at (6.00, 0.00) {6};
							\node[node] (node45) at (6.50, 0.00) {6};
							\node[node] (node46) at (7.00, 0.00) {12};
							\node[node] (node47) at (7.50, 0.00) {13};
				
							\node[node] (node3) at (4.75, -2.50) {2};
				
							\node[node] (node6) at (2.75, -3.50) {2};
							\node[node] (node7) at (6.75, -3.50) {2};
				
							\node[node] (node12) at (1.75, -4.50) {7};
							\node[node] (node13) at (3.75, -4.50) {8};
							\node[node] (node14) at (5.75, -4.50) {9};
							\node[node] (node15) at (7.75, -4.50) {10};
				
							\node[node] (node24) at (1.25, -5.50) {7};
							\node[node] (node25) at (2.25, -5.50) {14};
							\node[node] (node26) at (3.25, -5.50) {8};
							\node[node] (node27) at (4.25, -5.50) {15};
							\node[node] (node28) at (5.25, -5.50) {9};
							\node[node] (node29) at (6.25, -5.50) {16};
							\node[node] (node30) at (7.25, -5.50) {10};
							\node[node] (node31) at (8.25, -5.50) {14};
				
							\node[node] (node48) at (1.00, -6.50) {7};
							\node[node] (node49) at (1.50, -6.50) {7};
							\node[node] (node50) at (2.00, -6.50) {14};
							\node[node] (node51) at (2.50, -6.50) {14};
							\node[node] (node52) at (3.00, -6.50) {8};
							\node[node] (node53) at (3.50, -6.50) {8};
							\node[node] (node54) at (4.00, -6.50) {15};
							\node[node] (node55) at (4.50, -6.50) {15};
							\node[node] (node56) at (5.00, -6.50) {9};
							\node[node] (node57) at (5.50, -6.50) {9};
							\node[node] (node58) at (6.00, -6.50) {16};
							\node[node] (node59) at (6.50, -6.50) {16};
							\node[node] (node60) at (7.00, -6.50) {10};
							\node[node] (node61) at (7.50, -6.50) {10};
							\node[node] (node62) at (8.00, -6.50) {15};
							\node[node] (node63) at (8.50, -6.50) {16};
				
							\draw (node1) to (node2);
							\draw (node1) .. controls (8.50, 0.10) and (9.50, -0.10) .. (node3);
				
							\draw (node2) to (node4);
							\draw (node2) to (node5);
							\draw (node3) to (node6);
							\draw (node3) to (node7);
				
							\draw (node4) to (node8);
							\draw (node4) to (node9);
							\draw (node5) to (node10);
							\draw (node5) to (node11);
							\draw (node6) to (node12);
							\draw (node6) to (node13);
							\draw (node7) to (node14);
							\draw (node7) to (node15);
				
							\draw (node8) to (node16);
							\draw (node8) to (node17);
							\draw (node9) to (node18);
							\draw (node9) to (node19);
							\draw (node10) to (node20);
							\draw (node10) to (node21);
							\draw (node11) to (node22);
							\draw (node11) to (node23);
							\draw (node12) to (node24);
							\draw (node12) to (node25);
							\draw (node13) to (node26);
							\draw (node13) to (node27);
							\draw (node14) to (node28);
							\draw (node14) to (node29);
							\draw (node15) to (node30);
							\draw (node15) to (node31);
				
							\draw (node16) to (node32);
							\draw (node16) to (node33);
							\draw (node17) to (node34);
							\draw (node17) to (node35);
							\draw (node18) to (node36);
							\draw (node18) to (node37);
							\draw (node19) to (node38);
							\draw (node19) to (node39);
							\draw (node20) to (node40);
							\draw (node20) to (node41);
							\draw (node21) to (node42);
							\draw (node21) to (node43);
							\draw (node22) to (node44);
							\draw (node22) to (node45);
							\draw (node23) to (node46);
							\draw (node23) to (node47);
							\draw (node24) to (node48);
							\draw (node24) to (node49);
							\draw (node25) to (node50);
							\draw (node25) to (node51);
							\draw (node26) to (node52);
							\draw (node26) to (node53);
							\draw (node27) to (node54);
							\draw (node27) to (node55);
							\draw (node28) to (node56);
							\draw (node28) to (node57);
							\draw (node29) to (node58);
							\draw (node29) to (node59);
							\draw (node30) to (node60);
							\draw (node30) to (node61);
							\draw (node31) to (node62);
							\draw (node31) to (node63);
							
							\draw (0.25, 2.50) -- (7.25, 2.50);
							\draw (1.25, -4.00) -- (8.25, -4.00);

	\draw[shift={(1.25, 2.50)}] (-0.07,-0.15) -- (+0.07,+0.15) (-0.07,+0.15) -- (+0.07,-0.15) (-0.07, 0.00) -- (+0.07, 0.00) (0.00, -0.15) -- (0.00, +0.15);
	\draw[shift={(2.25, 2.50)}] (-0.07,-0.15) -- (+0.07,+0.15) (-0.07,+0.15) -- (+0.07,-0.15) (-0.07, 0.00) -- (+0.07, 0.00) (0.00, -0.15) -- (0.00, +0.15);
							\draw[shift={(5.25, 2.50)}] (-0.07,-0.15) -- (+0.07,+0.15) (-0.07,+0.15) -- (+0.07,-0.15) (-0.07, 0.00) -- (+0.07, 0.00) (0.00, -0.15) -- (0.00, +0.15);
							\draw[shift={(6.25, 2.50)}] (-0.07,-0.15) -- (+0.07,+0.15) (-0.07,+0.15) -- (+0.07,-0.15) (-0.07, 0.00) -- (+0.07, 0.00) (0.00, -0.15) -- (0.00, +0.15);
							\draw[shift={(2.25, -4.00)}] (-0.07,-0.15) -- (+0.07,+0.15) (-0.07,+0.15) -- (+0.07,-0.15) (-0.07, 0.00) -- (+0.07, 0.00) (0.00, -0.15) -- (0.00, +0.15);
							\draw[shift={(3.25, -4.00)}] (-0.07,-0.15) -- (+0.07,+0.15) (-0.07,+0.15) -- (+0.07,-0.15) (-0.07, 0.00) -- (+0.07, 0.00) (0.00, -0.15) -- (0.00, +0.15);
							\draw[shift={(6.25, -4.00)}] (-0.07,-0.15) -- (+0.07,+0.15) (-0.07,+0.15) -- (+0.07,-0.15) (-0.07, 0.00) -- (+0.07, 0.00) (0.00, -0.15) -- (0.00, +0.15);
							\draw[shift={(7.25, -4.00)}] (-0.07,-0.15) -- (+0.07,+0.15) (-0.07,+0.15) -- (+0.07,-0.15) (-0.07, 0.00) -- (+0.07, 0.00) (0.00, -0.15) -- (0.00, +0.15);

							\draw[shift={(1.00, 1.50)}] (-0.07,-0.15) -- (+0.07,+0.15) (-0.07,+0.15) -- (+0.07,-0.15) (-0.07, 0.00) -- (+0.07, 0.00) (0.00, -0.15) -- (0.00, +0.15);
							\draw[shift={(3.00, 1.50)}] (-0.07,-0.15) -- (+0.07,+0.15) (-0.07,+0.15) -- (+0.07,-0.15) (-0.07, 0.00) -- (+0.07, 0.00) (0.00, -0.15) -- (0.00, +0.15);
							\draw[shift={(5.00, 1.50)}] (-0.07,-0.15) -- (+0.07,+0.15) (-0.07,+0.15) -- (+0.07,-0.15) (-0.07, 0.00) -- (+0.07, 0.00) (0.00, -0.15) -- (0.00, +0.15);
							\draw[shift={(7.00, 1.50)}] (-0.07,-0.15) -- (+0.07,+0.15) (-0.07,+0.15) -- (+0.07,-0.15) (-0.07, 0.00) -- (+0.07, 0.00) (0.00, -0.15) -- (0.00, +0.15);
							\draw[shift={(2.00, -5.00)}] (-0.07,-0.15) -- (+0.07,+0.15) (-0.07,+0.15) -- (+0.07,-0.15) (-0.07, 0.00) -- (+0.07, 0.00) (0.00, -0.15) -- (0.00, +0.15);
							\draw[shift={(4.00, -5.00)}] (-0.07,-0.15) -- (+0.07,+0.15) (-0.07,+0.15) -- (+0.07,-0.15) (-0.07, 0.00) -- (+0.07, 0.00) (0.00, -0.15) -- (0.00, +0.15);
							\draw[shift={(6.00, -5.00)}] (-0.07,-0.15) -- (+0.07,+0.15) (-0.07,+0.15) -- (+0.07,-0.15) (-0.07, 0.00) -- (+0.07, 0.00) (0.00, -0.15) -- (0.00, +0.15);
							\draw[shift={(8.00, -5.00)}] (-0.07,-0.15) -- (+0.07,+0.15) (-0.07,+0.15) -- (+0.07,-0.15) (-0.07, 0.00) -- (+0.07, 0.00) (0.00, -0.15) -- (0.00, +0.15);
							\draw[shift={(7.125, 0.50)}] (-0.07,-0.15) -- (+0.07,+0.15) (-0.07,+0.15) -- (+0.07,-0.15) (-0.07, 0.00) -- (+0.07, 0.00) (0.00, -0.15) -- (0.00, +0.15);
							\draw[shift={(7.375, 0.50)}] (-0.07,-0.15) -- (+0.07,+0.15) (-0.07,+0.15) -- (+0.07,-0.15) (-0.07, 0.00) -- (+0.07, 0.00) (0.00, -0.15) -- (0.00, +0.15);
							\draw[shift={(8.125, -6.00)}] (-0.07,-0.15) -- (+0.07,+0.15) (-0.07,+0.15) -- (+0.07,-0.15) (-0.07, 0.00) -- (+0.07, 0.00) (0.00, -0.15) -- (0.00, +0.15);
							\draw[shift={(8.375, -6.00)}] (-0.07,-0.15) -- (+0.07,+0.15) (-0.07,+0.15) -- (+0.07,-0.15) (-0.07, 0.00) -- (+0.07, 0.00) (0.00, -0.15) -- (0.00, +0.15);

\draw[shift={(8.31, 2.50)}] (-0.07,-0.15) -- (+0.07,+0.15) (-0.07,+0.15) -- (+0.07,-0.15) (-0.07, 0.00) -- (+0.07, 0.00) (0.00, -0.15) -- (0.00, +0.15);
\end{tikzpictureGuestGraph}
\end{comment:kBPP:solutionAlgorithm}
\caption{$16$-balanced partition $\mathscr{V}^*$. Its objective value is $c(G, \mathscr{V}^*) = 21$. 
The  numbers on the vertices indicate the index of the partition set to which the  corresponding vertex belongs.}
\label{figure:16BalancedPartitionForH5}
\end{figure}
				
We have $t = h- k^\prime + 2 = 5 - 4 + 2 = 3$ and 
$e = \left\lfloor\frac{h + 1}{t}\right\rfloor - 1 = \left\lfloor\frac{5 + 1}{3}\right\rfloor - 1 = 1$.
\begin{enumerate}
\item Thus $i = 1$ and we cut  all edges $(u, v) \in E$ with $\mbox{level}(u) = h - i t = 5 - 1 \cdot 3 = 2$ and 
$\mbox{level}(v) = h - i t + 1 = 5 - 1 \cdot 3 + 1 = 3$, i.e.\ the edges cut by the two horizontal lines 
in Figure~\ref{figure:16BalancedPartitionForH5}.
\item Now,  consider the bottom band consisting of $p = 8$ binary regular trees of height $t - 1 = 3 - 1 = 2$ each and in each of them 
 cut all edges
 connecting the root with the right child. Each  root connected to its  corresponding left childs form a big partition set; we obtain
the big partition sets $V_i$, $3\le i\le 10$ depicted in  Figure~\ref{figure:16BalancedPartitionForH5}.
Notice that in Figure~\ref{figure:16BalancedPartitionForH5} a partition set $V_i$ contains the vertices 
marked by $i$, $1\le i\le 16$. 
 
  Set
$q \defeq \frac{2^{h - k^\prime + 1} - 1}{2^{h - k^\prime + 1}} p = \frac{2^{5 - 4 + 1} - 1}{2^{5 - 4 + 1}} 8 = 6$, 
and  cut all edges of  $p - q = 8 - 6 = 2$ arbitrarily  chosen right basic subtrees.  Pair each of the thereby arising 
 isolated vertices with the remaining $q = 6$ right basic subtrees to  obtain the big partition sets $V_i$ $11\le i\le 16$.
\item Finally, notice that  $h - (e + 1) t + 1 = 5 - (1 + 1) 3 + 1 = 0$. 
Thus we cut the edge connecting the root $v_1$ 
(note that $\mbox{level}(v_1) = 0$) with its right child according to the third step of the algorithm. 
We obtain the big partition sets $V_1$ and the small partition set  $V_2$ depicted in Figure~\ref{figure:16BalancedPartitionForH5}.
\end{enumerate}
				
The objective function value of this $16$-balanced partition $\mathscr{V}^* = \{V_1, V_2, \ldots, V_{16}\}$, i.e.\ the number of the cut 
edges,  equals   $c(G, \mathscr{V}^*) = 21$. By applying the results of the following section we conclude that this is the optimal 
$16$-balanced partition of $G$.
			\end{example}
			
		\subsection{Proof of the optimality for the algorithm described in 
Section~\ref{subsection:section:approximationRatio:kBalancedPartitioningProblem:solutionAlgorithmKBPP}}
			\label{subsection:section:approximationRatio:kBalancedPartitioningProblem:optimalityProofKBPP}
The optimality proof will make use of the  following reformulation  of the $k$-BPP.
			
Consider the input  graph $G = (V, E)$ of the $k$-BPP,  a  $k$-balanced partition $\mathscr{V} = V_1, V_2, \ldots, V_k$,
 and the respective induced subgraphs $G[V_i]$, $1\le i\le k$.
 Assume that $G[V_i]$ has $l_i$ connected components   $G_{i, j} = (V_{i, j}, E_{i, j})$ for $1\le j\le l_i$, for $1\le i\le k$.  
Define a new graph $G^\prime = (V^\prime, E^\prime)$ 
which contains one representative vertex $\bar{v}_{i,j}$ for each connected component $G_{i,j}$, $1\le i\le k$, $1\le j\le l_i$. 
Two vertices $\bar{v}_{i_1,j_1}$ and $\bar{v}_{i_2,j_2}$ are connected in $G'$ iff the connected components $G_{i_1,j_1}$, $G_{i_2,j_2}$
 are connected by an edge in $G$. Observe that if $G$ is a tree, then $G'$ is also a tree and the following equality holds

\begin{equation}
\label{equation:transformationKBPP}
 c(G, \mathscr{V})=|E^\prime|= |V^\prime| - 1\, .
\end{equation}

Thus the value of the objective function of the $k$-BPP corresponding to a $k$-balanced partition $\mathscr{V}$ of a tree $G$
  equals  the overall  number of the connected components of the subgraphs induced in $G$ by  the partition sets 
of $\mathscr{V}$ minus $1$. Hence the goal  of the $k$-BPP can be  rephrased as follows: Determine a $k$-balanced partition $\mathscr{V}^{\ast}$
such that the number of the connected components of   the subgraphs induced in $G$ by  the partition sets is minimised.	
			\begin{example}
				\label{example:transformationKBPP}
Let us consider the graph and the partition depicted in Figure~\ref{figure:16BalancedPartitionForH5}. 
The tree $G^\prime$ is depicted  in Figure~\ref{figure:graphGPrimeObtainedByTheTransformationOfTheGraphGAndThePartitionV}.
				\begin{figure}[htb]
					\centering
					\begin{comment:kBPP:optimalityProof}
						\begin{tikzpicture}[
								xscale=1.05, yscale=1.05,
								node/.style={circle, draw=black!100, fill=white!100, thick, inner sep=0pt, minimum size=9.5mm},
							]
							\node[node] (node1) at (1.50, 4.00) {$G_{1, 1}$};
							
							\node[node] (node2) at (5.50, 3.00) {$G_{2, 1}$};
							
							\node[node] (node3) at (0.00, 2.00) {$G_{3, 1}$};
							\node[node] (node4) at (1.00, 2.00) {$G_{4, 1}$};
							\node[node] (node5) at (2.00, 2.00) {$G_{5, 1}$};
							\node[node] (node6) at (3.00, 2.00) {$G_{6, 1}$};
							\node[node] (node7) at (4.00, 2.00) {$G_{7, 1}$};
							\node[node] (node8) at (5.00, 2.00) {$G_{8, 1}$};
							\node[node] (node9) at (6.00, 2.00) {$G_{9, 1}$};
							\node[node] (node10) at (7.00, 2.00) {$G_{10, 1}$};
							
							\node[node] (node11) at (0.50, 1.00) {$G_{11, 1}$};
							\node[node] (node12) at (1.50, 1.00) {$G_{12, 1}$};
							\node[node] (node13) at (2.50, 1.00) {$G_{13, 1}$};
							\node[node] (node14) at (3.50, 1.00) {$G_{11, 2}$};
							\node[node] (node15) at (4.50, 1.00) {$G_{14, 1}$};
							\node[node] (node16) at (5.50, 1.00) {$G_{15, 1}$};
							\node[node] (node17) at (6.50, 1.00) {$G_{16, 1}$};
							\node[node] (node18) at (7.50, 1.00) {$G_{14, 2}$};
							
							\node[node] (node19) at (3.00, 0.00) {$G_{12, 2}$};
							\node[node] (node20) at (4.00, 0.00) {$G_{13, 2}$};
							\node[node] (node21) at (7.00, 0.00) {$G_{15, 2}$};
							\node[node] (node22) at (8.00, 0.00) {$G_{16, 2}$};
							
							\draw (node1) to (node2);
							\draw (node1) to (node3);
							\draw (node1) to (node4);
							\draw (node1) to (node5);
							\draw (node1) to (node6);

							\draw (node2) to (node7);
							\draw (node2) to (node8);
							\draw (node2) to (node9);
							\draw (node2) to (node10);

							\draw (node3) to (node11);
							\draw (node4) to (node12);
							\draw (node5) to (node13);
							\draw (node6) to (node14);
							\draw (node7) to (node15);
							\draw (node8) to (node16);
							\draw (node9) to (node17);
							\draw (node10) to (node18);

							\draw (node14) to (node19);
							\draw (node14) to (node20);
							\draw (node18) to (node21);
							\draw (node18) to (node22);
						\end{tikzpicture}
					\end{comment:kBPP:optimalityProof}
					\caption{The graph $G^\prime$ corresponding to  the regular binary $G$ and the $16$-partition $\mathscr{V}$ 
depicted in Figure~\ref{figure:16BalancedPartitionForH5}.}
					\label{figure:graphGPrimeObtainedByTheTransformationOfTheGraphGAndThePartitionV}
				\end{figure}
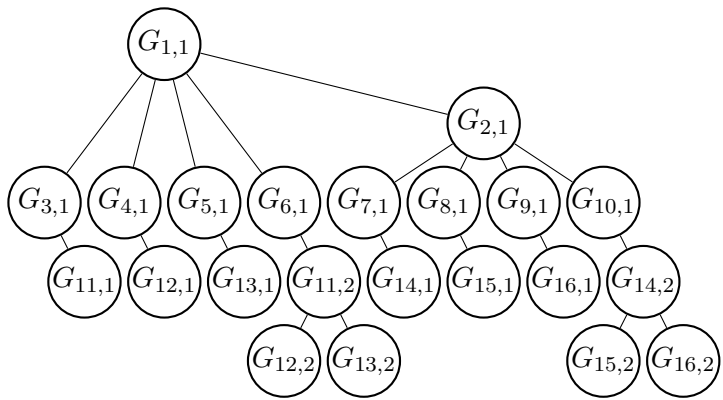
				
	The partition sets $V_i$, $1\le i\le 10$, generate one connected component each,    the partition sets $V_i$, $11\le i\le 16$ 
generate  two connected components each. Notice that equality~(\ref{equation:transformationKBPP})
is fulfilled:  the tree $G'$ has $22$ vertices and $21$ edges and  $c(G,\mathscr{V})=21$, see Example~\ref{example:kBPP:solutionAlgorithm}.
\end{example}

Denote by $n_i(\mathscr{V})$ be the number  of partition sets in   $\mathscr{V}$ 
which induce  $i$ connected components each in $G$, for $i \in \nz$. 
Notice now that the following observation holds.
\begin{observation}\label{monotonicity}
Consider the $2^{k'}$-BPP for a regular binary tree $G$ of height $h$, $h\ge k'$, and let $\mathscr{V}^*$
be the $k$-balanced partition computed by the algorithm described  in 
Section~\ref{subsection:section:approximationRatio:kBalancedPartitioningProblem:solutionAlgorithmKBPP}. 
 Then $n_1(\mathscr{V}^{\ast})\ge n_1(\mathscr{V})$ implies $c(G,\mathscr{V}^{\ast})\le c(G,\mathscr{V}^{\ast})$, for any $k$-balanced partition $\mathscr{V}$ of $G$. 
\end{observation}
\begin{proof}
As argued above for every $k$-balanced partition $\mathscr{V}$ the value $c(G,\mathscr{V})$ of the 
objective function equals the overall number of connected componets induced in $G$ by the partition sets 
of $\mathscr{V}$ minus 1, and hence $c(G,\mathscr{V})=\sum_{i\in \nz}in_i(\mathscr{V}) - 1$ holds. 
If $n_1(\mathscr{V}^{\ast})\ge n_1(\mathscr{V})$ and since $n_i(\mathscr{V}^{\ast})=0$ for $i\ge 3$,
 the following equalities and inequalities hold:
\[
c(G,\mathscr{V})= \sum_{i\in \nz}in_i(\mathscr{V})-1\ge n_1(\mathscr{V})+2\big(k-n_1(\mathscr{V})\big)-1=2k-n_1(\mathscr{V})-1\ge \]
\begin{equation}
\label{reni3} 2k-n_1(\mathscr{V}^{\ast})-1  = n_1(\mathscr{V}^{\ast})+2(k-n_1(\mathscr{V}^{\ast}))-1=
n_1(\mathscr{V}^{\ast})+2n_2(\mathscr{V}^{\ast})-1=c(G,\mathscr{V}^{\ast})\, .
\end{equation}
\end{proof}

\begin{theorem}
\label{theorem:optimalityProofForAlgorithmSolutionAlgorithmKBPPMainAlgorithm}
Let $G = (V, E)$ be a binary regular tree of height $h \geq 1$ and let $k = 2^{k^{\prime}}$, 
where $1 \leq k^{\prime} \leq h$. Then the  algorithm presented in 
Section~\ref{subsection:section:approximationRatio:kBalancedPartitioningProblem:solutionAlgorithmKBPP}  
yields an optimal $k$-balanced partition 
$\mathscr{V}^{\ast}$.
\end{theorem}
			
\begin{proof}
Due to Observation~\ref{monotonicity} it is enough to show that 
$n_1(\mathscr{V}^{\ast})\ge n_1(\mathscr{V})$ for any $k$-balanced partition $\mathscr{V}$ of $G$.

 Let us first show that  every $k$-balanced partition   $\mathscr{V} \eqdef \mathscr{V}_0$ can be transformed step by step into a sequence $\mathscr{V} \eqdef \mathscr{V}_0, \mathscr{V}_1,\ldots, \mathscr{V}_l=\mathscr{V}^{\prime}$, $l\in \nz$, of 
$k$-balanced partitions  with the following properties:
\begin{itemize}
\item[(a)]
 $n_1(\mathscr{V}_t)\ge n_1(\mathscr{V}_{t-1})$ holds for every $t\in\{1,2,\ldots,l\}$, and  
\item[(b)] the   big partition sets  of $\mathscr{V}^{\prime}$ which induce  one connected component in $G$ each
coincide 
with the big partition sets of $\mathscr{V}^{\ast}$ which induce one connected  component in $G$ each.
\end{itemize}

In the following the steps of this transformation are explained.  
Consider a $k$-balanced partition   $\mathscr{V}$ whose big partition sets which induce one connected 
component in $G$ each do 
not 
coincide with the corresponding partition sets of $\mathscr{V}^{\ast}$. 

For every vertex $v\in V$ let $P_{\mathscr{V}}(v)$  be the uniquely determined partition set in $\mathscr{V}$ such 
that $v\in P_{\mathscr{V}}(v)$. Consider the  vertices  $v \in V$ with $\mbox{level}(v) = h - it + 1$, where $1\le i\le e$,
 $e=\lfloor \frac{h+1}{t}\rfloor -1$, 
 and choose among them a vertex with the largest level such that $P_{\mathscr{V}}(v)\neq P_{\mathscr{V}^{\ast}}(v)$.
Perform now the following transformation steps.
\begin{description}
\item{Case 1.} If the partition set $P_{\mathscr{V}}(v)$ consists of the vertex $v \in V$ together with the regular 
binary subtree of height $h - k^\prime$ rooted at its right child, 
then exchange   the subtrees rooted at the left and the  right child of $v$, respectively,  to obtain a 
new 
$k$-balanced partition $\mathscr{V}^{\prime}$ for which obviously $n_1(\mathscr{V})=n_1(\mathscr{V}^{\prime})$ holds. 
\item{Case 2.} If   Case 1 does not arise, then $P_{\mathscr{V}}(v)$ contains neither the regular 
binary subtree of height $h - k^\prime$ rooted at the  right child of $v$ nor  the regular 
binary subtree of height $h - k^\prime$ rooted at its left child. At least one of these two subtrees does not build a 
small partition set in $\mathscr{V}$. Let this be the left subtree (otherwise we would apply an exchange of the two 
subtrees as in Case 1). Denote this subtree by $T$. 
If $P_{\mathscr{V}}(v)$ is a big component, then
  exchange the  vertices contained in  $P_{\mathscr{V}}(v)$ 
and   the vertices of  $T$  (the one by one assignment 
of the corresponding vertices is done arbitrarily).  
Denote  the resulting balanced partition by  $\mathscr{V}^{\prime}$.
Clearly $P_{\mathscr{V}^{\prime}}(v)$ induces one connected component in $G$. 
Moreover $P_{\mathscr{V}}(u)$ induces more than one connected component in $G$, for every $u\in  V(T)$, because  $T$ does 
not  build a small partition set in $\mathscr{V}$. Then $n_1(\mathscr{V}^{\prime})\ge n_1(\mathscr{V})$ holds 
(no partitition sets inducing one connected component in $G$ are destroied).
 If $P_{\mathscr{V}}(v)$ is the small component, then again  exchange the vertices contained in $P_{\mathscr{V}}(v)$ 
against  all but one of  the vertices in  the subtree $T$ of  height $h - k^\prime$ rooted at the  left child of $v$ (the one 
by one assignment 
of the corresponding vertices is again done arbitrarily). Then  add the remaining vertex of $T$ to the partition set containig $v$ and 
resulting after that exchange. Denote the resulting  balanced partition by  $\mathscr{V}^{\prime}$.
Clearly $P_{\mathscr{V}^{\prime}}(v)$ induces one connected component in $G$.   Further  $P_{\mathscr{V}}(u)$ induces more 
than one connected component in $G$, for every $u\in  V(T)$, because  $T$ does 
not  build a small partition set in $\mathscr{V}$. Thus $n_1(\mathscr{V}^{\prime})\ge n_1(\mathscr{V})$ holds again 
for the same reason as above. 
\end{description}

We repeat the above transformation step as long as there are vertices $v$ for which $P_{\mathscr{V}}(v)\neq P_{\mathscr{V}^{\ast}}(v)$, where $\mathscr{V}$ denotes the most recently constructed $k$-balanced partition. 
We end up with  a $k$-balanced partition $\mathscr{V}^{\prime}$  which contains as partition 
sets all big partitition sets   of $\mathscr{V}^{\ast}$ which induce one connected component in $G$, respectively, 
Moreover $\mathscr{V}^{\prime}$ fulfills the following inequality
\begin{equation}
\label{reni1}
n_1(\mathscr{V}^{\prime})\ge n_1(\mathscr{V})\, .
\end{equation}
 
Notice finally that by removing from $G$ all big partition sets of $\mathscr{V}^{\ast}$ which induce one connected 
component in $G$, respectively, we obtain a graph whose largest connected component contains at most $2^{h-k'}-1$ vertices.
So, if $\mathscr{V}^{\prime}$ contains any partition sets which induce one connected component in $G$ besides the   
big partition sets of $\mathscr{V}^{\ast}$ inducing one connected cpmponent, than these partition sets should be 
small ones.  Finally, since in every $k$-balanced partition there is only one small partition set and 
 the small partition set in   $\mathscr{V}^{\ast}$ induces one connected component in $G$ we get 
\begin{equation}\label{reni2}
n_1(\mathscr{V}^{\ast})\ge 
n_1(\mathscr{V}^{\prime})\, .
\end{equation} 
By combining (\ref{reni1}) and (\ref{reni2}) we get $n_1(\mathscr{V}^{\ast})\ge n_1(\mathscr{V})$ and this completes the proof.
\end{proof}
\subsection{The optimal value of  the $2^{k^{\prime}}$-BPP on regular binary trees }
\label{subsection:section:approximationRatio:kBalancedPartitioningProblem:objectiveValueKBPP}
According to equality~(\ref{reni3}) the optimal value $c(G, \mathscr{V}^*)$ of the $2^{k'}$-BPP on a
 regular binary tree of height $h$ for  $1\le k^{\prime}\le h$ is given as  
$c(G, \mathscr{V}^*)=2k-n_1(\mathscr{V}^{\ast})-1$. Recall that
 $n_1(\mathscr{V}^{\ast})$ is the number of partitions sets of $\mathscr{V}^{\ast}$ which induce exactly 
one connencted component each in $G$. Observe that for every $i \in \nz$, $1\le i\le e+1$, 
the algorithm presented in 
Section~\ref{subsection:section:approximationRatio:kBalancedPartitioningProblem:solutionAlgorithmKBPP} 
constructs eaxctly $2^{h-it+1}$ big partition sets inducing one connected component in $G$,  respectively.
More precisely, it constructs one such partition set  for each vertex of level   $h-it+1$, where 
the partition sets arise in the $i$-th bands of height $t-1$   by cutting the edge 
joining the above mentioned vertices  to their  right children, respectively. 
Finally in its last step the algorithm constructs the small  partition set which  also induces one connected component in $G$. 

\noindent Thus the following equality holds 
\begin{equation}\label{nr:1:components}
n_1(\mathscr{V}^{\ast})=1+\sum_{i=1}^{e+1} 2^{h-it+1}= 
1+ 2^{h+1}\frac{1-\left(\frac{1}{2^t}\right)^{e+1}}{2^t\left(1-\frac{1}{2^t}\right)}\, ,
\end{equation}
and this implies 
\begin{equation}\label{objfuncval:kBPP}
 c(G, \mathscr{V}^*)=2k-2 - 2^{h+1}\frac{1-\left(\frac{1}{2^t}\right)^{e+1}}{2^t\left(1-\frac{1}{2^t}\right)}\, .\end{equation}

\noindent For technical reasons we derive a lower bound  for $c(G, \mathscr{V}^*)$  
which is given as a closed formula depending just on $k$. 
\begin{lemma}\label{lemma:objectiveValueOfAnOptimalKBalancedPartitionLowerBound}
Let $G = (V, E)$ be a  regular binary tree of height $h \geq 1$ and let $k = 2^{k^\prime}$, 
where $1 \leq k^\prime \leq h$. Let  $\mathscr{V}^*$ be the optimal $k$-balanced partition in $G$
computed by the algorithm described in 
Section~\ref{subsection:section:approximationRatio:kBalancedPartitioningProblem:solutionAlgorithmKBPP}. 
The optimal value $c(G, \mathscr{V}^*)$ of the $k$-BPP in $G$ fulfills the following (in)equalities:
\[ c(G, \mathscr{V}^*) \ge \frac{10}{7}k - 2\, ,   \mbox{ if $k^{\prime}\le h-1$,} \]
\[ c(G, \mathscr{V}^*) = \frac{4}{3} k - \frac{3}{2} + \frac{1}{6} (-1)^{\log_2{k}}\, , \mbox{if $k^{\prime}= h$, and }\]
\[ c(G, \mathscr{V}^*) =\frac{3}{2} k - 2\, , \mbox{if $k^{\prime}\le \lfloor\frac{h}{2}\rfloor +1$.}\]
\begin{proof}
\begin{description}
\item{If  $k^{\prime}\le h-1$} we get  $t = h - k^\prime + 
2 \geq h - (h - 1) + 2 = 3$ and the following inequalities hold 
$$n_1(\mathscr{V}^{\ast})\le  1+2^{h+1}\left(\frac{1}{1-\frac{1}{2^t}}-1\right)= 1+\frac{2^{h+1}\frac{1}{2^t}}{1-\frac{1}{2^t}}=1+\frac{2^{h+1-t}}{1-\frac{1}{2^t}}=1+\frac{2^{k^{\prime}-1}}{1-\frac{1}{8}}=$$ $$1+\frac{2^{k^{\prime}}}{\frac{7}{4}}=1+\frac{4}{7}k\,. $$
The last inequality implies 
\begin{equation}
\label{equation:lemma:objectiveValueOfAnOptimalKBalancedPartitionLowerBound}
(G, \mathscr{V}^*)=2k-n_1(\mathscr{V}^{\ast})-1\ge \frac{10}{7} k - 2\, .
\end{equation}
\item{If $k^{\prime}=h$} we get  $t = h - k^\prime + 2 = 2$
and $e = \left\lfloor\frac{h + 1}{t}\right\rfloor - 1 = \left\lfloor\frac{h + 1}{2}\right\rfloor - 1$. 
 If $h$ is odd,  $e = \frac{h + 1}{2} - 1$ holds, and  if $h$ is even we get 
$e=\lceil\frac{h+1}{2}\rceil=\frac{h}{2}$. 
By setting this values for $e$ and $t$ in equation (\ref{nr:1:components}) and simplifying we get
\[ n_1(\mathscr{V}^{\ast})=\left \{\begin{array}{ll}
\frac{2}{3}2^h+\frac{2}{3} & \mbox{if $h$ is odd}\\
\frac{2}{3}2^h+\frac{1}{3} & \mbox{if $h$ is even}\end{array}\right . \]
By plugging these expressions for $n_1(\mathscr{V}^{\ast})$ into equation \ref{objfuncval:kBPP} we obtain 
\begin{equation}
\label{equation:lemma:kBPPObjectiveValueForKPrimeEqualH:proof}
c(G, \mathscr{V}^{\ast})=2k-n_1(\mathscr{V}^{\ast})-1= 
\frac{4}{3} k - \frac{3}{2} + \frac{1}{6} (-1)^{\log_2{k}} \mbox{ in both cases.}
\end{equation}

\item{If  $k^{\prime}\le \lfloor \frac{h}{2}\rfloor +1$} we get   $t = h - k^\prime + 2$ and 
$e = \left\lfloor\frac{h + 1}{t}\right\rfloor - 1 = \left\lfloor\frac{h + 1}{h - k^\prime + 2}\right\rfloor - 1 \leq \\ \left\lfloor\frac{h + 1}{h - \left\lfloor\frac{h}{2}\right\rfloor + 1}\right\rfloor - 1$. 
By distinguishing the two cases when  $h$ is odd and $h$ is even it can be easily observed that  
 $e=0$ in both cases. By substituting  $e$ by $0$ in  equation \ref{nr:1:components} we get
$$n_1(\mathscr{V}^{\ast})=1+2^{h+1}\frac{1}{2^t}=1+\frac{2^{h+1}}{2^{h-k^{\prime}+2}}=1+2^{k^{\prime}-1}=1+\frac{2^{k^{\prime}}}{2}=1+\frac{k}{2}\, , $$
and then by plugging this into equation \ref{objfuncval:kBPP}
\begin{equation}
\label{equation:lemma:kBPPObjectiveValueForKPrimeSmallerEqualH21:proof}
c(G, \mathscr{V}^*)= 2k-1-\frac{k}{2} -1 = \frac{3}{2} k - 2.
\end{equation}
\end{description}
\end{proof}
\end{lemma}
\medskip
\section{The approximation ratio of Algorithm~\ref{algorithm:arrangementPhiU}}\label{section:approxratio}
In order to estimate an approximation ratio $\rho$ for Algorithm~\ref{algorithm:arrangementPhiU} we will 
exploit the relationship 
between  the {\em DAPT} on binary regular trees and the  $k$-{\em BPP} and obtain a lower bound for the 
objective function value of the {\em DAPT} in terms of the special case of  $k$-BPP where $k$ is a power of two. 
\smallskip
	
	Let   the guest graph $G = (V, E)$ and  the host graph  $T$ be regular binary trees  of heights 
$h_G \geq 1$ and 
$h = h_G + 1$,  respectively, 
and let $\phi$ be an arbitrary arrangement with corresponding  objective value
 $OV(G, 2, \phi) = \sum_{(u, v) \in E}{d_T\big(\phi(u), \phi(v)\big)}$.
The length $d_T\big(\phi(u), \phi(v)\big)$ of the unique $\phi(u)$-$\phi(v)$-path in the binary regular tree $T$ is even
for any two vertices $u$ and $v$ in $G$ (see  and Observation~\ref{observation:distance} ).
Moreover,   $2 \leq d_T\big(\phi(u), \phi(v)\big) \leq 2 h$ holds for any two vertices $u$ and $v$ in $G$. 
Let  $a_i(\phi)$, $1 \leq i \leq h$, be  the number of edges which contribute to the value of the  objective function  by $2 i$, i.e.\ 
the number of edges $(u,v)\in E(G)$  for which  $d_T\big(\phi(u), \phi(v)\big)=2i$ holds. Then 
	\begin{equation}
		\label{equation:objectiveValueByUsingTheCoefficientsA1A2LdotsAH}
		OV(G, 2, \phi) = a_h(\phi) \cdot 2 h + a_{h - 1}(\phi) \cdot 2 (h - 1) + \ldots + a_1(\phi) \cdot 2 = 
2 \sum_{i = 1}^h{a_i(\phi) i}.
	\end{equation}
 The number of edges  which contribute to the value of the objective function by {\em at least  $2i$}, i.e.\ 
the number of edges $(u,v)\in E(G)$  for which  $d_T\big(\phi(u), \phi(v)\big)\ge 2i$ holds,   is  given by the 
following the {\bf partial sums}
	\begin{equation}
		\label{equation:partialSums}
		s_i(\phi) \defeq \sum_{j = i}^h{a_j(\phi)} \text{ for all } 1 \leq i \leq h.
	\end{equation}
Clearly,  the following equalities hold 
	\begin{equation}
		\label{equation:partialSumsInverseFunction}
		a_i(\phi) = \left\{ 
			\begin{array}{ll}
				s_i(\phi)  - s_{i + 1}(\phi)
					& \text{for } 1 \leq i \leq h - 1\\
				s_i(\phi)
					& \text{for } i = h\\
			\end{array} \right. \, .
	\end{equation}
By plugging this  into the objective function in (\ref{equation:objectiveValueByUsingTheCoefficientsA1A2LdotsAH}) we get

	\begin{equation}
		\label{equation:objectiveValueByUsingTheCoefficientsS1S2LdotsSH}	
			OV(G, 2, \phi)
				 = 2 \sum_{i = 1}^h{a_i(\phi) i}
				 = 2 \left(\left(\sum_{i = 1}^{h - 1}{i \big(s_i(\phi) - s_{i + 1}(\phi)\big)}\right) + h s_h(\phi) \right)
				 = 2 \sum_{i = 1}^h{s_i(\phi)}.		
	\end{equation}

	\begin{example}
\label{example:coefficientsA1PhiStarA2PhiStarLdotsAHStarAndPartialSumsS1PhiStarS2PhiStarLdotsSHPhiStar}
Let us consider the guest graph in Figure~\ref{figure:guestGraphGForHG3DepictingTheArrangementPhiA} and the arrangement $\phi_A$  obtained by applying 
Algorithm~\ref{algorithm:arrangementPhiU}; this arrangement is  depicted in 
Figure~\ref{figure:arrangementPhiUForHG3}.
 The coefficients $a_i(\phi_A)$, $s_i(\phi_A)$, for  $1\le i\le h$  are listed in 
Table~\ref{table:example:coefficientsA1PhiStarA2PhiStarLdotsAHStarAndPartialSumsS1PhiStarS2PhiStarLdotsSHPhiStar}.

\begin{table}[htb]
			\centering
			\begin{tabular}{l||r|r|r|r}
				$i$
					& $4$
						& $3$
							& $2$
								& $1$\\
				\hline\hline
				$a_i(\phi_A)$
					& $1$
						& $3$
							& $5$
								& $5$\\
				\hline
				$s_i(\phi_A)$
					& $1$
						& $4$
							& $9$
								& $14$\\
			\end{tabular}
\caption{Coefficients $a_i(\phi_A)$ and partial sums $s_i(\phi_A)$, for  $1\le i\le h$.}
\label{table:example:coefficientsA1PhiStarA2PhiStarLdotsAHStarAndPartialSumsS1PhiStarS2PhiStarLdotsSHPhiStar}
\end{table}

By applying  (\ref{equation:objectiveValueByUsingTheCoefficientsA1A2LdotsAH}) and 
 (\ref{equation:objectiveValueByUsingTheCoefficientsS1S2LdotsSH}) we obtain the corresponding objective function value, 
respectively, as follows: 
\begin{equation}
\label{equation:example:coefficientsA1PhiStarA2PhiStarLdotsAHStarAndPartialSumsS1PhiStarS2PhiStarLdotsSHPhiStarObjectiveValue}
			OV(G, 2, \phi_A) = 2 \sum_{i = 1}^h{a_i(\phi_A) i} = 2 (5 \cdot 1 + 5 \cdot 2 + 3 \cdot 3 + 1 \cdot 4) = 56,
\end{equation}
\begin{equation}
\label{equation:example:coefficientsA1PhiStarA2PhiStarLdotsAHStarAndPartialSumsS1PhiStarS2PhiStarLdotsSHPhiStarObjectiveValueLowerBound}
			OV(G, 2, \phi_A) = 2 \sum_{i = 1}^h{s_i(\phi_A)} = 2 (14 + 9 + 4 + 1) = 56.
\end{equation}
\end{example}
	
For $h_G = 0$ the arrangement $\phi_A$ is obviously optimal, so let us assume that $h_G \geq 1$ through the rest of this section.
The next  lemma and its corollary  give closed formulas for the coefficents $a_i(\phi_A)$ and the partial sums $s_i(\phi_A)$, 
$1\le i\le h$, corresponding to the arrangement 
computed by  Algorithm~\ref{algorithm:arrangementPhiU}.
\begin{lemma}\label{lemma:coefficients}
Let the  guest graph $G = (V, E)$ and the host graph $T$ be binary regular trees of heights $h_G \geq 1$ and $h = h_G + 1$, 
respectively, and  let $\phi_A$ be the arrangement computed by  Algorithm~\ref{algorithm:arrangementPhiU}. 
Then the coefficients  $a_i(\phi_A)$, $1\le i\le h$, are given as follows:
		\begin{equation}
			\label{equation:lemma:partialSumms}
			a_i(\phi_A) = \left\{
				\begin{array}{ll}
					\frac{2}{3} 2^{h_G} - \frac{1}{2} - \frac{1}{6}(-1)^{h_G}
						& \text{for } i = 1\\
					\frac{7}{12} 2^{h_G} + \frac{1}{2} + \frac{1}{6} (-1)^{h_G}
						& \text{for } i = 2\\
					3 \cdot 2^{h_G - i}
						& \text{for } 3 \leq i < h\\
					1
						& \text{for } i = h\\
				\end{array}\right.			
		\end{equation}	
		
		\begin{proof}
Consider first $i = 1$  and determine  the number of edges contributing to the objective value $OV(G, 2, \phi_A)$ by 
exactly $2$. Let us first neglect the pair-exchanges done in pseudocode 
line~\ref{algorithm:arrangementPhiU:hGEqual0ConditionElsePairExchangeStatement}. 			
Then there are only two possibilities how to arrange an edge in Algorithm~\ref{algorithm:arrangementPhiU} 
in such a way that it contributes by $2$ to the objective value 
\begin{enumerate}
\item Either it is taken over from the recursive arrangements in pseudocode 
line~\ref{algorithm:arrangementPhiU:hGEqual0ConditionElseArrangeTheSubproblems}
\item or it is produced by placing the root $v_1$ in pseudocode 
line~\ref{algorithm:arrangementPhiU:hGEqual0ConditionElseArrangeTheRoot}.
\end{enumerate}
In the latter case the following property P must hold: (P) A child of the root $v_1$  and $v_1$ itself 
 are placed to children vertices of a common father in $T$.
 Since the children of the root $v_1$ are roots in the previous recursion step, 
and since the roots are always placed on the middle leaf, $h_G = 1$ has to hold 
 in   the corresponding  recursive run, i.e.\ in the run when property P is fulfilled.  
So   exactly one edge contributing by $2$ to the value of the objective function
 arises in every such recursive run with $h_G=1$
(see also Figures~\ref{figure:guestGraphGForHG1} and \ref{figure:arrangementPhiUForHG1} in Appendix). 
There are $2^{h_G-1}$ such runs,  one for each vertex with level $h_G-1$ (playing the role of the root).  
Thus, if the pair echange step is negtlected,  there are $2^{h_G-1}$ edges which contribute $1$ to the 
value of the objective function. 
			
Let $pe(h_G)$ be the number of pair-exchanges done in pseudocode 
line~\ref{algorithm:arrangementPhiU:hGEqual0ConditionElsePairExchangeStatement} when applying the  algorithm on a guest graph of height
 $h_G$. 
We prove that
\begin{equation}
\label{equation:lemma:coefficients:iEqual1Pe1}
pe(h_G) = \frac{1}{6} 2^{h_G} - \frac{1}{2} - \frac{1}{6} (-1)^{h_G}
\end{equation}
by induction on the height $h_G$. 
For $h_G = 1$ the formula in (\ref{equation:lemma:coefficients:iEqual1Pe1}) yields $pe(h_G) = 0$ which
is obviously correct (see also Figures~\ref{figure:guestGraphGForHG1} and \ref{figure:arrangementPhiUForHG1} in Appendix). 
Analogous arguments as in the proof of Lemma~\ref{lemma:upperBound} show that the following recursive equations hold for $h_G\ge 1$:
\begin{equation}
\label{equation:lemma:coefficients:iEqual1Pe2}
				pe(h_G + 1) = \left\{
					\begin{array}{ll}
						2 pe(h_G) + 1
							& \text{for } h_G + 1 \text{ odd}\\
						2 pe(h_G)
							& \text{for } h_G + 1 \text{ even}\\
					\end{array}\right.			
\end{equation}

So by applying   (\ref{equation:lemma:coefficients:iEqual1Pe1})  and  (\ref{equation:lemma:coefficients:iEqual1Pe2}) we get: 
\[	pe(h_G + 1)= 2 pe(h_G)+1= 2 \left(\frac{1}{6} 2^{h_G} - \frac{1}{2} - \frac{1}{6} (-1)^{h_G}\right) + 1
								 =\]\[ \frac{1}{6} 2^{h_G + 1} - \frac{1}{2} - \frac{1}{6} (-1)^{h_G}\, ,
				\]
if $h_G$ is odd, and 
\[ 
	pe(h_G + 1)= 2 pe(h_G) = 2 \left(\frac{1}{6} 2^{h_G} - \frac{1}{2} - \frac{1}{6} (-1)^{h_G}\right) = \frac{1}{6} 2^{h_G + 1}
 - \frac{1}{2} - \frac{1}{6} (-1)^{h_G} \, ,
\]
if $h_G$ is even. 
Thus also $pe(h_G+1)$ fulfills 
(\ref{equation:lemma:coefficients:iEqual1Pe1}), which completes the inductive proof of (\ref{equation:lemma:coefficients:iEqual1Pe1}).
\smallskip 

In Lemma~\ref{lemma:improvementForEveryPairExchange} it was proven  that every pair-exchange done in pseudocode 
line~\ref{algorithm:arrangementPhiU:hGEqual0ConditionElsePairExchangeStatement} increases by $1$ the number of edges 
contributing by $2$ to the value of the objective function. Thus we get
\[
a_1(\phi_A)= 2^{h_G - 1} + pe(h_G)= \frac{2}{3} 2^{h_G} - \frac{1}{2} - \frac{1}{6}(-1)^{h_G}\, , 	
\]
and hence the claim of the lemma holds for $i=1$.
\medskip
			
Let $i = 2$. We use the same technique: We count vertices  
with $\mbox{level}(v) = h_G - 1$ and the number of vertices  with $\mbox{level}(v) = h_G - 2$ first.  
Edges   which  contribute by $4$ to the value of the objective function arise only  in recursive runs 
where the guest graph has height $1$ or $2$ and is rooted at a vertex with level $h_G-1$ or $h_G-2$, respectively.   In every such recursive run   exactly 
one edge of that kind arises,  see also Figures~\ref{figure:guestGraphGForHG1}, \ref{figure:arrangementPhiUForHG1}, 
\ref{figure:guestGraphGForHG2}  and \ref{figure:arrangementPhiUForHG2}). 
Then we consider the effect of the pair-exchanges done in pseudocode 
line~\ref{algorithm:arrangementPhiU:hGEqual0ConditionElsePairExchangeStatement}. 
According to Lemma~\ref{lemma:improvementForEveryPairExchange} each such pair-exchange reduces by one the number of edges 
which  contribute to the  value of the objective function  by $4$, so we get  
\[a_2(\phi_A)= 2^{h_G - 1} + 2^{h_G - 2} - pe(h_G)=  \frac{7}{12} 2^{h_G} + \frac{1}{2} + \frac{1}{6} (-1)^{h_G}\, , 
\]
and hence the claim of the lemma holds for $i=2$.
\medskip

Consider now the case $i\ge 3$. According to Lemma~\ref{lemma:improvementForEveryPairExchange}	the pair-exchanges done in pseudocode 
line~\ref{algorithm:arrangementPhiU:hGEqual0ConditionElsePairExchangeStatement}	have no effect on the number of edges of $G$
which   contribute to the  value of the objective function  by $2i$, if $i\ge 3$.  
	So for   $3 \leq i < h$  the pair-exchanges  can be neglected and we get 
	\[				a_i(\phi_A) = 2^{h_G - (i - 1)} + 2^{h_G - i}=3 \cdot 2^{h_G - i}\, , 
				\]
in compliance with the claim of the lemma. 
\medskip
			
Finally, for  $i = h$ there is exactly  one edge which contributes  by $2h$ to the  value of the objective function,
 namely the one joining 
the root of $G$ with his right child. 
\end{proof}
\end{lemma}
\begin{corollary}\label{corollary:partialSumms}
Let the  guest graph $G = (V, E)$ and the host graph $T$ be binary regular trees of heights $h_G \geq 1$ and $h = h_G + 1$, 
respectively, and  let $\phi_A$ be the arrangement computed by  Algorithm~\ref{algorithm:arrangementPhiU}. 
Then the coefficients  $s_i(\phi_A)$, $1\le i\le h$, are given as follows:
		\begin{equation}
			\label{equation:corollary:partialSumms}
			s_i(\phi_A) = \left\{
				\begin{array}{ll}
					2 \cdot 2^{h_G} - 2
						& \text{for } i = 1\\
					\frac{4}{3} 2^{h_G} - \frac{3}{2} + \frac{1}{6} (-1)^{h_G}
						& \text{for } i = 2\\
					6 \cdot 2^{h_G - i} - 2
						& \text{for } 3 \leq i \leq h
				\end{array}\right.			
		\end{equation}	
		
\begin{proof}
This corollary is a straightforward  consequence of the definition of $s_i(\phi)$, $1\le i\le h$,
(see (\ref{equation:partialSums})) and of Lemma~\ref{lemma:coefficients}.

\begin{itemize}
\item{For $i = h$ we get:}  $s_h(\phi_A) = a_h(\phi_A) = 1 = 6 \cdot 2^{h_G - (h_G + 1)} - 2 = 6 \cdot 2^{h_G - i} - 2$.
\smallskip

\item{For $3 \leq i < h$ we get by induction  on $i$ and starting with $i=h$:} 
 $s_i(\phi_A) = a_i(\phi_A) + s_{i + 1}(\phi_A) = 3 \cdot 2^{h_G - i} + 6 \cdot 2^{h_G - (i + 1)} - 2 = 6 \cdot 2^{h_G - i} - 2$.
\smallskip

\item{For $i = 2$ we get:} $s_2(\phi_A) = a_2(\phi_A) + s_3(\phi_A) = \frac{7}{12} 2^{h_G} + \frac{1}{2} + \frac{1}{6} (-1)^{h_G} + 6 \cdot 2^{h_G - 3} - 2 = \frac{4}{3} 2^{h_G} - \frac{3}{2} + \frac{1}{6} (-1)^{h_G}$.
\smallskip

\item{For $i = 1$ we get:} $s_1(\phi_A) = a_1(\phi_A) + s_2(\phi_A) = \frac{2}{3} 2^{h_G} - \frac{1}{2} - \frac{1}{6}(-1)^{h_G} + \frac{4}{3} 2^{h_G} - \frac{3}{2} + \frac{1}{6} (-1)^{h_G} = 2 \cdot 2^{h_G} - 2$.
\end{itemize}
\end{proof}
\end{corollary}
\smallskip

Next we  give a lower bound for $s_i(\phi)$, where $\phi$ is an arbitrary arrangement of the vertices of the guest graph $G$ with height $h_G$ into the leaves of the host 
graph $T$ with height $h \defeq h_G+1$ and  $i$ is  some integer between $1$ and $h$. 
$s_i(\phi)$ is the number of edges of $G$ which contribute by at least $2i$ to the value of  the objective 
function. Obviously, each such edge  joins vertices of $G$ which are arranged at the leaves of  
{\em different binary regular subtrees} of $T$ of height $i-1$ and rooted at vertices of level $h-(i-1)=h_G-i+2 \eqdef k^{\prime}$. 
Clearly, there are $2^{k^{\prime}}$ such subtrees of $T$.
Let us denote them  by $T^{(i)}_j$, for $1\le j\le 2^{k^{\prime}}$. Let $V^{(i)}_j\subset V(G)$ be the set of vertices of $G$   
 which are arranged at the leaves of $T^{(i)}_j$. Since for all leaves $b$ of $T$ but one there is some vertex $v \in V(G)$ 
whith $\phi(v)=b$,  $|V^{(i)}_j|=2^{i-1}$  holds for all but one index $j$, $1\le j\le 2^{k^{\prime}}$,   and for the exception, 
say $j_0$, $|V^{(i)}_{j_0}|=2^{i-1}-1$ holds. Thus $\mathscr{V}^{(i)} \defeq \{V^{(i)}_j| 1\le j\le 2^{k^{\prime}}\}$ is $k$-balanced partition of 
$G$ with  $k= 2^{k^{\prime}}$, $k^{\prime}=h_G-i+2$,  and $s_i(\phi)=c(G,\mathscr{V}^{(i)})$. Let $\mathscr{V}^*_i$ be the optimal 
$k$-balanced partition of $G$ (which can be  computed by the algorithm presented  in 
Section~\ref{subsection:section:approximationRatio:kBalancedPartitioningProblem:solutionAlgorithmKBPP}
 and for which  lower bounds of the objective function value as in 
Section~\ref{subsection:section:approximationRatio:kBalancedPartitioningProblem:objectiveValueKBPP}
 are known). Then  $s_i(\phi)\ge c(G,\mathscr{V}^*_i)$ holds for all $i$, $2\le i\le h$ and for every arrangement $\phi$. 
Let us  denote  the lower bounds $s_i^{L} \defeq c(G,\mathscr{V}^*_i)$ for $2\le i\le h$,
and $s_1^L \defeq |V(G)|-1=2^h-2$ for $i=1$. 
\medskip

\noindent Thus we get 
	\begin{equation}		\label{equation:lowerBound}
 		OV(G, 2, \phi) \geq 2 \sum_{i = 1}^h{s_i^L} \mbox{ for all arrangements $\phi$.}
	\end{equation}

Notice that $s_1(\phi) =|E(G)|=|V(G)|-1=s_1^L$ holds for every arrangement 
$\phi$ of the  guest graph  $G$.  Notice, moreover,  that for $h_G\le 4$ the bound in 
(\ref{equation:lowerBound}) is tight, in the sense that it matches the optimal value of the 
objective function of $DAPT(G,2)$, as illustrated in the following tables. 

	\begin{table}[H]
		\centering
		\begin{tabular}{l||r|r}
			$i$
				& $2$
					& $1$\\
			\hline\hline
			$s_i(\phi_A)$
				& $1$
					& $2$\\
			\hline
			$s_i^L$
				& $1$
					& $2$\\
		\end{tabular}
		\caption[Partial sums $s_i(\phi_A)$ and the corresponding lower bounds $s_i^L$, where $1 \leq i\leq h$, for a guest graph $G = (V, E)$ of height $h_G = 1$.]{Partial sums $s_i(\phi_A)$ and the corresponding lower bounds $s_i^L$, where $1 \leq i\leq h$, for a guest graph $G = (V, E)$ of height $h_G = 1$.}
		\label{table:example:example:partialSumsAndTheCorrespondingLowerBoundsForHG1}
	\end{table}
	\begin{table}[H]
		\centering
		\begin{tabular}{l||r|r|r}
			$i$
				& $3$
					& $2$
						& $1$\\
			\hline\hline
			$s_i(\phi_A)$
				& $1$
					& $4$
						& $6$\\
			\hline
			$s_i^L$
				& $1$
					& $4$
						& $6$\\
		\end{tabular}
		\caption[Partial sums $s_i(\phi_A)$ and the corresponding lower bounds $s_i^L$, where $1 \leq i\leq h$, for a guest graph $G = (V, E)$ of height $h_G = 2$.]{Partial sums $s_i(\phi_A)$ and the corresponding lower bounds $s_i^L$, where $1 \leq i\leq h$, for a guest graph $G = (V, E)$ of height $h_G = 2$.}
		\label{table:example:example:partialSumsAndTheCorrespondingLowerBoundsForHG2}
	\end{table}
	\begin{table}[H]
		\centering
		\begin{tabular}{l||r|r|r|r}
			$i$
				& $4$
					& $3$
						& $2$
							& $1$\\
			\hline\hline
			$s_i(\phi_A)$
				& $1$
					& $4$
						& $9$
							& $14$\\
			\hline
			$s_i^L$
				& $1$
					& $4$
						& $9$
							& $14$\\
		\end{tabular}
		\caption[Partial sums $s_i(\phi_A)$ and the corresponding lower bounds $s_i^L$, where $1 \leq i\leq h$, for a guest graph $G = (V, E)$ of height $h_G = 3$.]{Partial sums $s_i(\phi_A)$ and the corresponding lower bounds $s_i^L$, where $1 \leq i\leq h$, for a guest graph $G = (V, E)$ of height $h_G = 3$.}
		\label{table:example:example:partialSumsAndTheCorrespondingLowerBoundsForHG3}
	\end{table}
	\begin{table}[H]
		\centering
		\begin{tabular}{l||r|r|r|r|r}
			$i$
				& $5$
					& $4$
						& $3$
							& $2$
								& $1$\\
			\hline\hline
			$s_i(\phi_A)$
				& $1$
					& $4$
						& $10$
							& $20$
								& $30$\\
			\hline
			$s_i^L$
				& $1$
					& $4$
						& $10$
							& $20$
								& $30$\\
		\end{tabular}
		\caption[Partial sums $s_i(\phi_A)$ and the corresponding lower bounds $s_i^L$, where $1 \leq i\leq h$, for a guest graph $G = (V, E)$ of height $h_G = 4$.]{Partial sums $s_i(\phi_A)$ and the corresponding lower bounds $s_i^L$, where $1 \leq i\leq h$, for a guest graph $G = (V, E)$ of height $h_G = 4$.}
		\label{table:example:example:partialSumsAndTheCorrespondingLowerBoundsForHG4}
	\end{table}

In general, the  lower bound in (\ref{equation:lowerBound}) is  not tight. Already for $h_G=5$ 
there is a gap between the value of the objective function corresponding
 to the  arrangement $\phi_A$ generated by the  algorithm ${\cal A}$ and the lower bound,   
as shown  in  Table~\ref{table:example:example:theLowerBoundIsNotTight} below. 
Moreover, in  the following example  we prove that  $\phi_A$ is an optimal arrangement for $h_G=5$.
Thus
we  conclude that the lower bound in (\ref{equation:lowerBound}) does not match the optimal value of 
the objective function of $DAPT(G,2)$ already for $h_G=5$. 
\begin{example}
	\label{example:theLowerBoundIsNotTight}
	Consider the guest graph $G = (V, E)$ to be a  complete binary tree of height $h_G = 5$ 
ordered according to the canonical ordering. The  partial sums $s_i(\phi_A)$ and the 
 lower bounds $s_i^L$,  for 
$1 \leq i \leq h$, computed according to Corollary~\ref{corollary:partialSumms}, 
inequality (\ref{equation:lowerBound}) and equation (\ref{objfuncval:kBPP}), are given in  
Table~\ref{table:example:example:theLowerBoundIsNotTight} below.  
	\begin{table}[htb!]
		\centering
		\begin{tabular}{l||r|r|r|r|r|r}
			$i$
				& $6$
					& $5$
						& $4$
							& $3$
								& $2$
									& $1$\\
			\hline\hline
			$s_i(\phi_A)$
				& $1$
					& $4$
						& $10$
							& $22$
								& $41$
									& $62$\\
			\hline
			$s_i^L$
				& $1$
					& $4$
						& $10$
							& $21$
								& $41$
									& $62$\\
		\end{tabular}
		\caption[Partial sums $s_i(\phi_A)$ and the corresponding lower bounds $s_i^L$ for $1 \leq i\leq h$.]{Partial sums $s_i(\phi_A)$ and the corresponding lower bounds $s_i^L$ for $1 \leq i\leq h$.}
		\label{table:example:example:theLowerBoundIsNotTight}
	\end{table}

Thus in the  case $h_G=5$  the bound of inequality~(\ref{equation:lowerBound}) does not match the objective 
function value corresponding to $\phi_A$ because $s_3(\phi_a)>s_3^L$.
Now a natural question arises:

Question: Does  the bound in (\ref{equation:lowerBound}) match the optimal value of the objective 
function  
of the $DAPT(G,2)$ if $h_G=5$?

We show that $\phi_A$ is an optimal 
arrangement if $h_G=5$, and hence  the answer to the above  question is ``no''.
Indeed, assume that $\phi_A$ is not optimal and let $\phi_{\ast}$ be an 
optimal arrangement. Observe that $s_6(\phi_{\ast})=1$ has to hold because $s_6(\phi_{\ast})\ge 2$  
leads to a contradiction as follows:
\[ OV(G,2,\phi_{\ast})- OV(G,2,\phi_A) =2\sum_{i=1}^6  \big (s_i(\phi_{\ast})-s_i(\phi_{A})\big )=\]
\[2\sum_{i=1}^5  \big (s_i(\phi_{\ast})-s_i(\phi_A)\big )+ 2\big (s_6(\phi_{\ast})-s_6(\phi_A)\big )
>
2\sum_{i=1}^5  (s_i^L-s_i(\phi_{A}))+2 =-2+2= 0\, .\] 
 By similar arguments we would get   $s_i(\phi_{\ast})=s_i^L=s_i(\phi_A)$ for $i\in\{1,2,4,5\}$.
Since $s_6(\phi_{\ast})=1$ there will only be one edge of $G$ such that its endpoints are 
mapped to leaves of the right and the left basic subtrees of $T$, respectively. 
Clearly this edge can only be $(v_1,v_2)$ or $(v_1,v_3)$. 
Assume w.l.o.g.\ that this edge is $(v_1,v_3)$ and that $\phi_{\ast}$ arranges the right basic subtree of $G$ (of height $4$) at 
the leaves of the right basic subtree $T^r$ of $T$. 
Since algorithm ${\cal A}$ yields an optimal arrangement for $h_G\le 4$, we can than assume w.l.o.g.\ that $\phi_{\ast}$ and 
$\phi_A$ arrange the right basic subtree of $G$ in the same way. 
Now consider the left basic subtree of $G$ together with the root $v_1$ and denote this subgraph of $G$ by $G_1$.  
$\phi_{\ast}$
 arranges $G_1$ at the leaves of the left basic subtree $T^{l}$ of $T$.  
Let $G_1^{(a)}$ and $G_1^{(b)}$ the two subgraphs of $G_1$  arranged by $\phi_{\ast}$ at the leaves of the left and the 
right basic subtrees of $T^l$, denoted by $T^{ll}$ and $T^{lr}$, respectively.   Since the number of edges 
which contribute by at least $2\cdot 5$ to the value  $OV(G,2,\phi_{\ast})$ is $s_5(\phi_{\ast})=4$, 
 there are just two edges $e_i$  of $G_1$ such that $\phi_{\ast}$ arranges one endpoint of $e_i$ to some leaf of 
$T^{ll}$ and
the other endpoint of $e_i$ to some leaf of   $T^{lr}$,  for $i=1,2$. Recalling that  $|V(G_1^{(a)})|= |V(G_1^{(b)})|$ 
we can easily convince ourselves (in the worst case by using complete enumeration) that one of the edges  $e_i$, $i=1,2$, has to coincide with one of the two edges   $(v_1,v_2)$ or $(v_2,v_5)$ (or $(v_2,v_4)$, symmetrically).
We obtain two cases. (A) If $e_1=(v_1,v_2)$, 
then  $e_2=(v_2,v_5)$ must hold. (B) Otherwise, if $e_1=(v_2,v_5)$ and $e_2\neq (v_1,v_2)$, then $e_2$ 
has 
to join some  leave of the left basic subtree of $G_1\setminus \{v_1\}$ to its father,  e.g.\  
$e_2=(v_{19},v_{39})$. The edges $e_i$, $i=1,2$, fully determine the corresponding subgraphs 
$G_1^{(a)}$ and $G_1^{(b)}$,  each of them having $16$ vertices. For every realisation of $e_i$, $i=1,2$,
the problems $DAPT(G_1^{(a)},2)$ and   $DAPT(G_1^{(b)},2)$ can be solved  by complete enumeration to  observe 
that the correpsonding optimal values coincide with the values of the objective function 
corresponding to the arrangement of the respective subgraphs according to  $\phi_A$. 
Notice that it is enough to do 
the complete enumeration for the DAPTs resulting in the case A and for the DAPTs resulting 
in the case B  with   $e_2=(v_{19},v_{39})$; all other possible  realisations of $e_2$ in case B lead to 
$G_1^{(a)}$, $G_1^{(b)}$  which are isomorphic to    $G_1^{(a)}$, $G_1^{(b)}$ obtained for $e_2=(v_{19},v_{39})$, respectively.  
\smallskip

Notice finally, that the  answer ``no'' to the question posed above   is not surprising.
In general  a collection  of optimal $2^{k}$-balanced partitions,  $1\le k\le h_G$, 
of a  binary tree $G$ of height $h_G$, does not need to coincide with the $2^{k}$-balanced 
partitions of  $G$ 
defined in accordance with some feasible solution  of the data arrangement problem in $G$. 
The reason is that the collection of $2^{k}$-balanced partitions, $1\le k\le h_G$,   
which arises in accordance with some  arrangement $\phi$, is laminar,  meaning  that the 
partition 
sets of the $2^{k}$-balanced partition are obtained as particular partitions     of the 
partition sets of 
the $2^{k-1}$-balanced partition, for $2\le k \le h_G$. More concretely, 
 if $\mathscr{V} = \Big \{V_1^{(1)}, V_1^{(2)}\Big \}$ is the  $2$-balanced partition in the  laminar  
collection of balanced partitions, then the $4$-balanced partition is obtained by partitioning 
 $V_1^{(1)}$ and $V_1^{(2)}$ into  $2$-balanced partitions each, and so on, until the partition sets 
are pairs of vertices,  and hence  a $2^{h_G}$-balanced partition results. 
On the other side, in general,  a collection  of optimal $2^{k}$-balanced partitions  
 of a complete binary tree $G$ of height $h_G$, for $1\le k\le h_G$, 
is not necessarily laminar. 
\end{example}
\medskip

Now we can state the main result of this paper.
\begin{theorem}\label{theorem:approximationRatio}
	Let the guest graph  $G = (V, E)$ and the host graph  $T$ be regular binary  trees  of heights $h_G \geq 1$ and $h = h_G + 1$,  
respectively. Then Algorithm~\ref{algorithm:arrangementPhiU} is a $\frac{203}{200}$-approximation algorithm.
\begin{proof}
The cases $h_G= 0$, $h_G=1$, $h_G=2$ are obvious: the arrangements $\phi_A$ obtained form Algorithm~\ref{algorithm:arrangementPhiU} 
are optimal in these cases, respectively, as one can convince  himself by simple arguments or by  full enumeration (see also 
Figures~\ref{figure:guestGraphGForHG0} to \ref{figure:arrangementPhiUForHG2}  in Appendix).
\smallskip
			
Let $h_G \geq 3$. 
Consider the difference $OV(G, 2, \phi)- 2\sum_{i=1}^h s_i^L = 2 \sum_{i = 1}^h{s_i(\phi_A)}- 2\sum_{i=1}^h s_i^L$, and set 
\[
				D \defeq \sum_{i = 1}^h{\left(s_i(\phi_A) - s_i^L\right)}\, .
\]

As mentioned above  $s_h^L = 1 = s_h(\phi_A)$.
\smallskip

For  $s_2^L = c(G, \mathscr{V}_2^*)$ we have $k^\prime = h_G - 2 + 2 = h_G$ and 
$k = 2^{k^\prime} = 2^{h_G}$, and thus we can apply Lemma~\ref{lemma:objectiveValueOfAnOptimalKBalancedPartitionLowerBound} and 
Corollary~\ref{corollary:partialSumms} to obtain  $s_2^L = \frac{4}{3} 2^{h_G} - \frac{3}{2} + \frac{1}{6} (-1)^{h_G} = s_2(\phi_A)$.
\smallskip
			
Let us consider $s_i^L = c(G, \mathscr{V}_i^*)$, where $3 \leq i \leq h_G$. 
If $1 \leq k^\prime \leq \left\lfloor\frac{h_G}{2}\right\rfloor + 1$,  with  $k^\prime = h_G - i + 2$
 we obtain the condition 
$i \geq h_G - \left\lfloor\frac{h_G}{2}\right\rfloor + 1$. 
For $h_G\ge 3$ the inequality $h_G - \left\lfloor\frac{h_G}{2}\right\rfloor + 1\ge 3$ holds and  hence 
Corollary~\ref{corollary:partialSumms} applies. So by applying Lemma~\ref{lemma:objectiveValueOfAnOptimalKBalancedPartitionLowerBound} we get  $s_i^L = \frac{3}{2} 2^{h_G - i + 2} - 2 = 6 \cdot 2^{h_G - i} - 2 = s_i(\phi_A)$ if 
 $i \geq h_G - \left\lfloor\frac{h_G}{2}\right\rfloor + 1$.
\smallskip
			
The remaining indices are   $3 \leq i \leq h_G - \left\lfloor\frac{h_G}{2}\right\rfloor$. 
Apply  Corollary~\ref{corollary:partialSumms} and Lemma~\ref{lemma:objectiveValueOfAnOptimalKBalancedPartitionLowerBound} to obtain:
\[D\leq \sum_{i = 3}^{h_G - \left\lfloor\frac{h_G}{2}\right\rfloor}{\left(6 \cdot 2^{h_G - i} - \frac{10}{7}2^{h_G - i + 2}\right)}
						 = \frac{2}{7} 2^{h_G} \sum_{i = 3}^{h_G - \left\lfloor\frac{h_G}{2}\right\rfloor}{2^{-i}}\, .\]
The sum of the above  geometric progression is   
$\sum_{i = 3}^{h_G - \left\lfloor\frac{h_G}{2}\right\rfloor}{2^{-i}} = \frac{1}{4} - \frac{1}{2} 2^{h_G-\lceil \frac{h_G}{2}\rceil}\le \frac{1}{4} - \frac{\sqrt{2}}{2} 2^{-\frac{h_G}{2}}$. 
Summarizing we get 			
\[D \leq \frac{2}{7} 2^{h_G} \left(\frac{1}{4} - \frac{\sqrt{2}}{2} 2^{-\frac{h_G}{2}}\right)
 = \frac{1}{14} 2^{h_G} - \frac{\sqrt{2}}{7} 2^{-\frac{h_G}{2}}\, .\]
\medskip
			
By applying Lemma~\ref{lemma:upperBound} we get the following  approximation ratio $\rho$ 
\[\rho(h_G) \defeq \frac{OV(G, 2, \phi_A)}{OV(G, 2, \phi_A) - 2 D}
=\] \[ \frac{\frac{29}{3} \cdot 2^{h_G} - 4 h_G - 9 + \frac{1}{3} (-1)^{h_G}}{\left(\frac{29}{3} \cdot 2^{h_G} - 4 h_G - 9 + 
\frac{1}{3} (-1)^{h_G}\right) - 2 \left(\frac{1}{14} 2^{h_G} - \frac{\sqrt{2}}{7} 2^{-\frac{h_G}{2}}\right)} \, . \]
			After some standard  algebraic transofrmations  we obtain
\[	\rho(h_G) \defeq \frac{\frac{29}{3} 2^{h_G} - 4 h_G - \frac{26}{3}}{\frac{200}{21} 2^{h_G} - 4 h_G + \frac{2 \sqrt{2}}{7} 2^{\frac{h_G}{2}} - 
\frac{28}{3}}\, .\]
			
Finally, it is not difficult to verify that $\rho(h_G)$ is a strictly monotone increasing sequence (for $h_G \geq 4$) and since
		\[			\rho \defeq \lim_{h_G \to +\infty} \rho(h_G) = \lim_{h_G \to \infty} \frac{\frac{29}{3} 2^{h_G} - 4 h_G - \frac{26}{3}}{\frac{200}{21} 2^{h_G} - 4 h_G + \frac{2 \sqrt{2}}{7} 2^{\frac{h_G}{2}} - \frac{28}{3}} = \frac{203}{200} = 1.015, \]
			this completes the proof.
\end{proof}
\end{theorem}
\section{The complexity of the DAPT  with a tree as   a guest graph} \label{section:complexity}
In this section we show that the   DAPT where the guest graph $G$ is a tree on $n$ vertices and the host graph $T$ is a  complete 
$d$-regular tree of height $\lceil \log_d n \rceil$, for some fixed $d\in \nz$, $d\ge 2$,   is ${\cal NP}$-hard.
This result also settles  a more general open question posed by {\sc Luczak} and 
{\sc Noble}~\cite{LuczakNoble:OptimalArrangementOfDataInATreeDirectory} about the complexity of the GEP when both input
graphs are trees. 
\smallskip

First let us state a  simple result on the optimal value of the objective function of the $DAPT(G,d)$ in the case where $G$ is a start graph; this result was proven in  in  {\sc \c{C}ela} and 
{\sc Stan\v{e}k}~\cite{CelaStanek:HeuristicsForTheDataArrangementProblemOnRegularTrees}.
 	\begin{lemma}	\label{lemma:gIsAStar}
		Let $G = (V, E)$ be a star graph (i.e.\ a complete bipartite graph with $1$ vertex in
 one side of the partition and the rest of the vertices in the other side)  with $n$ vertices and the central vertex  $v_1$. 
 Let  the host graph $T$ be  a complete $d$-regular tree of height $h = \lceil\log_d{n}\rceil$, with  $d \in \nz$, $2 \leq d \leq n$. 
Then the optimal value of the objective function $OPT(G,d)$ is given by
		\begin{equation}
					OPT(G,d)= 2 \left(h \; n - \frac{d^h - 1}{d - 1}\right)\, .
		\end{equation}
Moreover, an arrangement is optimal if and only if  it  arranges the central vertex $v_1$ together with other $d^{h-1}-1$ arbitrarily selected  vertices 
of $G$ at the leaves of
some (arbitrarily selected) basic subtree of $T$ (and the other vertices arbitrarily).     
		\begin{proof}
			 See {\sc \c{C}ela} and {\sc Stan\v{e}k}~\cite{CelaStanek:HeuristicsForTheDataArrangementProblemOnRegularTrees}.
		\end{proof}
	\end{lemma}
\smallskip

Next we will consider another very special case where the guest graph $G$ is the disjoint union of three star graphs. 	
\begin{lemma}\label{lemma:gConsistsOfThreeStars}
Let the guest graph $G$ be the disjoint union of three star graphs  $S_i$, $1\le i\le 3$, i.e.\ 
$V(G)=\mathop{\dot{\bigcup}}_{i=1}^3 V(S_i)$ and $E(G)=\mathop{\dot{\bigcup}}_{i=1}^3 E(S_i)$.  
Assume that $|V(S_i)| \eqdef n_i$, $1\le i\le 3$,  
and $n_1 \geq n_2 \geq n_3$, where $V(S_i)$ is the vertex set of $S_i$, $1\le i\le 3$.
 Assume that  $n \defeq \big|V(G)\big|=n_1+n_2+n_3$ is  a power of $d$ for some fixed $d\in \nz$, $d\ge 2$, 
and   $n_1 \geq \frac{n}{d}$. Let  the host graph $T$ be   a complete $d$-regular tree of height 
$h = \log_d{n}$. Then the optimal value of the objective function  $OPT(G,d)$ is given by  
%
%
\begin{equation}
\label{equation:lemma:gConsistsOfThreeStars}
 OPT(G,d)=2 \left(h_1 n_1 - \frac{d^{h_1} - 1}{d - 1}\right) + 2 \left(h_2 n_2 - \frac{d^{h_2} - 1}{d - 1}\right) + 
2 \left(h_3 n_3 - \frac{d^{h_3} - 1}{d - 1}\right)\, ,
\end{equation}
where $h_i = \lceil\log_d{n_i}\rceil$, $i\in \{1,2,3\}$.		
\begin{proof}
Let $B = \{b_1, b_2, \ldots, b_n\}$ be the set of the  leaves of $T$ labelled according to the canonical order.
 Arrange the central vertex of $S_1$ together with other $2^{h-1}-1$ vertices of $S_1$ at the leaves 
$b_1,b_2,\ldots,b_{d^{h-1}}$ of the  leftmost basic subtree of  $T$. 
The other vertices of $S_1$ will be arranged later at other  appropriately chosen leaves of $T$.  Notice, however, that
 independently on the arrangement of these vertices 
the contribution of the   edges of $S_1$ to the objective function value of 
$DAPT(G,d)$ will be equal to  $2 \left(h \; n_1 - \frac{d^h - 1}{d - 1}\right)$, according to Lemma~\ref{lemma:gIsAStar}.

Arrange the vertices of $S_2$ to the $n_2$ leaves $b_{n-n_2+1},\ldots,b_n$ of $T$  with the largest indices. 
According to Lemma~\ref{lemma:gIsAStar} the  edges of $S_2$ will then  contribute by  
$2 \left(h_2\; n_2 - \frac{d^{h_2} - 1}{d - 1}\right)$  to the objective function value of 
$DAPT(G,d)$. Next, arrange the vertices of $S_3$ to the $n_3$ leaves $b_{d^{h-1}+1},\ldots,
 b_{n_1+n_3}$ of $T$ 
(which are still free because $n_1+n_2+n_3=d^h$ and only the leaves with the   $d^{h-1}\le n_1$ smallest indices as well 
as the leaves with the $n_3$ largest indices have been occupied already).  
According to  Lemma~\ref{lemma:gIsAStar} the  edges of $S_2$ will
then  contribute by  $2 \left(h_3\; n_3 - \frac{d^{h_3} - 1}{d - 1}\right)$
 to the objective function value of  $DAPT(G,d)$. 
Finally arrange the $n_1-d^{h-1}$ vertices of $S_1$ not arranged yet to the remaining $d^h-n_2-n_3-d^{h-1}=n_1-d^{h-1}$  leaves. 
Summarizing,  this arrangement yields an objective function value equal to the expression in 
(\ref{equation:lemma:gConsistsOfThreeStars}), and is therefore optimal because the following inequality  
\begin{eqnarray*}
 OV(G, d, \phi) & =&  \sum_{(u, v) \in E(G)} d_T\big(\phi(u), \phi(v)\big)=  \sum_{i=1}^3 \sum_{(u, v) \in E(S_i)} d_T\big(\phi(u), \phi(v)\big)\ge
\\
&&  2 \left(h_1 n_1 - \frac{d^{h_1} - 1}{d - 1}\right) + 2 \left(h_2 n_2 - \frac{d^{h_2} - 1}{d - 1}\right) + 
2 \left(h_3 n_3 - \frac{d^{h_3} - 1}{d - 1}\right )
\end{eqnarray*}

\noindent  holds for any arrangement $\phi$ of $DAPT(G,d)$ due to Lemmma~\ref{lemma:gIsAStar}.
\end{proof}
\end{lemma}

Next we state the main result of this section. 
\begin{theorem}
The DAPT with a host graph $T$  being a complete $d$-regular tree  is ${\cal NP}$-hard for every fixed $d \geq 2$ even if the guest graph $G$ is a tree.
\begin{proof}
The problem obviously belongs to ${\cal NP}$. 
The ${\cal NP}$-hardness is proven  by means of a reduction from   the {\em numerical matching with target sums} ({\em NMTS}) problem. 
The {\em NMTS} is ${\cal NP}$-hard  and is defined as follows 
(see {\sc Garey} and {\sc Johnson}~\cite{GareyJohnson:ComputersAndIntractabilityAGuideToTheTheoryOfNPCompleteness}): 
Let   three sets $X = \{x_1, x_2, \ldots, x_n\}$, $Y = \{y_1, y_2, \ldots, y_n\}$ and $Z = \{z_1, z_2, \ldots, z_n\}$ of positive integers, 
with $\sum_{i = 1}^n{z_i} = \sum_{i = 1}^n{x_i} + \sum_{i = 1}^n{y_i}$  with $n\ge 2$ be given. The goal is to decide
 whether  there exist two permutations $(j_1, j_2, \ldots, j_n)$ and $(k_1, k_2, \ldots, k_n)$ of the indices $\{1, 2, \ldots, n\}$, 
such that $z_i = x_{j_i} + y_{k_i}$ for all $i = 1, 2, \ldots, n$ hold.
\medskip

Consider an instance of the NMTS  and a given integer $d \geq 2$. 
We construct an instance $DAPT(G,d)$ of the DAPT as follows. 
Let $l_y \in \nz$ be the smallest natural number  such that $y_i \le d^{l_y-4}$,  
for all $1 \leq i \leq n$. 
 Let $l_x \in \nz$  be the smallest natural number such that $x_i+y_j+(d-1)d^{l_y-4}< d^{l_x-2}$, 
for all $1\le i,j \le n$, 
Finally, let  $l_z$ be the smallest natural number such that $z_i\le d^{l_z}-(d-1)d^{l_z-4}-(d-1)d^{l_z-2}$, for all $1 \leq i \leq n$. 
Let then $l \defeq \max\{l_x,l_y,l_z\}$, and let $L \in \nz$ be the smallest  natural number  such that $n d^l \leq d^{L - 1}$.
Define three  vertex disjoint star graphs  $S^x_i$, $S^y_i$ and  $S^z_i$, for every $1 \leq i \leq n$, with $|V(S^x_i)|=(d-1)d^{l-2}+x_i$, 
$|V(S^y_i)|=(d-1)d^{l-4}+y_i$ and $|V(S^z_i)|=d^l-(d-1)d^{l-4}-(d-1)d^{l-2}-z_i$, for $i=1,2,\ldots,n$. 
Notice that $\sum_{i=1}^n (|V(S^x_i)|+|V(S^y_i)|+|V(S^z_i)|=nd^l$. 
\medskip

Let $v_1(S^x_i)$, $v_1(S^y_i)$ and $v_1(S^z_i)$ be the central vertices of the stars graphs introduced above for $1\le i\le n$, 
respectively. 
Next introduce the  vertices $u_1, u_2, \ldots, u_{\widehat{n}}$, where $\widehat{n} = d^{L-1}  - 1$ and one vertex $v_1$. 
Finally consider a family of  stars $S^{\prime}_i$, $1\le i \le n^{\prime}$, with   $d^l$ vertices each 
and  $n^{\prime} \defeq (d-1)d^{L-1-l}-n\ge 0$. 
 Denote by $v_1(s^{\prime}_i)$,  $1\le i \le n^{\prime}$, their central vertices respectively. 
\medskip

The guest graph $G = (V, E)$ is defined in the following way. 
The vertex set is given by all vertices defined above, 
and the edge set contains all edges contained in the stars $S^x_i$, $S^y_i$ and $S^z_i$,  $1 \leq i \leq n$, and $S^{\prime}_i$, 
$1\le i \le n^{\prime}$, 
together with  edges connecting the vertex $v_1$ with the central vertices  $v_1(S^x_i)$, $v_1(S^y_i)$,  $v_1(S^z_i)$,  $1 \leq i \leq n$, 
and $v_1(S^{\prime}_i)$,  $1\le i \le n^{\prime}$, 
and with all vertices $u_1$, $u_2$, \ldots, $u_{\widehat{n}}$.
Thus we have 
\[ V = \Big [\bigcup_{i = 1}^n [V(S^x_i) \cup V(S^y_i) \cup V(S^z_i)]\Big ] \bigcup  \Big [\bigcup_{i=1}^{n^{\prime}}V(S^{\prime}_i)\Big ] \bigcup
 \{u_1, u_2, \ldots, u_{\widehat{n}}, v_1\} \mbox{ and }\]

\[  E = E1\cup E_2 \cup E_3\, ,  \mbox{ where } E_1=\Big\{(v_1, u_i)\colon  1\le i\le \widehat{n}\Big\}\, ,\] 
\[E_2=\big [\bigcup_{i=1}^{n^{\prime}}E(S^{\prime}_i)\big ]\bigcup \Big\{\big (v_1, v_1(S^{\prime}_i)\big )\colon  1\le i\le n^{\prime}\Big\} \mbox{ and }\]
\[  E_3= \bigcup_{i = 1}^n  \Big [E(S^x_i) \cup E(S^y_i) \cup E(S^z_i) \cup 
\big\{(v_1, v_1(S^x_i)), (v_1, v_1(S^y_i)), (v_1, v_1(S^z_i))\big\}\Big ]  \, .\]
The number of vertices in $G$ is given as $|V(G)|=\sum_{i=1}^n (|V(S^x_i)|+|V(S^y_i)|+|V(S^z_i)|  + \widehat{n}+n^{\prime}d^l+1=  d^L$. 
Thus the host graph is a $d$-regular tree $T$ of height $h = \\ \lceil\log_d{|V(G)|}\rceil = L$. 
\medskip

We show that the  optimal   value $OPT(G,d)$ of the objective function of $DAPT(G,d)$ equals the 
expression in 
(\ref{obj.val:YESinstance}) 
if and only if the corresponding NMTS instance is a YES-instance,  
and this would  complete the $\mathcal{NP}$-hardness proof. 

Prior to showing the above if and only if statement, consider  the optimal arrangement of the substar $G'$ induced  in $G$ 
by $v_1$ 
and its neighbours, i.e.\ by   the following  set of vertices 
$$\{v_1\}\cup  \{u_i\colon 1\le i \le \widehat{n}\}\cup \{v_1(S^{\prime}_i)\colon 1\le i\le n^{\prime}\} \cup \big [\cup_{i=1}^n 
\{v_1(S^{x}_i), v_1(S^{y}_i), v_1(S^{z}_i)\} \big ].$$
Assume w.l.o.g.\ that  the vertex $v_1$ is arranged at the most left leaf of  $T$ 
(i.e.\ the first leaf $b_1$ of $T$ in  the canonical ordering). 
The vertex $v_1$ has $3 n + \widehat{n} + n^{\prime}=  d^{L-1}-1+3n+n^{\prime}>d^{L-1}-1$ 
neighbours (note that $n^{\prime}\ge 0$) and hence $|V(G')|= 3 n + \widehat{n} + n^{\prime}+1=d^{L-1}+3n+n^{\prime}$.
According to Lemma~\ref{lemma:gIsAStar},  
any  arrangement of $G'$   which  arranges the   $d^{L-1}-1$ neighbours   $\{u_1, u_2, \ldots, u_{\widehat{n}}\}$ of $v_1$  at the 
leaves of the leftmost basic subtree of $T$ and the remaining $3n + n^{\prime}$  neighbours at some other (arbitrarily selected) 
leaves of $T$  is optimal.  
In particular such an arrangement would not arrange $v_1(S^x_i)$,  $v_1(S^y_i)$,  $v_1(S^z_i)$, $i=1,2,\ldots,n$, and $v_1(S^{\prime}_i)$, 
$1\le i\le  n^{\prime}$,
at the leaves of the leftmost basic subtree. 
\medskip

Let us now proof the if and only if statement formulated above. 
\begin{description}
\item{The ``if'' statement.} 

Assume that the NMTS instance is  a YES-instance. We show that the equality (\ref{obj.val:YESinstance}) holds. 
 Consider  two permutations $(j_1, j_2, \ldots, j_n)$ and $(k_1, k_2, \ldots, k_n)$ of the indices 
$\{1,2, \ldots, n\}$, such that $z_i = x_{j_i} + y_{k_i}$, for all $i = 1, 2, \ldots, n$. 
Then for all $i\in \{1,2,\ldots,n\}$ we have 
\begin{eqnarray}\nonumber 
|V(S^x_{j_i})|+|V(S^y_{k_i})|+|V(S^z_{i})| &=& (d-1)d^ {l-2}+x_{j_i}+(d-1)d^{l-4}+y_{k_i}+d^l-\\
\label{sumofstars:equ}
(d-1)d^{l-4}-(d-1)d^{l-2}-z_i&=&d^l.
\end{eqnarray} 
Thus, for every $i\in \{1,2,\ldots,n\}$  the disjoint stars graphs $S^x_{j_i}$, $S^y_{k_i}$ and $S^z_i$  can be optimally arranged at 
the leaves of the $i$-th rightmost 
 subtree of  $l$-th order according to Lemma~\ref{lemma:gConsistsOfThreeStars}.
Notice that the assumptions of  the lemma  are fulfilled  because 
 $|V(S^z_i)|  > (d-1)d^{l-1}\ge d^{l-1}$. 
(Indeed,  due to $y_{k_i}\le d^{l-4}$ and 
$x_{j_i}+y_{k_i}+(d-1)d^{y-4}< d^{l-2}$ we get $|V(S^y_{k_i})|+|V(S^x_{j_i})|=(d-1)d^{l-4}+y_{k_i}+(d-1)d^{l-2}+x_{j_i}< d^{l-1}$ 
which together with 
(\ref{sumofstars:equ}) implies $|V(S^z_i)|> (d-1)d^{l-1}\ge d^{l-1}$.) 
Since $nd^l\le d^{L-1}$ the $n$ rightmost 
 subtrees of the $l$-th order are subtrees of the rightmost basic subtree (of height $L-1$). 

It remains to arrange the vertices $v_1,u_1, u_2, \ldots, u_{\widehat{n}}$ and the stars $S^{\prime}_i$, $1\le i\le n^{\prime}$. 
This is done by arranging $v_1, u_1, u_2, \ldots, u_{\widehat{n}}$  at the leaves of the 
leftmost basic subtree and the stars $S^{\prime}_i$, $1\le i\le n^{\prime}$ in one of the $n^{\prime}$ still free  subtrees 
of $l$-th order each. These arrangements are cleary optimal for each of the stars  $S^{\prime}_i$, $1\le i\le n^{\prime}$ 
(recall that they have $d^l$ vertices each). 
Moreoever,  as mentioned above this is also an optimal arrangement  of $G'$. 
Thus this arrangement arranges optimally the  following subgraphs of G: $G'$, $S^{\prime}_i$, 
$1\le i\le n$, 
and the disjoint unions of the triples  $(S^{x}_{j_i},S^{y}_{k_i},S^z_i)$, for $1\le i\le n$, respectively.
  Since the edge sets of the above mentioned graphs yield a partition of the edge set $E(G)$, 
the  arrangement described above is an optimal arrangemnt of $G$.  
The corresponding value $OPT(G,d)$ of the objective function
  is given as  the sum of $OPT(G',d)$, $OPT(S^{\prime}_i,d)$, $1\le i\le n'$, and 
$OPT(S^x_{j_i}\cup S^y_{k_i} \cup S^z_i, d)$, for $1\le i\le n$.
According to  Lemma~\ref{lemma:gIsAStar} and Lemma~\ref{lemma:gConsistsOfThreeStars} we get
\[
OPT(G,d)= 2\Big (Ln^{\prime\prime}-\frac{d^L-1}{d-1}\Big )+2n^{\prime}\Big (ld^l-\frac{d^l-1}{d-1}\Big )+\]
\begin{equation}\label{obj.val:YESinstance}
2\sum_{i=1}^n \Big [l|V(S^z_i)|-\frac{d^{l-1}}{d-1}+(l-1)|V(S^x_{j_i})|-\frac{d^{l-1}-1}{d-1}+(l-3)|V(S^y_{k_i})|-\frac{d^{l-3}-1}{d-1}\Big ]
\, ,
\end{equation}
where $n^{\prime\prime} \defeq |V(G')|=\widehat{n}+n^{\prime}+d^{L-1}$.

\item{The ``only if'' statement.} 
Assume that  $OPT(G,d)$ is given as in (\ref{obj.val:YESinstance}) and consider some optimal arrangement $\phi$ of $G$.  
We show that the corresponding NMTS instance is a YES-instance.
Notice that (\ref{obj.val:YESinstance}) implies 
\begin{eqnarray*} OPT(G,d)&=&OPT(G',d)+\sum_ {i=1}^{n^{\prime}} OPT(S^{\prime}_i,d)+\\
 && \sum_{i=1}^n \big (OPT(S^x_i,d)+OPT(S^y_i,d)+OPT(S^z_i,d\big )\, . \end{eqnarray*}
Thus,  in particular, $\phi$ yields  optimal arrangements of $G'$,   $S^{\prime}_i$, 
$1\le i \le n^{\prime}$, and  $S^x_i$, $S^y_i$ and $S^z_i$, $1\le i\le n$.
Assume w.l.o.g.\ that $\phi$ arranges $v_1$ at the leftmost leaf of $T$. 
Then according to Lemma~\ref{lemma:gIsAStar} $\phi$ arranges neighbours of $v_1$, i.e.\ vertices of $G'$, to each leaf
 of  the leftmost basic subtree of $T$. Thus, none of the neighbours of $v_1$  arranged at 
the leaves of the       leftmost basic subtree of $T$ can be the central vertex of some of the stars $S^{\prime}_i$,  
$1\le i \le n^{\prime}$, or $S^x_i$, $S^y_i$ and $S^z_i$, $1\le i\le n$, because then according to Lemma~\ref{lemma:gIsAStar} 
$\phi$  would not lead to an optimal arrangement of the corresponding star. 
It follows that the vertices arranged at the leaves of the leftmost basic subtree of $T$ are $u_i$, $1\le i\le \widehat{n}$. 

\noindent According to Lemma~\ref{lemma:gIsAStar} $\phi$  arranges each star $S^{\prime}_i$, $1\le i \le n^{\prime}$, 
to the leaves of some subtree of $l$-th order. Hence there remain $(d-1)d^{L-1-l}-n^{\prime}=n$ subtrees of $l$-th order 
at the  leaves of which $\phi$ the stars $S^{x_i}$, $S^{y_i}$ and $S^{z_i}$, $1\le i\le n$.

\noindent  Recall now that $(d-1)d^{l-1}< |V(S^z_i)|$, $(d-1)d^{l-2}<|V(S^x_i)|<d^{l-1}$ and $(d-1)d^{l-4}<|V(S^y_i)|<d^{l-3}$ for $1\le i \le n$. Thus in each of the $n$ remaining free subtrees of $l$-th order  can not be arranged more than 
one of the stars $S^z_i$, $1\le i\le n$, and since there are $n$ such stars to be arranged in $n$ subtrees of $l$-th 
order exactly one them will be arranged in each subtree.  By analogous arguments we get that exactly one of the 
stars $S^x_i$ will be arranged in each of the subtrees of $l$-th order mentioned above, and finally,  exactly one of the 
stars $S^y_i$ will be arranged in each of these subtrees. Thus  in each of the $n$-th subtrees of $l$-th level exactly one 
star $S^z_i$, one star   $S^x_i$ and one star  $S^y_i$ will be arranged. For all $i \in \{1,2,\ldots,n\}$ denote by $j_i$ and 
$k_i$ the indices of the stars arranged together with $S^z_i$ in the same subtree, i.e.\ $S^x_{j_i}$, $S^y_{k_j}$ and $S^z_i$ 
are arranged in the same subtree of $l$-th order, $1\le i\le n$.  Clearly, $(j_1,j_2,\ldots,j_n)$ and $(k_1,k_2,\ldots,k_n)$
are permutations of $\{1,2,\ldots,n\}$, respectively. Since the stars $S^x_i$, $S^y_i$ and $S^z_i$, $1\le i \le n$,
 have $nd^l$ vertices alltogether, which is  the number of leaves of  the $n$ subtrees of $l$-th level, the equality 
\begin{equation} \label{noinstance:equ}
|V(S^x_{j_i})| + |V(S^y_{k_i})|+|V(S^z_i)|=d^l 
\end{equation}
must hold, for all $1\le i\le n$.  
By substituting the cardinalities of the vertex sets of the stars in (\ref{noinstance:equ}) we get $x_{j_i}+y_{k_i}-z_i=0$,
 for all $1\le i\le n$, and this completes the proof.
 
\end{description}
\end{proof}
\end{theorem}
\section{Final notes, conclusions and outlook}
	\label{section:finalNotesConclusionsAndOutlook}
	This paper deals with a special case of the data arrangement problem on regular trees (DAPT), namely 
 with the case where both the guest and the host graph are binary regular trees of  heights $h_G$ and $h = h_G + 1$, 
respectively. 
Some basic properties of the problem are identified  and an approximation algorithm with approximation ratio $1.015$ 
is proposed. 
Moreover, we provide closed formulas for the arrangement generated by the approximation algorithm  mentioned above 
and for its corresponding objective  function value. 

The analysis of the approximation algorithm and the estimation of the   approximation ratio involve 
 a special case of  the {\em $k$-balanced partitioning problem} 
({\em $k$-BPP})  as an auxiliary problem. 
More precisely we consider the {\em $k$-BPP} where the input graph is  a binary regular tree 
for a particular choice of $k$, namely for $k$ being a power of $2$,  and give (a lower bound for 
the) optimal value its objective function.

Further  we investigate the relationship between the considered particular cases of the DAPT and the $k$-BPP
 and   derive a lower bound for  the value of the objective function of the $DAPT(G,2)$ by solving $h_G$ instances 
of the {\em $2^{k^{\prime}}$-BPP}, where $h_G$ is the height of the guest graph $G$ and  $k^{\prime}=1,2\ldots,h_G$.
This lower bound leads then to the above mentioned approximation ratio. 

It would be  interesting to investigate whether some alternative analysis of the proposed algorithm for the $DAPT(G,2)$
 could lead to a 
better approximation ratio.
Numerical results provide some evidence that the answer to this question might be ``yes'' and suggest an empirical 
approximation  ratio of $1.0098$.

 Finally we settle an open question from the literature and prove that the DAPT with a host graph being a $d$-regular tree,
 $d\ge 2$, $d\in \nz$, is 
${\cal NP}$-hard even if the guest graph is a tree. 

The complexity of  the DAPT in the case where both the guest and the host
graph are binary regular trees remains an open question for further research. 
Other open questions  concerns the more general cases of the DAPT where  the guest 
graph $G$ is a $d$-regular tree and the host graph $T$ is a $d'$-regular tree 
with $d=d'\ge 3$ or $d\neq d'$.

\section*{Acknowledgements}
	\label{section:acknowledgements}
	The research was funded by the Austrian Science Fund (FWF): P23829.
	
\medskip

\bibliographystyle{abbrv}
\bibliography{TheDataArrangementProblemOnBinaryTrees}

\newpage

\appendix

\section{Appendix}
	\label{section:appednix}
	
	\subsection{\boldmath Arrangements $\phi_A$ obtained from Algorithm~\ref{algorithm:arrangementPhiU} for different heights $h_G$ of the guest graph $G$\unboldmath}
		\label{subsection:section:appednix:arrangementsPhiUObtainedFromAlgorithmForDifferentHeightsHG}

		\begin{multicols}{2}
			\vspace*{-1.0cm}
			\begin{figurehere}
				\centering
				\begin{comment:appendix}
					\begin{tikzpictureGuestGraph}
						\node[circle, draw=white!50, fill=white!100, thick, inner sep=0pt, minimum size=3mm] (node1w) at (0.25, 1.00) {};
			
						\node[circle, draw=white!100, fill=white!50, thick, inner sep=0pt, minimum size=5mm] (node2w) at (0.00, 0.00) {\color{white} 1};
						\node[circle, draw=white!100, fill=white!100, thick, inner sep=0pt, minimum size=5mm] (node3w) at (0.50, 0.00) {};
			
						\draw [-, white!50] (node1w) to (node2w);
						\draw [-, white!50] (node1w) to (node3w);

						\node[nodeBlack] (node1) at (0.00 + 0.25, 0.00) {1};
					\end{tikzpictureGuestGraph}
				\end{comment:appendix}
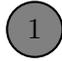
				\caption{Guest graph $G = (V, E)$ (binary regular tree of height $h_G = 0$).}
				\label{figure:guestGraphGForHG0}
			\end{figurehere}
			
			$ $
			
			\begin{figurehere}
				\centering
				\begin{comment:appendix}
					\begin{tikzpictureHostGraph}
						\node[otherNode] (node1) at (0.25, 1.00) {};
			
						\node[leafUsedBlack] (node2) at (0.00, 0.00) {1};
						\node[leafUnused] (node3) at (0.50, 0.00) {};
			
						\draw [-, black!50] (node1) to (node2);
						\draw [-, black!50] (node1) to (node3);
					\end{tikzpictureHostGraph}
				\end{comment:appendix}
				\caption{Arrangement $\phi_A$ obtained from Algorithm~\ref{algorithm:arrangementPhiU} for the guest graph 
	of height $h_G=1$ depicted in Figure~\ref{figure:guestGraphGForHG0}. Its objective function value is $OV(G, 2, \phi_A) = 0$.}
				\label{figure:arrangementPhiUForHG0}
			\end{figurehere}
		\end{multicols}

		\begin{multicols}{2}
			\vspace*{-0.95cm}
			\begin{figurehere}
				\centering
				\begin{comment:appendix}
					\begin{tikzpictureGuestGraph}
						\node[circle, draw=white!50, fill=white!100, thick, inner sep=0pt, minimum size=3mm] (node1w) at (0.75, 2.00) {};
			
						\node[circle, draw=white!50, fill=white!100, thick, inner sep=0pt, minimum size=3mm] (node2w) at (0.25, 1.00) {};
						\node[circle, draw=white!50, fill=white!100, thick, inner sep=0pt, minimum size=3mm] (node3w) at (1.25, 1.00) {};
			
						\node[circle, draw=white!100, fill=white!50, thick, inner sep=0pt, minimum size=5mm] (node4w) at (0.00, 0.00) {\color{white} 2};
						\node[circle, draw=white!100, fill=white!50, thick, inner sep=0pt, minimum size=5mm] (node5w) at (0.50, 0.00) {\color{white} 1};
						\node[circle, draw=white!100, fill=white!50, thick, inner sep=0pt, minimum size=5mm] (node6w) at (1.00, 0.00) {\color{white} 3};
						\node[circle, draw=white!100, fill=white!100, thick, inner sep=0pt, minimum size=5mm] (node7w) at (1.50, 0.00) {};
			
						\draw [-, white!50] (node1w) to (node2w);
						\draw [-, white!50] (node1w) to (node3w);
			
						\draw [-, white!50] (node2w) to (node4w);
						\draw [-, white!50] (node2w) to (node5w);
						\draw [-, white!50] (node3w) to (node6w);
						\draw [-, white!50] (node3w) to (node7w);
						
						\draw [-, white] (node5w) to [out=-90,in=-90] (node4w);
						\draw [-, white] (node5w) to [out=-90,in=-90] (node6w);

						\node[nodeBlack] (node1) at (0.25 + 0.50, 1.00) {1};
			
						\node[nodeBlack] (node2) at (0.00 + 0.50, 0.00) {2};
						\node[nodeRed] (node3) at (0.50 + 0.50, 0.00) {3};
						
						\draw (node1) to (node2);
						\draw (node1) to (node3);
					\end{tikzpictureGuestGraph}
				\end{comment:appendix}
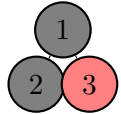
				\caption{Guest graph $G = (V, E)$ (binary regular tree of height $h_G = 1$).}
				\label{figure:guestGraphGForHG1}
			\end{figurehere}

			$ $
			
			\begin{figurehere}
				\centering
				\begin{comment:appendix}
					\begin{tikzpictureHostGraph}
						\node[otherNode] (node1) at (0.75, 2.00) {};
			
						\node[otherNode] (node2) at (0.25, 1.00) {};
						\node[otherNode] (node3) at (1.25, 1.00) {};
			
						\node[leafUsedBlack] (node4) at (0.00, 0.00) {2};
						\node[leafUsedBlack] (node5) at (0.50, 0.00) {1};
						\node[leafUsedRed] (node6) at (1.00, 0.00) {3};
						\node[leafUnused] (node7) at (1.50, 0.00) {};
			
						\draw [-, black!50] (node1) to (node2);
						\draw [-, black!50] (node1) to (node3);
			
						\draw [-, black!50] (node2) to (node4);
						\draw [-, black!50] (node2) to (node5);
						\draw [-, black!50] (node3) to (node6);
						\draw [-, black!50] (node3) to (node7);
						
						\draw [-] (node5) to [out=-90,in=-90] (node4);
						\draw [-] (node5) to [out=-90,in=-90] (node6);
					\end{tikzpictureHostGraph}
				\end{comment:appendix}
				\caption{Arrangement $\phi_A$ obtained from Algorithm~\ref{algorithm:arrangementPhiU} for the guest graph 
	of height $h_G=1$ depicted in Figure~\ref{figure:guestGraphGForHG1}. Its objective function value is $OV(G, 2, \phi_A) = 6$.}
				\label{figure:arrangementPhiUForHG1}
			\end{figurehere}
		\end{multicols}

		\begin{multicols}{2}
			\vspace*{-1.02cm}
			\begin{figurehere}
				\centering
				\begin{comment:appendix}
					\begin{tikzpictureGuestGraph}
						\node[circle, draw=white!50, fill=white!100, thick, inner sep=0pt, minimum size=3mm] (node1w) at (1.75, 3.00) {};
			
						\node[circle, draw=white!50, fill=white!100, thick, inner sep=0pt, minimum size=3mm] (node2w) at (0.75, 2.00) {};
						\node[circle, draw=white!50, fill=white!100, thick, inner sep=0pt, minimum size=3mm] (node3w) at (2.75, 2.00) {};
			
						\node[circle, draw=white!50, fill=white!100, thick, inner sep=0pt, minimum size=3mm] (node4w) at (0.25, 1.00) {};
						\node[circle, draw=white!50, fill=white!100, thick, inner sep=0pt, minimum size=3mm] (node5w) at (1.25, 1.00) {};
						\node[circle, draw=white!50, fill=white!100, thick, inner sep=0pt, minimum size=3mm] (node6w) at (2.25, 1.00) {};
						\node[circle, draw=white!50, fill=white!100, thick, inner sep=0pt, minimum size=3mm] (node7w) at (3.25, 1.00) {};
			
						\node[circle, draw=white!100, fill=white!50, thick, inner sep=0pt, minimum size=5mm] (node8w) at (0.00, 0.00) {\color{white} 4};
						\node[circle, draw=white!100, fill=white!50, thick, inner sep=0pt, minimum size=5mm] (node9w) at (0.50, 0.00) {\color{white} 2};
						\node[circle, draw=white!100, fill=white!50, thick, inner sep=0pt, minimum size=5mm] (node10w) at (1.00, 0.00) {\color{white} 5};
						\node[circle, draw=white!100, fill=white!50, thick, inner sep=0pt, minimum size=5mm] (node11w) at (1.50, 0.00) {\color{white} 1};
						\node[circle, draw=white!100, fill=white!50, thick, inner sep=0pt, minimum size=5mm] (node12w) at (2.00, 0.00) {\color{white} 6};
						\node[circle, draw=white!100, fill=white!50, thick, inner sep=0pt, minimum size=5mm] (node13w) at (2.50, 0.00) {\color{white} 3};
						\node[circle, draw=white!100, fill=white!50, thick, inner sep=0pt, minimum size=5mm] (node14w) at (3.00, 0.00) {\color{white} 7};
						\node[circle, draw=white!100, fill=white!100, thick, inner sep=0pt, minimum size=5mm] (node15w) at (3.50, 0.00) {};
			
						\draw [-, white!50] (node1w) to (node2w);
						\draw [-, white!50] (node1w) to (node3w);
			
						\draw [-, white!50] (node2w) to (node4w);
						\draw [-, white!50] (node2w) to (node5w);
						\draw [-, white!50] (node3w) to (node6w);
						\draw [-, white!50] (node3w) to (node7w);
			
						\draw [-, white!50] (node4w) to (node8w);
						\draw [-, white!50] (node4w) to (node9w);
						\draw [-, white!50] (node5w) to (node10w);
						\draw [-, white!50] (node5w) to (node11w);
						\draw [-, white!50] (node6w) to (node12w);
						\draw [-, white!50] (node6w) to (node13w);
						\draw [-, white!50] (node7w) to (node14w);
						\draw [-, white!50] (node7w) to (node15w);
						
						\draw [-, white] (node11w) to [out=-90,in=-90] (node9w);
						\draw [-, white] (node11w) to [out=-90,in=-90] (node13w);
			
						\draw [-, white] (node9w) to [out=-90,in=-90] (node8w);
						\draw [-, white] (node9w) to [out=-90,in=-90] (node10w);
						\draw [-, white] (node13w) to [out=-90,in=-90] (node12w);
						\draw [-, white] (node13w) to [out=-90,in=-90] (node14w);
					
						\node[nodeBlackDashed] (node1) at (0.75 + 1.00, 2.00) {1};
			
						\node[nodeBlack] (node2) at (0.25 + 1.00, 1.00) {2};
						\node[nodeRed] (node3) at (1.25 + 1.00, 1.00) {3};
			
						\node[nodeBlack] (node4) at (0.00 + 1.00, 0.00) {4};
						\node[nodeBlackDashed] (node5) at (0.50 + 1.00, 0.00) {5};
						\node[nodeRed] (node6) at (1.00 + 1.00, 0.00) {6};
						\node[nodeRedDashed] (node7) at (1.50 + 1.00, 0.00) {7};
						
						\draw (node1) to (node2);
						\draw (node1) to (node3);
			
						\draw (node2) to (node4);
						\draw (node2) to (node5);
						\draw (node3) to (node6);
						\draw (node3) to (node7);
					\end{tikzpictureGuestGraph}
				\end{comment:appendix}
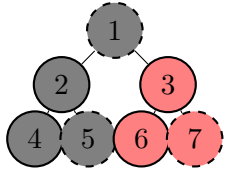
				\caption{Guest graph $G = (V, E)$ (binary regular tree of height $h_G = 2$).}
				\label{figure:guestGraphGForHG2}
			\end{figurehere}
			
			$ $

			\begin{figurehere}
				\centering
				\begin{comment:appendix}
					\begin{tikzpictureHostGraph}
						\node[otherNode] (node1) at (1.75, 3.00) {};
			
						\node[otherNode] (node2) at (0.75, 2.00) {};
						\node[otherNode] (node3) at (2.75, 2.00) {};
			
						\node[otherNode] (node4) at (0.25, 1.00) {};
						\node[otherNode] (node5) at (1.25, 1.00) {};
						\node[otherNode] (node6) at (2.25, 1.00) {};
						\node[otherNode] (node7) at (3.25, 1.00) {};
			
						\node[leafUsedBlack] (node8) at (0.00, 0.00) {4};
						\node[leafUsedBlack] (node9) at (0.50, 0.00) {2};
						\node[leafUsedBlackDashed] (node10) at (1.00, 0.00) {5};
						\node[leafUsedBlackDashed] (node11) at (1.50, 0.00) {1};
						\node[leafUsedRed] (node12) at (2.00, 0.00) {6};
						\node[leafUsedRed] (node13) at (2.50, 0.00) {3};
						\node[leafUsedRedDashed] (node14) at (3.00, 0.00) {7};
						\node[leafUnused] (node15) at (3.50, 0.00) {};
			
						\draw [-, black!50] (node1) to (node2);
						\draw [-, black!50] (node1) to (node3);
			
						\draw [-, black!50] (node2) to (node4);
						\draw [-, black!50] (node2) to (node5);
						\draw [-, black!50] (node3) to (node6);
						\draw [-, black!50] (node3) to (node7);
			
						\draw [-, black!50] (node4) to (node8);
						\draw [-, black!50] (node4) to (node9);
						\draw [-, black!50] (node5) to (node10);
						\draw [-, black!50] (node5) to (node11);
						\draw [-, black!50] (node6) to (node12);
						\draw [-, black!50] (node6) to (node13);
						\draw [-, black!50] (node7) to (node14);
						\draw [-, black!50] (node7) to (node15);
						
						\draw [-] (node11) to [out=-90,in=-90] (node9);
						\draw [-] (node11) to [out=-90,in=-90] (node13);
			
						\draw [-] (node9) to [out=-90,in=-90] (node8);
						\draw [-] (node9) to [out=-90,in=-90] (node10);
						\draw [-] (node13) to [out=-90,in=-90] (node12);
						\draw [-] (node13) to [out=-90,in=-90] (node14);
					\end{tikzpictureHostGraph}
				\end{comment:appendix}
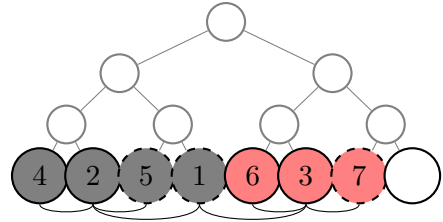
				\caption{Arrangement $\phi_A$ obtained from Algorithm~\ref{algorithm:arrangementPhiU} for the guest graph 
	of height $h_G=2$ depicted in Figure~\ref{figure:guestGraphGForHG2}. Its objective function value is $OV(G, 2, \phi_A) = 22$.}
				\label{figure:arrangementPhiUForHG2}
			\end{figurehere}
		\end{multicols}
\end{document}